\documentclass{amsart}

\usepackage{pgf,tikz,pgfplots}
\pgfplotsset{compat=1.15}
\usepackage{mathrsfs}
\usetikzlibrary{arrows}

\usepackage{graphicx}

\usepackage[T1]{fontenc}
\usepackage{microtype}
\usepackage{lmodern}
\usepackage[colorlinks=true,urlcolor=blue, citecolor=red,linkcolor=blue,linktocpage,pdfpagelabels, bookmarksnumbered,bookmarksopen]{hyperref}
\usepackage[hyperpageref]{backref}
\usepackage{amsthm} 
\usepackage{latexsym,amsmath,amssymb}

\usepackage{accents}
\usepackage{esint}

\usepackage{soul}
\usepackage{mathtools} 
\usepackage{xparse} 
\usepackage[capitalize]{cleveref}

\usepackage[shortlabels]{enumitem}

\title[Minimizing curves for the nonlocal Willmore energy]{Convex minimizing curves of the scaling-invariant nonlocal Willmore energy}
\author[G. Giacomin]{Giovanni Giacomin}
\address[Giovanni Giacomin]
{Department of Mathematics and Statistics, University of Western Australia, 35 Stirling Highway, WA6009 Crawley,
Australia.} \email{giovanni.giacomin@research.uwa.edu.au}

\author[A. Schikorra]{Armin Schikorra}
\address[Armin Schikorra]{Department of Mathematics,
University of Pittsburgh,
301 Thackeray Hall,
Pittsburgh, PA 15260, USA}
\email{armin@pitt.edu}

\newcommand{\N}{{\mathbb N}}

\renewcommand{\S}{{\mathbb S}}

\newtheorem{theorem}{Theorem}
\newtheorem{lemma}[theorem]{Lemma}
\newtheorem{corollary}[theorem]{Corollary}
\newtheorem{proposition}[theorem]{Proposition}

\theoremstyle{definition}
\newtheorem{definition}[theorem]{Definition}

\theoremstyle{remark}
\newtheorem{remark}[theorem]{Remark}

\newcommand\diam{{\rm diam\,}}
\newcommand\dist{{\rm dist\,}}

\newcommand\lip{{\rm Lip\,}}

\newcommand\supp{{\rm supp\,}}

\newcommand{\R}{\mathbb{R}}

\newcommand{\brac}[1]{\left (#1 \right )}
\newcommand{\abs}[1]{\left\lvert #1 \right \rvert}

\newcommand{\barint}{
\rule[.036in]{.12in}{.009in}\kern-.16in \displaystyle\int }

\newcommand{\barcal}{\text{$ \rule[.036in]{.11in}{.007in}\kern-.128in\int $}}

\def\mvint_#1{\mathchoice
          {\mathop{\vrule width 6pt height 3 pt depth -2.5pt
                  \kern -8pt \intop}\nolimits_{\kern -3pt #1}}%
          {\mathop{\vrule width 5pt height 3 pt depth -2.6pt
                  \kern -6pt \intop}\nolimits_{#1}}%
          {\mathop{\vrule width 5pt height 3 pt depth -2.6pt
                  \kern -6pt \intop}\nolimits_{#1}}%
          {\mathop{\vrule width 5pt height 3 pt depth -2.6pt
                  \kern -6pt \intop}\nolimits_{#1}}}

\numberwithin{theorem}{section} \numberwithin{equation}{section}

\newcommand{\aleq}{\lesssim}
\newcommand{\ageq}{\succsim}
\newcommand{\aeq}{\approx}

\newcommand{\laps}[1]{|D|^{#1}}

\usepackage{scalerel}[2014/03/10]
\usepackage[usestackEOL]{stackengine}
\def\avint{\,\ThisStyle{\ensurestackMath{%
			\stackinset{c}{.2\LMpt}{c}{.5\LMpt}{\SavedStyle-}{\SavedStyle\phantom{\int}}}%
		\setbox0=\hbox{$\SavedStyle\int\,$}\kern-\wd0}\int}

\let\latexchi\chi
\makeatletter
\renewcommand\chi{\@ifnextchar_\sub@chi\latexchi}
\newcommand{\sub@chi}[2]{
  \@ifnextchar^{\subsup@chi{#2}}{\latexchi^{}_{#2}}%
}
\newcommand{\subsup@chi}[3]{
  \latexchi_{#1}^{#3}%
}
\makeatother
\newcommand{\eps}{\varepsilon}

\begin{document}

\begin{abstract}
We consider the scaling-invariant nonlocal Willmore energy, defined via the nonlocal mean curvature by Caffarelli, Roquejoffre and Savin.
Our main result is the existence of minimizers in the class of convex $C^1$-curves.
\end{abstract}

\maketitle
\tableofcontents

\section{Introduction and main results}

Given a subset $E\subset\R^2$ the nonlocal mean curvature  at $x\in \partial E$ introduced by Caffarelli, Roquejoffre and Savin in~\cite{caffarelli2009nonlocal} is defined as
\begin{equation*}
H_{\partial E}^s(x):=\mbox{P.V.}\int_{\R^2}\frac{\chi_{E^c}(y)-\chi_{E}(y)}{\left|x-y\right|^{2+s}}\,dy \equiv \lim_{\eps \to 0} \int_{\R^2 \setminus B_\eps(x)}\frac{\chi_{E^c}(y)-\chi_{E}(y)}{\left|x-y\right|^{2+s}}\,dy,
\end{equation*}
whenever that limit exists, e.g. in the case of smooth $\partial E$. The fractional mean curvature is obtained as the first variation of the fractional perimeter, just like the classical mean curvature is the first variation of the usual perimeter (or area) functional.

Having a nonlocal mean curvature, we can naturally define a nonlocal Willmore energy:
For $s\in (0,1)$ and $p\geq 1$ we set
\begin{equation*}
\hat{\mathscr{W}}_{s,p}(\partial E):=\int_{\partial E}\left|H_{\partial E}^s(x)\right|^p\,d\sigma^1(x).
\end{equation*} 
This energy is well-defined on measurable sets $E$ for which $H_{\partial E}^s(x)$ exists for a.e. $x \in \partial E$ -- for all other measurable sets we set $\hat{\mathscr{W}}_{s,p}(\partial E) := +\infty$. Also, observe that it does not matter if we consider $E$ or $E^c$, the nonlocal Willmore energy stays the same.

If we parametrize a suitably nice set $\partial E$ by a curve $\gamma: \S^1 \to \R^2$, we see that 
\[
\begin{split}
H_{\partial E}^s(x) =& \pm c_s \int_{\S^1} \frac{\langle n(\gamma(y)), \gamma(y)-\gamma(x)\rangle}{\left|\gamma(y)-\gamma(x)\right|^{2+s}}\left|\gamma'(y)\right|\,dy\\
\equiv &\pm c_s \int_{\S^1} \frac{\langle n(\gamma(y)),\gamma(y)-\gamma(x)-\gamma'(y)(y-x)\rangle}{\left|\gamma(y)-\gamma(x)\right|^{2+s}}\left|\gamma'(y)\right|\,dy.
\end{split}
\]
Here, for $z \in \partial E$ the vector $n(z) \in \S^1$ denotes the unit normal, and the sign of $H_{\partial E}^s$ depends on whether it is the outwards facing or inwards facing normal (or, in other words, if we consider $E$ or $E^c$). Observe that if $\gamma$ is smooth, then the last integral in the equality above is absolutely convergent.

This leads to a nonlocal Willmore energy for curves  $\gamma: \S^1 \to \R^2$,
\begin{equation*}
\mathscr{W}_{s,p}(\gamma)=\int_{\S^1}\abs{\int_{\S^1} \frac{\langle n(\gamma(y)),\gamma(y)-\gamma(x)\rangle }{\left|\gamma(y)-\gamma(x)\right|^{2+s}}\left|\gamma'(y)\right|\,dy}^{p}\left|\gamma'(x)\right|\,dx,
\end{equation*}
and in the smooth setting we have
\begin{equation*}
\mathscr{W}_{s,p}(\gamma)=c_{s,p} \hat{\mathscr{W}}_{s,p}(E).
\end{equation*}
In collaboration with Blatt and Scheuer we discussed some aspects of the \emph{subcritical} nonlocal Willmore energy in \cite{blatt2023fractional}, i.e. when $s > \frac{1}{p}$. In this work we are interested in the critical case, which corresponds to $p = \frac{1}{s}$, when the nonlocal Willmore energy becomes scaling invariant. Scaling invariance means: If for some $\gamma: \S^1 \to \R^2$ and $\rho>0$ we consider the rescaled curve $\gamma_{\rho}(\cdot):=\rho\gamma(\frac{\cdot}{\rho}):\rho \S^1\to \R^2$, then we have 
\begin{equation*}
\mathscr{W}_{s,\frac{1}{s}}(\gamma)=\mathscr{W}_{s,\frac{1}{s}}(\gamma_\rho).
\end{equation*}  

For simplicity of notation we shall denote from now on
\begin{equation*}
\mathscr{W}_{s}(\gamma)=\int_{\S^1}\left(\int_{\S^1}\frac{\langle n(\gamma(y)),\gamma(y)-\gamma(x)\rangle }{\left|\gamma(y)-\gamma(x)\right|^{2+s}}\left|\gamma'(y)\right|\,dy\right)^{\frac{1}{s}}\left|\gamma'(x)\right|\,dx,
\end{equation*}

It is obvious that nonlocal minimal surfaces, i.e. surfaces with vanishing fractional mean curvature $H_{s} \equiv 0$, are absolute minimizers of the nonlocal Willmore energy -- but in general a curve or surface which is critical for the Willmore-energy is not necessarily a minimal surface (in the sense of vanishing mean curvature). For example there are no compact nonlocal minimal surfaces \cite{CFMN18,CFSW18}. While nonlocal minimal and constant mean curvature surfaces have been studied very thoroughly, for an overview see e.g. \cite{DV16,DV22}, the theory of nonlocal Willmore surfaces is only at the very beginning.
Similar to recent results on nonlocal minimal surfaces \cite{CSSV23,CFS23} one could hope that eventually one can transfer properties for $s$-Willmore surfaces to classical Willmore surfaces by letting $s \to 1$.

Regarding the existence of minimizers, scaling invariant variational functionals are prone to bubbling effects, and the existence of minimizers becomes a nontrivial question. The classical example is the Sacks-Uhlenbeck theory for harmonic maps in homotopy groups, \cite{SU81}. But also in the theory of classical Willmore surfaces \cite{S86,KS12,R16} existence of minimizers is a highly nontrivial question.

In this work we are interested in minimizing \emph{convex} curves, w.l.o.g. they are in arc length parametrization, more precisely we consider the class
\begin{equation}\label{def:setX}
X:=\left\lbrace \gamma\in C^1(\S^1,\R^2)\Bigg|\, \begin{aligned}
&\mbox{$\gamma$ is a homeomorphism}\\ 
&\left|\gamma'\right|\equiv 1 \mbox{  a.e. in  }\S^1\\
&\mbox{$\gamma$ is convex in $\S^1$}\\
\end{aligned} \right\rbrace.
\end{equation}

The main result of this paper is the following:
\begin{theorem}\label{MainTheorem-23}
For any $s \in (0,1)$ there exists a minimizer  of $\mathscr{W}_s$ in $X$.
\end{theorem}

In~\cite{blatt2023fractional} we obtained an analogous result for $C^1$-convex curves for the \emph{subcritical} case $p > \frac{1}{s}$. There, the arguments crucially rely on the non-scaling invariance of the subcritical regime. In our critical regime we need completely different techniques, more similar to the minimizing methods used for Willmore's energy for surfaces (which is also scaling invariant).

The nonlocal Willmore energy is only defined on co-dimension one objects, since the nonlocal perimeter is only defined on those sets (albeit, see \cite{MS23}). In particular, we can reformulate our result for sets $E$ whose boundary are represented by curves $\gamma$.
Let 
\begin{equation*}
C:=\left\lbrace E\subset\R^2\Bigg|\,\begin{aligned}
&\mbox{$E$ is convex and bounded}\\
&\mbox{$\partial E$ is $C^1$}\\
&\mbox{$E$ has nonempty interior}
\end{aligned} \right\rbrace.
\end{equation*}
Then we have
\begin{theorem}
For any $s \in (0,1)$ there exists a minimizer  of $\hat{\mathscr{W}}_{s,\frac{1}{s}}$ in $C$.
\end{theorem}
%

\subsection*{Strategy and Outline}
Our arguments are substantially inspired by the relation between nonlocal Willmore energy and the tangent-point energies, studied first by Buck and Orloff~\cite{buck-orloff}, Gonzalez and Maddocks~\cite{GM99}, and Blatt and Reiter \cite{BR15}. Tangent-point energies are just one class of topological knot or surface energies, and thus related to the M\"obius energy \cite{OH91,FHW94}, the O'Hara energies \cite{OH92,OH94,BRS19}, and Menger curvature energies  \cite{Menger09,Meurer18,Goering21,GG20}.
Indeed, our arguments draw substantial inspiration from the corresponding existence argument for minimizers of the tangent-point energy,  \cite{BRSV21}. There are, however, substantial technical differences that we need to overcome. One is the sign of the mean curvature $H_{s}$, which for general curves could be changing -- we remedy this by working only with convex curves, an unsatisfying but for now seemingly unavoidable assumption (cf. also \cite{CCC20}). Another one is the change of underlying Sobolev space, which creates surprising technical difficulties. The basic idea is as follows:
\begin{itemize}
\item We first show that the Willmore energy controls the BMO-norm of $\gamma'$ in a locally quantifiable way, \Cref{pr:VMOcontrol}.
\item Then we establish that local BMO-control suffices to obtain (quantifiably) local control of the BiLipschitz constant of the curve, \Cref{la:bilipvmo}. This in turn gives quantitative control of the graphical behavior of $\gamma$, \Cref{blatt:bilip}.
\item On sets where we are graphs, we can use convexity to obtain a control in the sharp Sobolev space $W^{1+s,\frac{1}{s}}_2$ \Cref{th:sobolevcontrol}.
\item Having Sobolev control allows us to pass to the limit for sequences. We then show that \emph{qualitatively} any convex curve with finite Willmore energy is actually $C^1$ (a fact, probably interesting on its own), \Cref{n77754fg}.
\item Those are the main ingredients that are needed to obtain reflexivity, lower semicontinuity in the class of convex homeomorphisms, which we carry out in \Cref{th:convergence}
\end{itemize}

{\bf Acknowledgement}
Helpful discussions with S. Blatt, S. Dipierro, J. Scheuer, and E. Valdinocci are gratefully acknowledged. The proof of \Cref{blatt:bilip} is based on an idea by S. Blatt.

A.S. is funded by NSF Career DMS-2044898. A.S. is an Alexander-von-Humboldt Fellow.
G.G. is funded by the Australian Laureate Fellowship FL190100081 “Minimal surfaces, free boundaries and partial differential equations”.
\section{Preliminaries and Notation}

For a vector $v =(v_1,v_2)^T \in \R^2$ we denote its rotation by $-\frac{\pi}{2}$
\[
v^\perp = (v_2,-v_1)^T.
\]
We also use the notation 
\[
R_{\theta} v
\]
for counterclockwise rotation by $\theta$.
For two nonnegative quantities $A,B$ we write $A \aleq B$ if there exists a multiplicative constant $C$ such that $A \leq C\, B$, and write $A \aeq B$ if $A \aleq B$ and $B \aleq A$. Special dependencies in such inequalities are denoted by subscript, but essentially all constants depend on $s$.

\subsection{Sobolev spaces}
Our argument relies on a suitable Sobolev control. It turns out, that the ``usual'' Gagliardo-Slobokdeckij space seems not to be the right space to work with -- instead we work with the Bessel potential space. However, for localization purposes it is beneficial to keep a seminorm similar to the Gagliardo seminorm. 
\begin{definition}
For $I \subset \R$, the Sobolev space $W^{s,p}_2(I)$ is defined in the usual way via its semi-norm,
\[
 [f]_{W^{s,p}_2(I)} := \brac{\int_{I} \brac{\int_{I} \frac{|f(x)-f(y)|^2}{|x-y|^{1+2s}} dy}^{\frac{p}{2}}\, dx}^{\frac{1}{p}}.
\]
\end{definition}

The main reason we are interested in using this space is the following identification, which is due to Stein \cite[Theorem 1]{Stein61}, for the general versions we refer to \cite{PratsSaksman17,LifengWang23}. To state the result, we adopt the notation
\begin{equation*}
\left|D\right|^sf:=(-\Delta)^\frac{s}{2}f.
\end{equation*}
Here, $(-\Delta)^{\frac{s}{2}}$ denotes the usual fractional Laplacian, cf. \cite{Hitchhiker}.

\begin{theorem}\label{th:steinident}
Assume that $s \in (0,1)$ and $p \in (\frac{2}{1+2s},\infty)$, then 
\[
 [f]_{W^{s,p}_2(\R)} \aeq \|\laps{s} f\|_{L^p(\R)}.
\]
In particular, for any $s \in (0,1)$
\[
 [f]_{W^{s,\frac{1}{s}}_2(\R)} \aeq \|\laps{s} f\|_{L^{\frac{1}{s}}(\R^n)}.
\]
\end{theorem}

We are interested in local estimates, and for this will need a Sobolev-type inequality on a ball. Towards this, we first observe the following Poincar\'e type inequality
\begin{lemma}\label{la:poinc}
For $0 <t <1$ we have for the interval $B =(-1,1)$ and any $f \in L^2(B)$
\[
\brac{\int_{B} \int_{B} |f(x)-f(y)|^{\frac{1}{t}} \, dx\, dy}^{t}+ 
  \brac{\int_{B} \int_{B} |f(x)-f(y)|^2 \, dx\, dy}^{\frac{1}{2}} \aleq [f]_{W^{t,\frac{1}{t}}_2(B)}
 \]
 In particular, if $\int_{B} f = 0$ we have 
 \[
  \|f\|_{L^{\frac{1}{t}}(B)} + \|f\|_{L^{2}(B)} \aleq  [f]_{W^{t,\frac{1}{t}}_2(B)}.
 \]
\end{lemma}
\begin{proof}
The inequality
\begin{equation}\label{eq:la:poincL1tclaim}
\brac{\int_{B} \int_{B} |f(x)-f(y)|^{\frac{1}{t}} \, dx\, dy}^{t} \aleq [f]_{W^{t,\frac{1}{t}}_2(B)}
\end{equation}
is a consequence of H\"older's inequality. Indeed, denote by $(f)_{B} := |B|^{-1} \int f$, then, since $t \in (0,1)$,
\[
\begin{split}
\brac{\int_{B} \int_{B} |f(x)-f(y)|^{\frac{1}{t}} \, dx\, dy}^{t} \leq& 2\brac{\int_{B} \int_{B} |f(x)-(f)_B|^{\frac{1}{t}} \, dx\, dy}^{t}\\ 
=& C\brac{ \int_{B} |f(x)-(f)_B|^{\frac{1}{t}} \, dx}^{t}\\
\aleq& \brac{ \int_{B} \brac{\int_{B}|f(x)-f(z)|\, dz}^{\frac{1}{t}} \, dx}^{t}\\
\aleq& \brac{ \int_{B} \brac{\int_{B}|f(x)-f(z)|^2\, dz}^{\frac{1}{2t}} \, dx}^{t}\\ 
\aleq& [f]_{W^{t,\frac{1}{t}}_2(B)}.
\end{split}
\]
This proves \eqref{eq:la:poincL1tclaim}.

The inequality
\begin{equation}\label{eq:la:poincL2claim}
 \brac{\int_{B} \int_{B} |f(x)-f(y)|^2 \, dx\, dy}^{\frac{1}{2}} \aleq [f]_{W^{t,\frac{1}{t}}_2(B)}
\end{equation}
follows from \eqref{eq:la:poincL1tclaim} and H\"older's inequality if $\frac{1}{t} \geq 2$. So from now on we will focus on proving \eqref{eq:la:poincL2claim} under the condition

\begin{equation}\label{eq:la:poinc:tsmall}
 \frac{1}{t} \leq 2.
\end{equation}
We argue by contradiction. Assume \eqref{eq:la:poincL2claim} is not true for any constant in $\aleq$. Then there exists a sequence of counterexamples $f_k \in L^2(B)$ such that 
\[
  \brac{\int_{B} \int_{B} |f_k(x)-f_k(y)|^2\, dx\, dy}^{\frac{1}{2}} > k\, [f_k]_{W^{t,\frac{1}{t}}_2(B)}
 \]
Both sides do not change if we replace $f_k$ by $f_k - (f_k)_{B}$, where $(f_k)_{B} := \frac{1}{2} \int_{B} f_k$, so w.l.o.g. we may assume 
\begin{equation}\label{eq:la:poinc:mv}
 (f_k)_{B} = 0 \quad \mbox{for all}\quad k.
\end{equation}
Dividing both sides by $\brac{\int_{B} \int_{B} |f_k(x)-f_k(y)|^2}^{\frac{1}{2}}$ we may also assume that for all $k \in \N$
\begin{equation*}
   \brac{\int_{B} \int_{B} |f_k(x)-f_k(y)|^2 dx\, dy}^{\frac{1}{2}} = 1\\
   \end{equation*}
   and
   \begin{equation}\label{eq:la:poinc:Wt1tleq1k}
  [f_k]_{W^{t,\frac{1}{t}}_2(B)} \leq \frac{1}{k}.
\end{equation}
The first equality, combined with the assumption $(f)_{B}=0$ implies in particular 
\begin{equation}\label{eq:la:poinc:L2bdd}
\frac{1}{2}\leq \|f_k\|_{L^2(B)} \leq 1.
\end{equation}
Fix $\eta \in C_c^\infty((-5/4,5/4))$, $\eta \equiv 1$ in $(-1,1)$, and set
\[
 g_k := \eta \tilde{f}_k,
\]
where
\[
 \tilde{f}_k(x) := \begin{cases}
                  f_k(-2-x) \quad &x \in (-2,-1)\\
                  f_k(x) \quad &x \in (-1,1)\\
                  f_k(2-x) \quad &x \in (1,2).\\
                 \end{cases}
\]
We see that 
\begin{equation}\label{eq:la:poinc:unifbddL2}
 \sup_{k}\|g_k\|_{L^2(\R)} \aleq \sup_{k}\|\tilde{f}_k\|_{L^2((-2,2))} \aleq \sup_{k} \|f_k\|_{L^2((-1,1))} \leq 1.
\end{equation}
Moroever,
\[
\begin{split}
  &\int_{\R} \brac{\int_{\R} \frac{|g_k(x)-g_k(y)|^2}{|x-y|^{1+2t}} dy}^{\frac{1}{2t}}\, dx\\
\aleq&  \int_{(-3/2,3/2)} \brac{\int_{(-3/2,3/2)} \frac{|g_k(x)-g_k(y)|^2}{|x-y|^{1+2t}} dy}^{\frac{1}{2t}}\, dx\\
&+\int_{\R \setminus (-3/2,3/2)} \brac{\int_{(-5/4,5/4)} \frac{|g_k(y)|^2}{(1+|x|)^{1+2t}} dy}^{\frac{1}{2t}}\, dx\\
&+\int_{(-5/4,5/4)} \brac{\int_{\R \setminus (-3/2,3/2)} \frac{|g_k(x)|^2}{(1+|y|)^{1+2t}} dy}^{\frac{1}{2t}}\, dx\\
\aeq&  \int_{(-3/2,3/2)} \brac{\int_{(-3/2,3/2)} \frac{|g_k(x)-g_k(y)|^2}{|x-y|^{1+2t}} dy}^{\frac{1}{2t}}\, dx\\
&+ \brac{\int_{(-5/4,5/4)} |g_k(y)|^2 dy}^{\frac{1}{2t}}\\
&+\int_{(-5/4,5/4)} |g_k(x)|^{\frac{1}{t}}  dx\\
\aleq&\int_{(-3/2,3/2)} \brac{\int_{(-3/2,3/2)} \frac{|g_k(x)-g_k(y)|^2}{|x-y|^{1+2t}} dy}^{\frac{1}{2t}}\, dx + \|g_k\|_{L^2(\R)}^{\frac{1}{t}}.
  \end{split}
\]
In the last inequality we used \eqref{eq:la:poinc:tsmall}.

Next we use the discrete product rule
\[
  \eta(x) \tilde{f}_k(x)-\eta(y) \tilde{f}_k(y) = \eta(y) \brac{\tilde{f}_k(x)-\tilde{f}_k(y)} + (\eta(x) -\eta(y)) \tilde{f}_k(x)
\]
to obtain
\[
\begin{split}
 &\int_{(-3/2,3/2)} \brac{\int_{(-3/2,3/2)} \frac{|g_k(x)-g_k(y)|^2}{|x-y|^{1+2t}} dy}^{\frac{1}{2t}}\, dx \\
 \aleq&\int_{(-3/2,3/2)} \brac{\int_{(-3/2,3/2)} \frac{|\tilde{f}_k(x)-\tilde{f}_k(y)|^2}{|x-y|^{1+2t}} dy}^{\frac{1}{2t}}\, dx \\
 &+\int_{(-3/2,3/2)} |\tilde{f}_k(x)|^{\frac{1}{t}} \brac{\int_{(-3/2,3/2)} \frac{|x-y|^2}{|x-y|^{1+2t}} dy}^{\frac{1}{2t}}\, dx \\
 \aleq& \int_{(-3/2,3/2)} \brac{\int_{(-3/2,3/2)} \frac{|\tilde{f}_k(x)-\tilde{f}_k(y)|^2}{|x-y|^{1+2t}} dy}^{\frac{1}{2t}}\, dx + \|\tilde{f}_k\|_{L^2((-3/2,3/2))}\\
 \end{split}
\]
where in the last inequality we used $t < 1$ and \eqref{eq:la:poinc:tsmall}.
Next, by substitution and our choice of reflection we find 
\[
\begin{split}
 &\int_{(-3/2,3/2)} \brac{\int_{(-3/2,3/2)} \frac{|\tilde{f}_k(x)-\tilde{f}_k(y)|^2}{|x-y|^{1+2t}} dy}^{\frac{1}{2t}}\, dx\\
 \aleq&\int_{(-1,1)} \brac{\int_{(-1,1)} \frac{|f_k(x)-f_k(y)|^2}{|x-y|^{1+2t}} dy}^{\frac{1}{2t}}\, dx.\\
 \end{split}
\]
In summary, we have shown 
\[
 \sup_{k} \|g_k\|_{L^2(\R)} + [g_k]_{W^{t,\frac{1}{t}}_2(\R)} < +\infty.
\]
By \Cref{th:steinident} and Rellich-Kondrachov theorem we find a subsequence (not relabelled) such that $g_k \xrightarrow{k \to \infty} g$ strongly in $L^2_{loc}(\R)$. Since $f_k = g_k$ in $(-1,1)$ we conclude that for $f:= g\Big |_{(-1,1)}$ we have in view of \eqref{eq:la:poinc:L2bdd}
\[
 \|f\|_{L^2((-1,1))} = \lim_{k \to \infty} \|g_k\|_{L^2((-1,1))}=\lim_{k \to \infty} \|f_k\|_{L^2((-1,1))}\geq \frac{1}{2}.
\]
Moreover, by strong $L^2$-convergence and \eqref{eq:la:poinc:mv}
\[
 (f)_{B} =0.
\]
On the other hand, by Fatou's lemma and \eqref{eq:la:poinc:Wt1tleq1k} we have 
\[
 [f]_{W^{t,\frac{1}{t}}_2((-1,1))} \leq \liminf_{k \to \infty} [f_k]_{W^{t,\frac{1}{t}}_2((-1,1))}  = 0,
\]
thus $f \equiv const$ -- since $(f)_{B} = 0$ we conclude $f \equiv 0$ contradicting $\|f\|_{L^2(-1,1)}\geq \frac{1}{2}$. This proves \eqref{eq:la:poincL2claim} under the assumption \eqref{eq:la:poinc:tsmall}. We can conclude.
\end{proof}

\begin{lemma}[Sobolev inequality]\label{la:sob}
For $0 < s <t <1$ we have for any interval $B \subset \R$
\[
 [f]_{W^{s,\frac{1}{s}}_2(B)} \aleq_{s,t} [f]_{W^{t,\frac{1}{t}}_2(B)}
 \]
\end{lemma}
\begin{proof}
From \Cref{th:steinident} (observe that $\frac{1}{s} > \frac{2}{1+2s}$ is always satisfied) and the Sobolev embedding theorem
\begin{equation}\label{sobineq.o}
 [f]_{W^{s,\frac{1}{s}}_2(\R)} \aeq \|\laps{s} f\|_{L^{\frac{1}{s}}(\R)} \aleq  \|\laps{t} f\|_{L^{\frac{1}{t}}(\R)}  \aeq [f]_{W^{t,\frac{1}{t}}_2(\R)}.
\end{equation}
Now, if $f: B \to \R$, we can assume by scaling w.l.o.g. $B= (-1,1)$. By invariance under additive constants we also assume 
\begin{equation}\label{eq:la:sob:mvzero}
(f)_{B} = 0.
\end{equation}
Set 
\[
 \tilde{f}(x) := \begin{cases}
                  f(-2-x) \quad x \in (-2,-1)\\
                  f(x) \quad x \in (-1,1)\\
                  f(2-x) \quad x \in (1,2).\\
                 \end{cases}
\]
Then, for every $\eta \in C_c^\infty((-3/2,3/2))$, $\eta \equiv 1$ in $(-1,1)$, using~\eqref{sobineq.o} we have 
\[
 [f]_{W^{s,\frac{1}{s}}_2((-1,1))}\leq [\eta \tilde{f}]_{W^{s,\frac{1}{s}}_2(\R)} \aleq [\eta \tilde{f}]_{W^{t,\frac{1}{t}}_2(\R)}.
\]
In the proof of \Cref{la:poinc} we already estimated that
\[
 [\eta \tilde{f}]_{W^{t,\frac{1}{t}}_2(\R)} \aleq [f]_{W^{t,\frac{1}{t}}_2(B)} + \|f\|_{L^{\frac{1}{t}}(B)} + \|f\|_{L^{2}(B)}.
\]
Since we have \eqref{eq:la:sob:mvzero}, the statement of \Cref{la:poinc} implies 
\[
 [f]_{W^{s,\frac{1}{s}}_2((-1,1))} \aleq [\eta \tilde{f}]_{W^{t,\frac{1}{t}}_2(\R)}\aleq [f]_{W^{t,\frac{1}{t}}_2(B)}.
\]
We can conclude.
\end{proof}

We also record the following important potential formula for the fractional Laplacian. It is most likely well-known to experts, but we give the argument for completeness.
\begin{lemma}\label{Toperator}
Let $f\in C_c^\infty(\R)$. Then for each $s\in (0,1)$ and $x\in \R$ it holds that 
\begin{equation}
\int_{\R} \frac{f(x)-f(y)-f'(y)(x-y)}{\vert x-y \vert^{1+s+1}}\,dy=C_{s} (-\Delta)^{\frac{s+1}{2}} f(x),
\end{equation}
for some $C_{s}$ depending only on $s$.
\end{lemma}

\begin{proof}
We notice that 
\[
\begin{split}
 Tf(x) :=& \int_{\R} \frac{f(x)-f(y)-f'(y)(x-y)}{|x-y|^{2+s}}\, dy\\
 =& \int_{\R} \frac{f(x)-f(x+h)+f'(x+h)(h)}{|h|^{2+s}}\, dh\\
 =& \frac{1}{2} \int_{\R} \frac{f(x)-f(x+h)+f'(x+h)(h)}{|h|^{2+s}}\, dh+\frac{1}{2} \int_{\R} \frac{f(x)-f(x-h)+f'(x-h)(-h)}{|h|^{2+s}}\, dh\\
 =& \frac{1}{2} \int_{\R} \frac{2f(x)-f(x+h)-f(x-h)+f'(x+h)(h) -f'(x-h)(h)}{|h|^{2+s}}\, dh\\
 =& \frac{1}{2} \int_{\R} \frac{2f(x)-f(x+h)-f(x-h)+\partial_h f(x+h)\, (h) +\partial_h f(x-h)(h)}{|h|^{2+s}}\, dh\\
 =& \frac{1}{2} \int_{\R} \frac{2f(x)-f(x+h)-f(x-h)+\partial_h \brac{f(x+h)+f(x-h)}\, (h) }{|h|^{2+s}}\, dh\\
 =& \frac{1}{2} \int_{\R} \frac{2f(x)-f(x+h)-f(x-h)-\partial_h \brac{2f(x) - f(x+h)-f(x-h)}\, (h) }{|h|^{2+s}}\, dh.\\
 \end{split}
\]
Integrating by parts  and observing that 
\[
 \partial_h \frac{h}{|h|^{2+s}} = \frac{|h|^{2+s}-h (2+s) |h|^{1+s} \frac{h}{|h|}}{|h|^{4+2s}} = -(1+s) \frac{1}{|h|^{2+s}}
,\]
we deduce that
\[
\begin{split}
 Tf(x)  =& \frac{1}{2} \int_{\R} \frac{2f(x)-f(x+h)-f(x-h)-\partial_h \brac{2f(x) - f(x+h)-f(x-h)}\, (h) }{|h|^{2+s}}\, dh\\
 =& \frac{1}{2} \int_{\R} \frac{2f(x)-f(x+h)-f(x-h)-(1+s) \brac{2f(x) - f(x+h)-f(x-h)} }{|h|^{2+s}}\, dh\\
 =& \frac{1-(1+s)}{2} \int_{\R} \frac{2f(x)-f(x+h)-f(x-h)}{|h|^{2+s}}\, dh\\
 =& \frac{1-(1+s)}{2} c_{d,s} (-\Delta)^{\frac{s+1}{2}} f(x).
 \end{split}
\]

\end{proof}

\subsection{The fractional mean curvature}
While it is irrelevant for most calculation-based analysis, for convenience and geometric arguments we shall assume that whenever possible, curves are parametrized counter-clockwise,
\begin{equation}\label{conclock}
\gamma\mbox{  is a counter clockwise parametrization of its image}, 
\end{equation}
For some interval $I\subset\R$ and $\gamma\in Lip(I,\R^2)$, for $s\in (0,1)$ and $p\geq 1$, we define the $s,p$-nonlocal Willmore energy of $\gamma$ as 
\begin{equation*}
\mathscr{W}_{s,p}\brac{\gamma\big |_{I}}=\left(\int_{I}\left(\int_{I} \frac{\langle n(\gamma(y)),\gamma(y)-\gamma(x)\rangle}{\left|\gamma(y)-\gamma(x)\right|^{2+s}}\left|\gamma'(y)\right|\,dy\right)^p\left|\gamma'(x)\right|\,dx\right)^\frac{1}{p},
\end{equation*}
where $n(\gamma(y))$ is a.e. a unit normal to the curve $\gamma$ at $\gamma(y)$.
We remind the reader of our abuse of notation by writing $\mathscr{W}_{s} := \mathscr{W}_{s,\frac{1}{s}}$.
Under the assumption \eqref{conclock} for a.e. $y\in I$ we define fix the direction of the vectorfield $n(\gamma(y))$ as
\begin{equation*}
n(\gamma(y)):=\gamma'(y)^\perp.
\end{equation*}

We define
\begin{equation*}
\tilde{X}:=\left\lbrace \gamma\in Lip\left(\S^1,\R^2\right)\Bigg|\, \begin{aligned}
&\mbox{$\gamma$ is a homeomorphism}\\ 
&\left|\gamma'\right|\equiv 1 \mbox{  a.e. in  }\S^1\\
&\mbox{$\gamma$ is convex in $\S^1$}\\
 &\mbox{$\gamma$ satisfies~\eqref{conclock} in $\S^1$}
\end{aligned} \right\rbrace.
\end{equation*}
Clearly $X$ in \eqref{def:setX} is a subset of $\tilde{X}$ -- up to inverting the orientation so that \eqref{conclock} holds.

For every $\gamma\in \tilde{X}$ the $s,p$-nonlocal Willmore energy takes the form
\begin{equation}\label{sdcv6530}
\mathscr{W}_{s,p}(\gamma):=\left(\int_{\S^1}\left(\int_{\S^1}\frac{\langle n(\gamma(y)),\gamma(y)-\gamma(x)\rangle}{\left|\gamma(y)-\gamma(x)\right|^{2+s}} \right)^p\,dx\right)^s.
\end{equation}

In the definitions of $X$ and $\tilde{X}$ we assumed convexity, the following is the precise definition we are going to use.
\begin{definition}[convexity]\label{def:convexcurve}
Let $\gamma: I \to \R^2$ be continuous. Given some subset $A\subset I$, we say that $\gamma$ is convex in $A$ if 
\begin{equation}\label{fcdghte}
\begin{split}
\mbox{for all  }& x\in A\mbox{  there exists  } \nu_x\in \S^1 \mbox{ such that  }\\
&\nu_x\cdot (\gamma(x)-\gamma(y))\geq 0 \mbox{  for all  } y\in A.
\end{split}
\end{equation}  
\end{definition}

It is worth pointing out that for every $\gamma\in \tilde{X}$ and for a.e. $x\in\S^1$ one has that $\nu_x$ satisfying the inequality in~\eqref{fcdghte} is unique, and actually $\nu_x=n(\gamma(x))$, see for instance Corollary~\ref{uniqnorma}. In particular for convex curves we have
\[
 \langle n(\gamma(y),\gamma(y)-\gamma(x)\rangle \geq 0 \quad \text{a.e. $x,y \in I$}
\]

Our analysis in this paper relies on a relation between curves with finite Willmore energy and curves that belong to the Sobolev space $W^{1+s,\frac{1}{s}}_2$. One direction is the following Lemma
\begin{lemma}\label{la:energyspace} Let $s \in (0,1)$ and $\gamma: \S^1 \to \R^2$ be a BiLipschitz homeomorphism with $|\gamma'| \equiv 1$ a.e. in $\S^1$ and such that 
\[
 [\gamma']_{W^{s,\frac{1}{s}}_2(\S^1)} < +\infty.
\]
Then, it holds that 
\[
\mathscr{W}_{s}(\gamma) =\int_{\S^1}\abs{\int_{\S^1}{\frac{\langle n(\gamma(y)),\gamma(y)-\gamma(x)\rangle}{\left|\gamma(y)-\gamma(x)\right|^{2+s}}} dy}^{\frac{1}{s}}\,dx < +\infty.
\]

\end{lemma}
Let us stress two things: We will not really use the previous lemma in our analysis, obtaining a quantitative converse of \Cref{la:energyspace} is the main challenge of the present paper -- see \Cref{s:sobcontrol}. Secondly, the assumption that $\gamma$ is a homeomoprhism is not needed, certain crossings and figure-eight shapes are permissible, this will be investigated more in future work.

\begin{proof}
Since $\langle n(\gamma(y)), \gamma'(y)\rangle = 0$ a.e. we have
\[
\begin{split}
&\langle n(\gamma(y)),\gamma(y)-\gamma(x)\rangle\\
=&\langle n(\gamma(y)), \gamma(y)-\gamma(x)-\gamma'(y)(y-x)\rangle.
\end{split}
\]
Since $\gamma$ is a BiLipschitz homeomorphism we also have $|\gamma(x)-\gamma(y)| \aeq |x-y|$.

Then, we first observe that 
\[
\begin{split}
&\int_{\S^1}\left(\int_{\S^1}\abs{\frac{\langle n(\gamma(y)) - n(\gamma(x)), \gamma(y)-\gamma(x)-\gamma'(y)(y-x)\rangle }{\left|\gamma(y)-\gamma(x)\right|^{2+s}}} \,dy\right)^\frac{1}{s}\,dx\\
\aleq&\int_{\S^1}\left(\int_{\S^1}\frac{|\gamma'(y)-\gamma'(x)|\, |\gamma(y)-\gamma(x)-\gamma'(y)(x-y)|}{|x-y|^{2+s}}\, dy \right)^\frac{1}{s}\,dx.\\
\end{split}
\]
Using the fundamental theorem of calculus, we have 
\[
|\gamma(y)-\gamma(x)-\gamma'(y)(x-y)| \leq |x-y|\int_{[x,y]} |\gamma'(z)-\gamma'(y)|\, dz
\]
By H\"older's inequality we then find 
\[
\begin{split}
&\int_{\S^1}\left(\int_{\S^1}\abs{\frac{\langle n(\gamma(y)) - n(\gamma(x)), \gamma(y)-\gamma(x)-\gamma'(y)(y-x)\rangle }{\left|\gamma(y)-\gamma(x)\right|^{2+s}}} dy\right)^\frac{1}{s}\,dx\\
\aleq&[\gamma']_{W^{\frac{s}{2},\frac{2}{s}}_2(\S^1)}^2\\
\aleq&[\gamma']_{W^{s,\frac{1}{s}}_2(\S^1)}^2.
\end{split}
\]
In the last line we use Sobolev embedding \Cref{la:sob}.

Thus, in order to show $\mathscr{W}_s(\gamma) < \infty$ we exchange $n(\gamma(y))$ with $n(\gamma(x))$ and simply need to show 
\[
\begin{split}
&\int_{\S^1}\abs{\int_{\S^1}{\frac{\langle n(\gamma(x)),\gamma(y)-\gamma(x)-\gamma'(y)(x-y)\rangle}{\left|\gamma(y)-\gamma(x)\right|^{2+s}}} dy}^{\frac{1}{s}}\,dx\\
\aleq& \int_{\S^1}\abs{\int_{\S^1}{\frac{\gamma(y)-\gamma(x)-\gamma'(y)(x-y)}{\left|\gamma(y)-\gamma(x)\right|^{2+s}}} dy}^{\frac{1}{s}}\,dx < +\infty.
\end{split}
\]
Next, we observe (see e.g. \cite{BRS16,BRSV21} and references within)
\[
\begin{split}
&1-\frac{|x-y|^2}{|\gamma(y)-\gamma(x)|^2}\\
=&\frac{|\gamma(y)-\gamma(x)|^2-|x-y|^2}{|\gamma(x)-\gamma(y)|^2}\\
=&\frac{\int_{[x,y]} \int_{[x,y]} \gamma'(z_1) \cdot \gamma'(z_2) -1 \,dz_1\, dz_2 }{|\gamma(x)-\gamma(y)|^2}\\
=&\frac{\int_{[x,y]} \int_{[x,y]} |\gamma'(z_1) - \gamma'(z_2)|^2\, dz_1\, dz_2}{|\gamma(x)-\gamma(y)|^2}\\
\end{split}
\]
This implies that 
\[
\abs{1-\frac{|x-y|^2}{|\gamma(y)-\gamma(x)|^2}} \aeq 
=\mvint_{[x,y]} \mvint_{[x,y]} |\gamma'(z_1) - \gamma'(z_2)|^2\, dz_1\, dz_2.\\
\]
Thus 
\[
\begin{split}
& \int_{\S^1}\abs{\int_{\S^1}{\frac{\gamma(y)-\gamma(x)-\gamma'(y)(x-y)}{\left|\gamma(y)-\gamma(x)\right|^{2+s}}} \,dy}^{\frac{1}{s}}\,dx\\
\aleq&\int_{\S^1}\abs{\int_{\S^1}{\frac{\gamma(y)-\gamma(x)-\gamma'(y)(x-y)}{\left|x-y\right|^{2+s}}}\, dy}^{\frac{1}{s}}\,dx + [\gamma']_{W_2^{\frac{s}{3},\frac{3}{s}}(\S^1)}^3\\
\aleq&\int_{\S^1}\abs{\int_{\S^1}{\frac{\gamma(y)-\gamma(x)-\gamma'(y)(x-y)}{\left|x-y\right|^{2+s}}}\, dy}^{\frac{1}{s}}\,dx + [\gamma']_{W_2^{s,\frac{1}{s}}(\S^1)}^3,
\end{split}
\]
where in the last line we used again the Sobolev embedding \Cref{la:sob}. We now conclude by using \Cref{Toperator}, which -- after a sterographic projection to the sphere implies 
\[
\int_{\S^1}\abs{\int_{\S^1}{\frac{\gamma(y)-\gamma(x)-\gamma'(y)(x-y)}{\left|x-y\right|^{2+s}}}\, dy}^{\frac{1}{s}}\,dx \aleq \|\laps{s+1} \gamma\|_{L^{\frac{1}{s}} (\S^1)} \aeq [\gamma']_{W_2^{s,\frac{1}{s}}(\S^1)},
\]
where in the last inequality we used Stein's theorem, \Cref{th:steinident}.
\end{proof}

\subsection{The maximum principle for the nonlocal mean curvature}
The maximum principle for nonlocal mean curvature is well known and will play a crucial role in our argument.
\begin{proposition}[Maximum Principle for the nonlocal mean curvature]\label{milan}
Let $A,B\subset\R^n$ be measurable and $s\in (0,1)$. Moreover, we assume that $A\subset B$ and there exists some point $z\in\partial A\cap \partial B$. 

{If $H^s_A(z)$ and $H^s_{B}(z)$ exist, then we have}
\begin{equation}\label{MPNMC}
H_{A}^s(z)\geq H_{B}^s(z).
\end{equation} 
Furthermore, if $B$ is convex it holds that 
\begin{equation}\label{POCVNBBG}
H_{B}^s(z)\geq 0.
\end{equation}
{for any $z$ where $H_{B}^s(z)$ exists.}
\end{proposition} 

\begin{proof}
According to the hypothesis we have that 
\begin{equation*}
B^c\subset A^c.
\end{equation*}
Therefore, for each $y\in \R^n$ we have that 
\begin{equation*}
\chi_{A^c}(y)-\chi_{A}(y)\geq \chi_{B^c}(y)-\chi_{B}(y).
\end{equation*}
From this equation we evince that for each $\epsilon>0$ 
\begin{equation*}
\int_{\R^n\setminus B_{\epsilon}(z)} \frac{\chi_{A^c}(y)-\chi_{A}(y)}{\left|z-y\right|^{n+s}}\,dy\geq \int_{\R^n\setminus B_{\epsilon}(z)}\frac{ \chi_{B^c}(y)-\chi_{B}(y)}{\left|z-y\right|^{n+s}}\,dy.
\end{equation*}
Taking the limit with respect to $\epsilon\to 0^+$ in both sides we conclude the proof of~\eqref{MPNMC}, since we assume that this limit exists.

Since $B$ is convex, for each $z\in \partial B$ there exists an half plane $\pi_z$ such that 
\begin{equation*}
B\subset \pi_z.
\end{equation*}
Therefore, applying~\eqref{MPNMC} we obtain that 
\begin{equation*}
H_{B}^s(z)\geq H_{\pi_z}^s(z)=0.\qedhere
\end{equation*}
\end{proof}

In the following proposition we establish sufficient conditions for the boundary parametrization of a set $E\subset\R^2$ in order to rewrite the nonlocal mean curvature as a surface integral. We stress that this result is not sharp  and that it is possible to determine more  general conditions on the regularity of a set for such identity to hold. The proof of Proposition~\ref{salame} is in Appendix~\ref{catamarano}.

\begin{proposition}\label{salame}
Let $s\in (0,1)$, $\gamma\in Lip(\S^1,\R^2)$ be a homeomorphism such that $\left|\gamma'\right|\equiv 1$ a.e. in $\S^1$ and $\gamma$ satisfies the convexity assumption~\eqref{fcdghte} in $\S^1$. Then, if we denote $E:=\mbox{co}\left(\gamma\left(\S^1\right)\right)$, for a.e. $x\in \S^1$ it holds that 
\begin{equation}\label{-fgrop0-0}
H_{\partial E}^s(\gamma(x))=\frac{2}{s}\int_{\S^1} \frac{\langle n(\gamma(y)),\gamma(y)-\gamma(x)\rangle}{\left|\gamma(y)-\gamma(x)\right|^{2+s}}\,dy,
\end{equation}  
where $n(\gamma(y)):=\gamma'(y)^\perp$.
\end{proposition}

\section{Small energy controls BMO-parametrization}\label{VMO}
The main result of this section is the fact hat when the Willmore energy is suitably small, then the $BMO$-norm of the derivative of the parametrization $\gamma$ is dominated by the Willmore energy. More precisely we have

\begin{proposition}\label{pr:VMOcontrol}
For any $p \in [1,\infty)$, $s \in (0,1)$ there exists a uniform $\eps > 0$ such that the following holds:

Assume that $\gamma \in C^1(B_{R_0},\R^2)$%
, $|\gamma'| \equiv 1$ satisfies (where $n(\gamma(x)) = \gamma'(x)^\perp$)
\begin{equation}\label{eq:absvaluewillmore}
  \int_{B_{R_0}} \brac{\int_{B_{R_0}} \frac{|\langle n(\gamma(y_),\gamma(y)-\gamma(x)\rangle|}{|\gamma(x)-\gamma(y)|^{2+s}}\, dy}^{\frac{1}{s}}\, dx< \eps
\end{equation}
then for all $B_r(x_0) \subset B_{R_0}$
\begin{equation}\label{eq:goalVMO}
\begin{split}  
&\brac{\mvint_{B_r(x_0)}  \mvint_{B_r(x_0)}|\gamma'(x)-\gamma'(y)|^p\, dx\, dy}^{\frac{1}{p}} \\
&\aleq_{p} \brac{\int_{B_{R_0}} \brac{\int_{B_{R_0}} \frac{|\langle n(\gamma(y)),\gamma(y)-\gamma(x)\rangle|}{|\gamma(x)-\gamma(y)|^{2+s}} \,dy}^{\frac{1}{s}}\, dx}^s.
\end{split}
\end{equation}
\end{proposition}
Let us point out a subtlety in \Cref{pr:VMOcontrol}. There is no convexity assumption. However, only if $\gamma$ is a convex curve $\langle n(\gamma(y)),\gamma(y)-\gamma(x)\rangle$ has a sign, and the absolute value inside the integral can be ignored -- that is for convex curves \eqref{eq:absvaluewillmore} is indeed the Willmore energy restricted to a ball. 

In order to prove Proposition~\ref{pr:VMOcontrol}, we need the following local John-Nirenberg type result, see~\cite[Corollary 6.22]{MR3099262}.
\begin{lemma}\label{la:johnirenberglocal}
For any $u \in L^\infty(B_1)$ and any $p \in [1,\infty)$ we have
\begin{equation}
\begin{split}
\sup_{B_r(x_0) \subset B_1} &\brac{\mvint_{B_r(x_0)} \mvint_{B_r(x_0)}|u(x)-u(y)|^p\, dx\, dy}^{\frac{1}{p}} \\
&\aleq_p \sup_{B_r(x_0) \subset B_1}\mvint_{B_r(x_0)} \mvint_{B_r(x_0)}|u(x)-u(y)|\, dx\, dy,
\end{split}
\end{equation}
whenever the right-hand side is finite.
\end{lemma}

Lastly we need a certain mean valued estimate for double intergrals.
\begin{lemma}\label{la:meanvalueest}
For any $p > 1$
\[
 \mvint_{B(x_0,r)} \mvint_{B(x_0,r)} \mvint_{[x,y]} |U(y,z)|\, dz\, dx\, dy \aleq_p \brac{\mvint_{B(x_0,r)} \mvint_{B(x_0,r)} |U(x,y)|^p dx\, dy }^{\frac{1}{p}}
\]

\end{lemma}
\begin{proof}
By translation and scaling, w.l.o.g. $x_0 = 0$ and $r =1$, i.e. we need to show
\[
 \int_{-1}^1 \int_{-1}^1 \mvint_{[x,y]} |U(y,z)|\, dz\, dx\, dy \aleq_p \brac{\int_{-1}^1 \int_{-1}^1 |U(x,y)|^p dx\, dy }^{\frac{1}{p}}.
\]
We have 
\begin{equation*}
\begin{split}
&\int_{-1}^1 \int_{-1}^1 \mvint_{[x,y]} |U(y,z)|\, dz\, dx\, dy\\
=&\int_{-1}^1 \int_{-1}^y \int_{-1}^z\frac{1}{y-x}\left|U(y,z)\right|\,dx\,dz\,dy+\int_{-1}^1 \int_{y}^1 \int_{z}^1\frac{1}{x-y}\left|U(y,z)\right|\,dx\,dz\,dy\\
=&\int_{-1}^1 \int_{-1}^y \left|U(y,z)\right|\log\left(\frac{\left|1+y\right|}{\left|z-y\right|}\right)\,dz\,dy+\int_{-1}^1 \int_{y}^1 \log\left(\frac{\left|1-y\right|}{\left|z-y\right|}\right)\left|U(y,z)\right|\,dz\,dy\\
\leq &\left(\int_{-1}^1 \int_{-1}^1 \left|U(y,z)\right|^p\,dz\,dy\right)^\frac{1}{p}\left(\int_{-1}^1 \int_{-1}^y\left(\log\left(\frac{\left|1+y\right|}{\left|z-y\right|}\right)\right)^{p'}\,dz\,dy\right)^\frac{1}{p'}\\
&+\left(\int_{-1}^1 \int_{-1}^1 \left|U(y,z)\right|^p\,dz\,dy\right)^\frac{1}{p}\left(\int_{-1}^1 \int_{y}^1\left(\log\left(\frac{\left|1-y\right|}{\left|z-y\right|}\right)\right)^{p'}\,dz\,dy\right)^\frac{1}{p'}.
\end{split}
\end{equation*}
Now we use that for any $\gamma > 0$ there exists $C(\gamma)$ such that
\[
 \abs{\log a} \leq C(\gamma) \brac{a^\gamma + a^{-\gamma}} \quad \forall a \in (0,\infty).
\]
Thus, we have for some small $\delta < 1$ (and a suitable choice of $\gamma$)
\begin{equation*}
\begin{split}
&\int_{-1}^1 \int_{-1}^y \brac{\log\brac{\frac{y+1}{\left|z-y\right|}}}^{p'}\, dz\, dy\\
\aleq &\int_{-1}^1 \int_{-1}^y \frac{(y+1)^\delta}{\left|y-z\right|^\delta}\, dz\, dy+\int_{-1}^1 \int_{-1}^y \frac{\left|y-z\right|^\delta}{(y+1)^\delta}\, dz\, dy\\
\leq & 2^\delta\int_{-1}^1 \int_{-1}^{y} \left|y-z\right|^{-\delta}\, dz\, dy+2^\delta\int_{-1}^1 \int_{-1}^{y} (y+1)^{-\delta}\, dz\, dy\\
 =&2^\delta\int_{-1}^1 \frac{1}{1-\delta} (y+1)^{1-\delta}\, dy+2^\delta\int_{-1}^1 (y+1)^{1-\delta}\, dy\\
 =&C(\delta).
\end{split}
\end{equation*}
Analogously, it holds that 
\begin{equation*}
\int_{-1}^1 \int_{y}^1\left(\log\left(\frac{\left|1-y\right|}{\left|z-y\right|}\right)\right)^{p'}\,dz\,dy\aleq C^*(\delta).\qedhere
\end{equation*}
We can conclude.
\end{proof}

\begin{proof}[Proof of \Cref{pr:VMOcontrol}]
Since $|\gamma'| \equiv 1$ we have 
\[
\langle \gamma'(x), \gamma'(x)-\gamma'(z)\rangle = \frac{1}{2} |\gamma'(x)-\gamma'(z)|^2.
\]
With the fundamental theorem of calculus we then have 
\[
\langle \gamma'(y), \gamma(x)-\gamma(y)-\gamma'(y)(x-y) \rangle = \int_{y}^x \langle \gamma'(y), \gamma'(z)-\gamma'(y)\rangle = \frac{1}{2}\int_{y}^x |\gamma'(x)-\gamma'(y)|^2.
\]
Using this and the fact that $n(\gamma(y)),\gamma'(y)$ forms an orthonormal basis of $\R^2$,
\[
\begin{split}
 |\gamma'(x)-\gamma'(y)| =& \frac{|\gamma'(x)(x-y)-\gamma'(y)(x-y)|}{|x-y|}\\
 \leq& \frac{|\gamma(x)-\gamma(y)-\gamma'(y)(y-x)|}{|x-y|} + \frac{|\gamma(y)-\gamma(x)-\gamma'(x)(x-y)|}{|x-y|}\\
 \aleq& \frac{|\langle n(y), \gamma(x)-\gamma(y)-\gamma'(y)(x-y)\rangle |}{|x-y|} + \frac{|\langle n(x),\gamma(y)-\gamma(x)-\gamma'(x)(y-x)\rangle|}{|x-y|}\\
 &+\frac{|\langle \gamma'(y), \gamma(x)-\gamma(y)-\gamma'(y)(x-y)\rangle |}{|x-y|} + \frac{|\langle \gamma'(x),\gamma(y)-\gamma(x)-\gamma'(x)(y-x)\rangle|}{|x-y|}\\
 \aleq&\frac{|\langle n(\gamma(y)), \gamma(x)-\gamma(y)-\gamma'(y)(x-y)\rangle |}{|x-y|} + \frac{|\langle n(\gamma(x)),\gamma(y)-\gamma(x)-\gamma'(x)(y-x)\rangle|}{|x-y|}\\
 &+\mvint_{[x,y]} |\gamma'(y)-\gamma'(z)|^2\, dz + \mvint_{[x,y]} |\gamma'(z)-\gamma'(x)|^2\, dz\\
  =&\frac{|\langle n(\gamma(y)), \gamma(x)-\gamma(y)\rangle |}{|x-y|} + \frac{|\langle n(\gamma(x)),\gamma(y)-\gamma(x)\rangle|}{|x-y|}\\
 &+\mvint_{[x,y]} |\gamma'(y)-\gamma'(z)|^2\, dz + \mvint_{[x,y]} |\gamma'(z)-\gamma'(x)|^2\, dz.
\end{split}
 \]
We now observe that whenever $B_r(x_0) \subset B_R$ (we use that $r\ageq |x-y| \geq |\gamma(x)-\gamma(y)|$ since $|\gamma'| \equiv 1$)
\[
\begin{split}
 &\mvint_{B_r(x_0)} \mvint_{B_r(x_0)} \frac{|\langle n(\gamma(y)), \gamma(x)-\gamma(y)\rangle |}{|x-y|}\, dx\, dy\\
 \aleq&r^{1+s}\mvint_{B_r(x_0)} \mvint_{B_r(x_0)} \frac{|\langle n(\gamma(y)), \gamma(x)-\gamma(y)\rangle |}{|\gamma(x)-\gamma(y)|^{2+s}}\, dy\, dx\\
  =&r^{s-1}\int_{B_r(x_0)} \int_{B_r(x_0)} \frac{|\langle n(\gamma(y)), \gamma(x)-\gamma(y)\rangle |}{|\gamma(x)-\gamma(y)|^{2+s}}\, dy\, dx\\
  \aleq&\brac{\int_{B_r(x_0)} \brac{\int_{B_r(x_0)} \frac{|\langle n(\gamma(y)), \gamma(x)-\gamma(y)\rangle |}{|\gamma(x)-\gamma(y)|^{2+s}}\, dy}^{\frac{1}{s}}\, dx}^s \\
   \aleq&\brac{\int_{B_R} \brac{\int_{B_R} \frac{|\langle n(\gamma(y)), \gamma(x)-\gamma(y)\rangle |}{|\gamma(x)-\gamma(y)|^{2+s}}\, dy}^{\frac{1}{s}}\, dx}^s \\
 \end{split}
\]
and similarly, using Fubini's Theorem, we obtain
\[
\begin{split}
 &\mvint_{B_r(x_0)} \mvint_{B_r(x_0)} \frac{|\langle n(\gamma(x)),\gamma(y)-\gamma(x)\rangle|}{|x-y|}\, dx\, dy\\
   \aleq&\brac{\int_{B_R} \brac{\int_{B_R} \frac{|\langle n(\gamma(y)), \gamma(x)-\gamma(y)\rangle |}{|\gamma(x)-\gamma(y)|^{2+s}}\, dy}^{\frac{1}{s}}\, dx}^s.
 \end{split}
\]
Moreover, by \Cref{la:meanvalueest}, for any $p > 1$
\begin{equation*}
\mvint_{B_r(x_0)} \mvint_{B_r(x_0)}\mvint_{[x,y]} |\gamma'(y)-\gamma'(z)|^2\, dz\, dx\, dy \aleq \brac{\mvint_{B_r(x_0)} \mvint_{B_r(x_0)}|\gamma'(x)-\gamma'(y)|^{2p}\, dx\, dy}^{\frac{1}{p}}
\end{equation*}
Combining the above estimates we find
\begin{equation}\label{eq:ineq}
\begin{split}
 &\sup_{B_r(x_0) \subset B_1}\mvint_{B_r(x_0)} \mvint_{B_r(x_0)}|\gamma'(x)-\gamma'(y)|\, dx\, dy\\
 \aleq & \sup_{B_r(x_0) \subset B_1}\brac{\mvint_{B_r(x_0)} \mvint_{B_r(x_0)}|\gamma'(x)-\gamma'(y)|^{2p}\, dx\, dy}^{\frac{1}{p}} \\
 & + \brac{\int_{B_R} \brac{\int_{B_R} \frac{|\langle n(\gamma(y)), \gamma(x)-\gamma(y)\rangle |}{|\gamma(x)-\gamma(y)|^{2+s}}\, dy}^{\frac{1}{s}}\, dx}^s.
 \end{split}
\end{equation}

Applying \Cref{la:johnirenberglocal} we find
\begin{equation}\label{eq:ineq2}
\begin{split}
 &\sup_{B_r(x_0) \subset B_1}\mvint_{B_r(x_0)} \mvint_{B_r(x_0)}|\gamma'(x)-\gamma'(y)|\, dx\, dy\\
 \aleq& \sup_{B_r(x_0) \subset B_1}\brac{\mvint_{B_r(x_0)} \mvint_{B_r(x_0)}|\gamma'(x)-\gamma'(y)|\, dx\, dy}^{2} \\
 & + \brac{\int_{B_R} \brac{\int_{B_R} \frac{|\langle n(\gamma(y)), \gamma(x)-\gamma(y)\rangle |}{|\gamma(x)-\gamma(y)|^{2+s}}\, dy}^{\frac{1}{s}}\, dx}^s.
 \end{split}
\end{equation}

Now set
\[
 F(R) := \sup_{B_r(x_0) \subset B_R}\mvint_{B_r(x_0)} \mvint_{B_r(x_0)}|\gamma'(x)-\gamma'(y)|\, dx\, dy.
\]
Then \eqref{eq:ineq2} become
\begin{equation}\label{eq:inequ3}
 F(R) \leq C F(R)^2 + \brac{\int_{B_R} \brac{\int_{B_R} \frac{|\langle n(\gamma(y)), \gamma(x)-\gamma(y)\rangle |}{|\gamma(x)-\gamma(y)|^{2+s}}\, dy}^{\frac{1}{s}}\, dx}^s.
\end{equation}

As for properties of $F(R)$, it is clearly increasing function. Moreover, since $\gamma \in C^1$ we have
\[
 \lim_{R \to 0} F(R) = 0.
\]
Lastly, $R \mapsto F(R)$ is continuous: {Fix any $\eps > 0$ and $R > 0$. Since $\gamma'$ is continuous, there exists an $r_1$ such that we have $\mvint_{B_r(x)} \mvint_{B_r(x)}|\gamma'(x)-\gamma'(y)|\, dx\, dy < \frac{\eps}{4}$ whenever $r \in (0,r_1)$ and $x \in B_{2R}$. So if $|R-R'| \leq \frac{r_1}{100}$ and $B_r(x_0) \subset B_R$ with $r \geq r_1$ then we can find $\tilde{x}_0 \in B_{R'}$ and $\tilde{r}$ with $B_{\tilde{r}}(\tilde{x}_0) \subset B_{R'}$ such that 
\[
 |\tilde{x}_0-x_0| \aleq_{r_1} |R-R'|, \quad \text{and} \quad |r-\tilde{r}| \aleq_{r_1} |R-R'|.
\]
The same holds if $B_r(x_0) \subset B_{R'}$ and $r \geq r'$ then we find a ball $B_{\tilde{r}}(\tilde{x}_0) \subset B_R$ with the above estimates. So if we choose $|R-R'| \ll 1$, we have in these cases by continuity of $\gamma'$
\[
 \abs{\mvint_{B_r(x_0)} \mvint_{B_r(x_0)}|\gamma'(x)-\gamma'(y)|\, dx\, dy-\mvint_{B_{\tilde{r}}(\tilde{x}_0)} \mvint_{B_{\tilde{r}}(\tilde{x}_0)}|\gamma'(x)-\gamma'(y)|\, dx\, dy} \leq \frac{\eps}{4}.
\]
Thus,
\[
 |F(R)-F(R')| \leq \frac{\eps}{4} + \frac{\eps}{4}  + \frac{\eps}{4} < \eps.
\]
}

These properties combined with \eqref{eq:inequ3} and a a continuity argument imply that there is some $\delta>0$ (depending on $C$) such that if
\[
 \brac{\int_{B_R} \brac{\int_{B_R} \frac{|\langle n(\gamma(y)), \gamma(x)-\gamma(y)\rangle |}{|\gamma(x)-\gamma(y)|^{2+s}}\, dy}^{\frac{1}{s}}\, dx}^s < \delta
\]
then
\[
 F(R) \leq C_2 \brac{\int_{B_R} \brac{\int_{B_R} \frac{|\langle n(\gamma(y)), \gamma(x)-\gamma(y)\rangle |}{|\gamma(x)-\gamma(y)|^{2+s}}\, dy}^{\frac{1}{s}}\, dx}^s.
\]
This, combined with \Cref{la:johnirenberglocal}, implies \eqref{eq:goalVMO}.
\end{proof}

\subsection{Consequence of VMO-control: BiLipschitz-control and graph}\label{ConVMO}
\begin{lemma}[Bilipschitz control]
\label{la:bilipvmo}
For any $\delta > 0$ there exists $\epsilon_0> 0$ such that the following holds:

Assume that $\gamma\in Lip(B_{R_0},\R^2)$, $|\gamma'| \equiv 1$ a.e. and it satisfies
\begin{equation*}
 \sup_{B_r(x_0) \subset B_{R_0}} \mvint_{B_r(x_0)} \mvint_{B_r(x_0)}|\gamma'(x)-\gamma'(y)|\, dx\, dy < \epsilon_0
 \end{equation*}
then
\begin{equation*}
 1-\delta \leq \frac{\abs{\gamma(x)-\gamma(y)}}{|x-y|} \leq 1 \quad \mbox{for all } x,y \in B_{R_0}.
 \end{equation*}
\end{lemma}
\begin{proof}
We have
\[
\begin{split}
 |\gamma(y)-\gamma(x)|^2 - |x-y|^2 =& \int_x^y\int_x^y\brac{ \langle \gamma'(z_1),\gamma'(z_2) \rangle -1}\, dz_1\, dz_2\\
=& \frac{1}{2} \int_x^y\int_x^y\abs{\gamma'(z_1)-\gamma'(z_2)}^2\, dz_1\, dz_2.\\
 \end{split}
 \]
In view of \Cref{la:johnirenberglocal} and the assumption we then have
\[
 \abs{\frac{|\gamma(y)-\gamma(x)|^2}{|x-y|^2} - 1} \leq \frac{\epsilon_0}{2}.
\]
This readily implies the claim.
\end{proof}

\begin{lemma}\label{blatt:bilip}
Assume that $\gamma \in C^1(B_{R_0},\R^2)$, $\left|\gamma'\right|\equiv 1$ and $\gamma$ is a convex curve in the sense of \eqref{fcdghte}. Then, for each $\tau\in B_{\frac{R_0}{4}}$ there exists some $\delta\in (0,1)$ such that if
\begin{equation}\label{gb'kouno}
 1-\delta \leq \frac{\abs{\gamma(x)-\gamma(y)}}{|x-y|} \leq 1 \quad \mbox{for all}\quad x,y\in  B_{R_0},
\end{equation}
then
\begin{equation}\label{res.ook8}
 \gamma\left(B_{\frac{R_0}{4}}(\tau)\right) \quad \text{is a graph over $\left\lbrace t\gamma'(\tau)+\gamma(\tau)|\,t\in\R \right\rbrace$},
\end{equation}
which is either convex or concave.
\end{lemma}

\begin{proof}
The proof below is based on ideas by S. Blatt, \cite{SimonBlattgraph}.

First we prove that for every $\tau$ as in the statement $\gamma\left(B_{\frac{R_0}{4}}(\tau)\right)$ is a graph. To do so, we proceed by contradiction and  assume that there exists some $\tau\in B_{\frac{R_0}{4}}$ such that $\gamma\left(B_{\frac{R_0}{4}}(\tau)\right)$ is not a graph over $\left\lbrace t\gamma'(\tau)+\gamma(\tau)|\,t\in\R \right\rbrace$.

Now, up to translations and rotations of $\gamma$ we can assume that 
\begin{equation*}
\gamma'(\tau)=(1,0)^T\quad\mbox{and}\quad \gamma(\tau)=0.
\end{equation*}
Also, as consequence of this and Lemma~\ref{la:vxinconvexity}, we obtain that either $\gamma_2(t)\geq 0$ or $\gamma_2(t)\leq 0$ for all $t\in B_{R_0}$. Without loss of generality we assume that 
\begin{equation}\label{possecnor}
\gamma_2(t) \geq 0\quad\mbox{for all}\quad t\in B_{R_0}.   
\end{equation}

Since, by assumption, $\gamma$ is not a graph on $B(x_0,R)$ there must be some $x_0 \in B_{\frac{R_0}{4}}(\tau)$ such that $\gamma'(x_0) = (0,1)^T$ and $\gamma_1(x_0)>0$ (if $x_0 >\tau $) or $\gamma'(x_0) = (0,-1)^T$ and $\gamma_1(x_0)<0$ (if $x_0 < \tau$).

We only consider the case $x_0 > \tau$ and $\gamma'(x_0) = (0,1)^T$, the proof of the other case is analogous. Moreover, without loss of generality we assume that $\tau<0$.

Then, since $\gamma(\tau)=0$  and $|\gamma'|=1$ we obtain that
\begin{equation*}
\left|\gamma (t)\right|\leq\left|t-\tau\right|,
\end{equation*}
from which we evince that
\begin{equation*}
\gamma(B_{R_0}) \subset [\tau-R_0,R_0-\tau]^2.
\end{equation*}
Also, it follows from~\eqref{possecnor} that 
\begin{equation*}
\gamma(B_{R_0}) \subset [\tau-R_0,R_0-\tau] \times [0,R_0-\tau].
\end{equation*}

Now, if $\nu_{x_0}$ is given as in~\eqref{fcdghte}, then as a consequence of Lemma~\ref{la:vxinconvexity} we have that $\nu_{x_0}$ is uniquely defined and either $\nu_{x_0}=(1,0)^T$ or $\nu_{x_0}=(-1,0)^T$. Now, applying~\eqref{fcdghte} in $x_0$ we obtain
\begin{equation*}
0 \leq \langle \nu_{x_0}, \gamma(x_0)-\gamma(\tau)\rangle= \langle \nu_{x_0}, \gamma(x_0) \rangle,   
\end{equation*}
from which it follows that $\nu_{x_0}=n(\gamma(x_0))=(1,0)^T$. Making use of this and~\eqref{fcdghte} a second time we obtain for all $t\in B_{R_0}$ 
\begin{equation*}
\left\langle
\begin{pmatrix}
1\\
0
\end{pmatrix}, \gamma(x_0)-\gamma(t)\right\rangle\geq 0
\end{equation*}
which gives that 
\begin{equation}\label{fnuc32ll}
\gamma_1(t)\leq \gamma_1(x_0)\quad\mbox{for all}\quad t\in B_{R_0}.
\end{equation}
Furthermore, using $\left|\gamma'\right|\equiv 1$ and $\gamma(\tau)=0$, we evince 
\begin{equation*}
\gamma_1(x_0)<\frac{R_0}{4}.
\end{equation*}
From this and~\eqref{fnuc32ll} we obtain 
\begin{equation*}
\gamma(B_{R_0}) \subset \left[\tau-R_0, \frac{R_0}{4}\right] \times \left[0,R_0-\tau\right].
\end{equation*}
In particular, we can compute that
\begin{equation*}
\begin{split}
\left|\gamma(-R_0)-\gamma(R_0)\right| &\leq \diam \brac{\left[\tau-R_0,\frac{R_0}{4}\right] \times[0,R_0-\tau]}\\ 
&= \left|(\tau-R_0,0) - \left(\frac{R_0}{4},R_0-\tau\right)\right| \\
&= \left|\left(\tau-R_0-\frac{R_0}{4},\tau-R_0\right)\right| \\
&\leq \frac{7}{4}R_0
\end{split}
\end{equation*}
Thus, if we pick $\delta < 1-\frac{7}{8}$, we obtain that 
\[
 \frac{|\gamma(-R_0)-\gamma(R_0)|}{2R_0} \leq  \frac{7}{8}< 1-\delta.
\]
This contradicts~\eqref{gb'kouno} and we conclude that $\gamma\left(B_{\frac{R_0}{4}}(\tau)\right)$ is a graph over $\left\lbrace t\gamma'(\tau)+\gamma(\tau)|\,t\in\R \right\rbrace$. The convexity or concavity of such graph follows immediately from Proposition~\ref{convcurvfunc}. 

\end{proof}

\section{Sobolev control}\label{s:sobcontrol}
The goal of this section is show that locally if our curve is convex and the Willmore energy is small we obtain a Sobolev control.
\begin{theorem}\label{th:sobolevcontrol}
For any $s \in (0,1)$ there exists $\eps = \eps(s) > 0$ such that the following holds:

Assume that for $R > 0$ we have a curve $\gamma \in C^1(B_{20R},\R^2)$, $|\gamma'| \equiv 1$ satisfies (where $n(x) = \gamma'(x)^\perp$)
\begin{equation}\label{eq:pr:VMOcontrol:small}
  \int_{-20R}^{20R} \brac{\int_{-20R}^{20R} \frac{\langle n(\gamma(y)),\gamma(y)-\gamma(x)\rangle}{|\gamma(x)-\gamma(y)|^{2+s}} dy}^{\frac{1}{s}} dx< \eps.
\end{equation}
Assume also that $\gamma$ is convex in the sense of \Cref{def:convexcurve} in $B_{20 R}$.

Then
\begin{equation}\label{eq:goalreflcontrol}
  [\gamma']_{W_2^{s,\frac{1}{s}}(B_{R})} \aleq_s \int_{-10R}^{10R} \brac{\int_{-10R}^{10R} \frac{\langle n(\gamma(y)),\gamma(y)-\gamma(x)\rangle}{|\gamma(x)-\gamma(y)|^{2+s}} dy}^{\frac{1}{s}} dx.
  \end{equation}
\end{theorem}

Our first observation is that under the assumption that we are locally a graph  (or the weaker assumptions, that the derivatives $\gamma'$ do not make a full turn) we have a pointwise control of corresponding second order difference quotient.
\begin{lemma}\label{la:sobc:1}
Assume that $\gamma \in \lip(B_R,\R^2)$ with $|\gamma'|\equiv 1$ a.e. in $B_R$. If
\[
 \sup_{z_1,z_2 \in B_R} |\gamma'(z_1)-\gamma'(z_2)| = c < 2,
\]
then for almost every $x,y \in B_R$ we have 
\[
 \int_{[x,y]} |\langle n(\gamma(y)) , \gamma'(z) \rangle|\, dz \ageq_c |\gamma(x)-\gamma(y)-\gamma'(y)(x-y)|.
\]
\end{lemma}
\begin{proof}
We have by Pythagoras' theorem
\[
 |\gamma'(y)-\gamma'(z)|^2 = |n(\gamma(y))\cdot (\gamma'(y)-\gamma'(z))|^2 +|\gamma'(y)\cdot (\gamma'(y)-\gamma'(z))|^2.
\]
Now $|\gamma'| \equiv 1$ implies
\[
|\gamma'(y)\cdot (\gamma'(y)-\gamma'(z))|^2 = \frac{1}{4} |\gamma'(y)-\gamma'(z)|^4.
\]
Thus we have 
\[
\begin{split}
 |n(\gamma(y))\cdot (\gamma'(y)-\gamma'(z))|^2 =& |\gamma'(y)-\gamma'(z)|^2 -\frac{1}{4} |\gamma'(y)-\gamma'(z)|^4\\
 =& |\gamma'(y)-\gamma'(z)|^2 \brac{1-\frac{1}{4} |\gamma'(y)-\gamma'(z)|^2}\\
 \geq&|\gamma'(y)-\gamma'(z)|^2 \brac{1-\frac{1}{4} c^2}.\\
 \end{split}
\]
So we have (by the fundamental theorem)
\[
 \int_{[x,y]} |n(\gamma(y)) \cdot \gamma'(z)|\, dz \ageq_c \int_{[x,y]} |\gamma'(y)-\gamma'(z)| \ageq |\gamma(y)-\gamma(x)-\gamma'(y)(x-y)|.
\]

\end{proof}

On the other hand, we have also the following result.
\begin{lemma}\label{la:sobc:2}
Assume that $\gamma \in C^1(B_R)$ with $|\gamma'|\equiv 1$. Assume that $\gamma$ is the graph of a convex function
then for every  $x,y \in B_R$ we have 
\[
 \int_{[x,y]} |\langle n(\gamma(y)) , \gamma'(z) \rangle|\, dz = \langle n(\gamma(y)),\gamma(y)-\gamma(x)\rangle.
\]
\end{lemma}
\begin{proof}

Indeed, being $\gamma$ the graph of a convex function $f$, we have that 
\begin{equation*}
\langle n(\gamma(y)),\gamma'(z) \rangle=\frac{1}{\sqrt{1+f'(\gamma_1(x))^2}\sqrt{1+f'(\gamma_1(y))^2}}\left(f'(\gamma_1(y))-f'(\gamma_1(z))\right)\geq 0,
\end{equation*}
since $f'$ is non decreasing and $\gamma_1(y)>\gamma_1(z)$ for every $z\in (x,y)$.

That is, if $x \leq z \leq y$ then
\[
\langle n(\gamma(y)) , \gamma'(z) \rangle\geq 0
\]
and if $y \leq z \leq x$ then
\[
\langle n(\gamma(y)) , \gamma'(z) \rangle\leq 0
\]
In both cases the claim follows from the fundamental theorem of calculus.
\end{proof}

\begin{proof}[Proof of \Cref{th:sobolevcontrol}]
W.l.o.g. $R =1$. Taking $\eps$ small enough in \eqref{eq:pr:VMOcontrol:small} we obtain from \Cref{pr:VMOcontrol} combined with \Cref{blatt:bilip}, \Cref{la:bilipvmo} that $\gamma$ is a convex graph, and biLipschitz on $B_5$.
 
Thus we can apply \Cref{la:sobc:1} and \Cref{la:sobc:2} to obtain
\[
 |\gamma(x)-\gamma(y) -\gamma'(y)(x-y)| \aleq \langle n(\gamma(y),\gamma(y)-\gamma(x)\rangle.
\]

Integrating this, and using that $|\gamma(x)-\gamma(y)| \leq |x-y|$ we conclude

\begin{equation}\label{hiatus}
\begin{split}
&\int_{[-5,5]}\left(\int_{[-5,5]} \frac{\left|\gamma(x)-\gamma(y)-\gamma'(y)(x-y)\right|}{\left|x-y\right|^{2+s}}\,dy \right)^\frac{1}{s}\,dx\\
\aleq&\int_{[-5,5]}\left(\int_{[-5,5]}  \frac{\langle n(\gamma(y)), \gamma(y)-\gamma(x)\rangle}{\left|\gamma(x)-\gamma(y)\right|^{2+s}}\,dy  \right)^\frac{1}{s}\,dx 
\end{split}
\end{equation}

Now, pick the usual bump function $\eta \in C_c^\infty\left(B\left(0,3\right)\right)$, $\eta \equiv 1$ in $B\left(0,2\right)$. We set
\begin{equation*}
\tilde{\gamma}(x) := \eta(x)(\gamma(x)-q(x))+q(x) \equiv \eta(x) \bar{\gamma}(x) + q(x),
\end{equation*}
where $\bar{\gamma}:= \gamma- q$ and $q(x)=q_1+x\,q_2$ is an affine map which essentially corresponds to zero-th and first moment of $\gamma$. More precisely,
\begin{equation*}
q_1 := \mvint_{(-2,2)}\gamma(z)\,dz,\quad q_2:=\mvint_{(-2,2)} \gamma'(z)\,dz.
\end{equation*}
We notice that
\begin{equation*}
\mvint_{(-2,2)} q(x)\,dx=\mvint_{(-2,2)} \gamma(x)\,dx \quad \mbox{and } \mvint_{(-2,2)} q'(x)\,dx=\mvint_{(-2,2)} \gamma'(x)\,dx.
\end{equation*}

By the discrete product rule
\begin{equation*}
\begin{split}
&\tilde{\gamma}(x)-\tilde{\gamma}(y)-\tilde{\gamma}'(y)(x-y)\\
=&\eta(x)\bar{\gamma}(x) - \eta(y)\bar{\gamma}_{\mathcal{P}}(y) - (\eta\bar{\gamma})'(y)(x-y)\\
=&\eta(x)(\bar{\gamma}(x)-\bar{\gamma}(y)-\bar{\gamma}'(y)(x-y))\\
&+(\eta(x)- \eta(y)-\eta'(y) (x-y))\bar{\gamma}(y)\\
& +(\eta(x) - \eta(y))\, \bar{\gamma}'(y)(x-y).
\end{split}
\end{equation*}
Since $q$ is affine, we have
\begin{equation*}
\gamma(x)-\gamma(y)-\gamma'(y)(x-y)=\bar{\gamma}(x)-\bar{\gamma}(y)-\bar{\gamma}'(y)(x-y).
\end{equation*} 

With these observations, \eqref{hiatus} readily implies
\begin{equation*}
\begin{split}
& \int_{[-5,5]}\left( \int_{[-5,5]}\frac{\eta(x) \abs{\bar{\gamma}(x)-\bar{\gamma}(y)-\bar{\gamma}'(y)(x-y)}}{\left|x-y\right|^{2+s}}\,dy\right)^\frac{1}{s}\,dx\\
\aleq &  \int_{[-5,5]}\left(\int_{[-5,5]}  \frac{\langle n(\gamma(y)), \gamma(y)-\gamma(x)\rangle}{\left|\gamma(x)-\gamma(y)\right|^{2+s}}\,dy  \right)^\frac{1}{s}\,dx .
\end{split}
\end{equation*}

Also, we have for $x \in [-5,5]$, for any $p \in (1,\infty)$ with $p' < \frac{1}{s}$, using H\"older's inequality,
\begin{equation*}
\begin{split}
\int_{[-5,5]} \frac{\abs{\eta(x) - \eta(y)}\, \abs{\bar{\gamma}'(y)}}{|x-y|^{1+s}}\, dy \aleq &\int_{[-5,5]} \abs{\gamma'(y)-q'(y)} |x-y|^{-s}\, dy \\
\aleq & \brac{\int_{[-5,5]} \abs{\gamma'(y)-q'(y)}^{p}\,dy}^{\frac{1}{p}}.
\end{split}
\end{equation*}
Thus, we obtain that
\[
\begin{split}
  &\brac{\int_{[-5,5]} \brac{\int_{[-5,5]} \frac{\abs{(\eta(x) - \eta(y))\, \bar{\gamma}'(y)(x-y)}}{|x-y|^{2+s}}\, dy}^{\frac{1}{s}} dx}^s\\
 \aleq &\brac{\int_{[-5,5]} \int_{[-5,5]} |\gamma'(x)-\gamma'(y)|^{p}\, dx\, dy}^{\frac{1}{p}}.
  \end{split}
\]

Finally, we have for $x \in [-5,5]$, for any $p \in (1,\infty)$ with $p' < \frac{1}{s}$ (observe  that ${(\bar{\gamma}(y))}_{[-5,5]}=0$)
\begin{equation*}
\begin{split}
&\left|\int_{[-5,5]} \frac{(\eta(x)- \eta(y)-\eta'(y) (x-y))\bar{\gamma}(y)}{|x-y|^{2+s}}\, dy\right|\\
\aleq &\int_{[-5,5]} |x-y|^{-s} \abs{\bar{\gamma}(y)} dy\\
= & \int_{[-5,5]} |x-y|^{-s} \abs{\bar{\gamma}(y)-{(\bar{\gamma}(y))}_{[-5,5]}} dy\\ 
\leq  & \int_{[-5,5]}\int_{[-5,5]} |x-y|^{-s} \abs{\bar{\gamma}(y)-\bar{\gamma}(z)}\,dz\,dy\\ 
\leq & \int_{[-5,5]}\int_{[-5,5]} \int_{[y,z]} |x-y|^{-s} \abs{\bar{\gamma}'(z_2)}\,dz_2\,dz\,dy\\ 
=& \int_{[-5,5]}\int_{[-5,5]} \int_{[y,z]} |x-y|^{-s} \abs{\bar{\gamma}'(z_2)-(\bar{\gamma}')_{[-5,5]}}\,dz_2\,dz\,dy\\ 
\leq & \int_{[-5,5]}\int_{[-5,5]} \int_{[-5,5]}\int_{[-5,5]} |x-y|^{-s} \abs{\gamma'(z_2)-\gamma'(z_3)}\,dz_3\,dz_2\,dz\,dy\\
\aleq &  \int_{[-5,5]}\int_{[-5,5]} \abs{\gamma'(z_2)-\gamma'(z_3)}\,dz_3\,dz_2.
\end{split}
\end{equation*}
So that
\begin{equation*}
\begin{split}
&\brac{\int_{[-5,5]} \abs{\int_{[-5,5]} \frac{(\eta(x)- \eta(y)-\eta'(y) (x-y))\bar{\gamma}(y)}{|x-y|^{2+s}}\, dy}^{\frac{1}{s}} dx}^s\\
\aleq &\brac{\int_{[-5,5]} \int_{[-5,5]} |\gamma'(x)-\gamma'(y)|^{p}\, dx\, dy}^{\frac{1}{p}}.
\end{split}
\end{equation*}

In conclusion, using also \eqref{eq:goalVMO},
\[
\begin{split}
& \int_{[-5,5]}\left( \int_{[-5,5]}\frac{\abs{\tilde{\gamma}(x)-\tilde{\gamma}(y)-\tilde{\gamma}'(y)(x-y)}}{\left|x-y\right|^{2+s}}\,dy\right)^\frac{1}{s}\,dx\\
\aleq &  \int_{[-5,5]}\left(\int_{[-5,5]}  \frac{\langle n(\gamma(y)), \gamma(y)-\gamma(x)\rangle}{\left|\gamma(x)-\gamma(y)\right|^{2+s}}\,dy  \right)^\frac{1}{s}\,dx. 
\end{split}
 \]
The claim now follows from the following proposition, \Cref{pr:localsobcontrol} and again  \eqref{eq:goalVMO}.
\end{proof}

\begin{proposition}\label{pr:localsobcontrol}
Let $\gamma: [-4,4] \to \R^2$ and set for $\eta \in C_c^\infty([-3,3])$ $\eta \equiv 1$ in $[-2,2]$,
\[
 \tilde{\gamma}(x) := \eta(x)(\gamma(x)-q(x))+q(x) \equiv \eta(x) \bar{\gamma}(x) + q(x),
\]
where $\bar{\gamma}:= \gamma- q$ and $q(x)=q_1+x\,q_2$ is an affine map which essentially corresponds to zero-th and first moment of $\gamma$. More precisely,
\begin{equation*}
q_1 := \mvint_{(-2,2)}\gamma(z)\,dz,\quad q_2:=\mvint_{(-2,2)} \gamma'(z)\,dz.
\end{equation*}

Assume that 
\[
 \Gamma := \int_{(-4,4)} \brac{\int_{(-4,4)} \frac{|\tilde{\gamma}(x) -\tilde{\gamma}(y) -\tilde{\gamma}'(y)(x-y)|}{|x-y|^{2+s}}\, dy}^{\frac{1}{s}} dx  < \infty.
\]
Then for some $q \geq 1$
\[
 [\gamma']_{W^{s,\frac{1}{s}}_2((-1,1))} \aleq \Gamma^s + \brac{\int_{(-4,4)}\int_{(-4,4)} |\gamma'(x)-\gamma'(y)|^q\, dx\, dy}^{\frac{1}{q}}
\]
\end{proposition}
\begin{proof}

We recall, thanks to Poincar\'{e} inequality, for $1\leq p <+\infty$ we deduce 
\begin{equation}\label{jnol012-}
\left\|\bar{\gamma}\right\|_{L^p(B(0,2))} =\left\|\gamma-q\right\|_{L^p(B(0,2))}\aleq  \left\|\gamma'-q'\right\|_{L^p(B(0,2))}.
\end{equation}
Therefore, we obtain 
\begin{equation}\label{eq:normalizedgammasob}
\begin{split}
\|\tilde{\gamma}'-q'\|_{L^p(B(0,2))} &\leq \left\|\eta'\left(\gamma-q\right)\right\|_{L^p(B(0,2))}+\left\|\eta(\gamma'-q')\right\|_{L^p(B(0,2))}\\
&\aleq  \left\|\gamma-q\right\|_{L^p(B(0,2))}+\left\|\gamma'-q'\right\|_{L^p(B(0,2))} \\
&\aleq  \left\|\gamma'-q'\right\|_{L^p(B(0,2))}\\
& \aleq \left(\int_{(-2,2)}\int_{(-2,2)}  \left|\gamma'(z)-\gamma'(y)\right|^p\,dz\,dy\right)^{\frac{1}{p}}.
\end{split}
\end{equation}

Thanks to the support of $\eta$ (and the fact that $q$ is affine) we have 
\[
\begin{split}
 &\int_{(-4,4)} \brac{\int_{\R} \frac{|\tilde{\gamma}(x) -\tilde{\gamma}(y) -\tilde{\gamma}'(y)(x-y)|}{|x-y|^{2+s}}\, dy}^{\frac{1}{s}} dx\\
 \aleq&\int_{(-4,4)} \brac{\int_{(-4,4)} \frac{|\tilde{\gamma}(x) -\tilde{\gamma}(y) -\tilde{\gamma}'(y)(x-y)|}{|x-y|^{2+s}}\, dy}^{\frac{1}{s}} dx  \\
 &+\int_{(-{3},{3})} \brac{\int_{\R \setminus (-4,4)} \frac{|\tilde{\gamma}(x) -\tilde{\gamma}(y) -\tilde{\gamma}'(y)(x-y)|}{|x-y|^{2+s}}\, dy}^{\frac{1}{s}} dx.  \\
 \end{split}
\]
We see that 
\[
\begin{split}
 &\int_{(-{3},{3})} \brac{\int_{\R \setminus (-4,4)} \frac{|\tilde{\gamma}(x) -\tilde{\gamma}(y) -\tilde{\gamma}'(y)(x-y)|}{|x-y|^{2+s}}\, dy}^{\frac{1}{s}} dx \\
 =&\int_{(-{3},{3})} \brac{\int_{\R \setminus (-4,4)} \frac{|\eta(x) \bar{\gamma}(x)|}{|x-y|^{2+s}}\, dy}^{\frac{1}{s}} dx \\
 =&\int_{(-{3},{3})} |\bar{\gamma}(x)|^{\frac{1}{s}}\, dx\\
 \overset{\eqref{jnol012-}}{\aleq}& \int_{(-4,4)}\int_{(-4,4)} |\gamma'(x)-\gamma'(y)|^{\frac{1}{s}}\, dx\, dy.
 \end{split}
\]
Now we have by the representation formula of the fractional Laplacian, \Cref{Toperator}
\[
 \|\laps{s+1} \brac{\eta \bar{\gamma}} \|_{L^{\frac{1}{s}}(-4,4)} \aleq \brac{\int_{(-4,4)} \brac{\int_{\R} \frac{|\tilde{\gamma}(x) -\tilde{\gamma}(y) -\tilde{\gamma}'(y)(x-y)|}{|x-y|^{2+s}}\, dy}^{\frac{1}{s}} dx}^s.
\]
By Stein's theorem, \Cref{th:steinident},
\[
 [\eta \bar{\gamma}]_{W^{s,\frac{1}{s}}_2(\R)} \aleq \|\laps{s+1} \brac{\eta \bar{\gamma}} \|_{L^{\frac{1}{s}}(-4,4)} + \|\laps{s+1} \brac{\eta \bar{\gamma}} \|_{L^{\frac{1}{s}}(\R \setminus (-4,4))}.
\]
Using the distance between $\supp \eta$ and $\R \setminus (-4,4)$, see, e.g., \cite[Lemma A.1]{BRS16}, we find
\[
 \|\laps{s+1} \brac{\eta \bar{\gamma}} \|_{L^{\frac{1}{s}}(\R \setminus (-4,4))} \aleq \|\eta \bar{\gamma}\|_{L^1(\R)} \overset{\eqref{jnol012-}}{\aleq} \int_{(-4,4)}\int_{(-4,4)} |\gamma'(x)-\gamma'(y)|^{\frac{1}{s}}\, dx\, dy.
\]
Thus, we have shown 
\[
 [\bar{\gamma}]_{W^{s,\frac{1}{s}}_2((-1,1))} \leq [\eta \bar{\gamma}]_{W^{s,\frac{1}{s}}_2({\R})} \aleq \Gamma^s + \brac{\int_{(-4,4)} |\gamma'(x)-\gamma'(y)|^{\frac{1}{s}}\, dx\, dy}^s.\qedhere
\]

\end{proof}

\section{Self-repulsiveness of Willmore energy for convex curves}\label{SElfRep}
Strongly inspired by results of Strzelecki-von der Mosel~\cite{SvdM12}, as adapted to the tangent point energy \cite[Theorem 4.9.]{BRSV21} we obtain the following statement, which states that points with locally small energy do not collide with any other points.

\begin{theorem}\label{th:smallenergyinjectivity}
There exists $\delta > 0$ such that the following holds.

Let $\gamma \in C^1(\S^1,\R^2)$ be a homeomorphism, $|\gamma'| \equiv 1$.
Also assume
\begin{equation}\label{dcsdaohbedfo6}
\int_{\S^1} \brac{\int_{B_{\rho}(x_0)} \frac{ \left|\langle n(\gamma(y)), \gamma(y)-\gamma(x) \rangle\right|}{|\gamma(x)-\gamma(y)|^{2+s}} \, dy}^{\frac{1}{s}}\, dx < \delta.
\end{equation}
If for any $z_0 \in \S^1$ we have
\[
 |\gamma(x_0)-\gamma(z_0)| < \frac{1}{10}\rho,
\]
then there exists $\bar{x} \in B_{\rho}(x_0)$ such that $\gamma(\bar{x}) = \gamma(z_0)$. In particular, we have $z_0\in B_\rho(x_0)$.
\end{theorem}

This result is a consequence of the following Lemma (see \cite[Lemma 4.8]{BRSV21} which is an adaptation of \cite[Lemma 2.1, Lemma 2.3]{SvdM12}).

We omit the proofs which are almost verbatim to \cite[Lemma 4.8]{BRSV21} and \cite[Theorem 4.9]{BRSV21}.

\begin{lemma}[Strzelecki-von der Mosel] \label{S-VdM}
Let $s\in (0,1)$. For any $\epsilon>0$ there exists $\delta$ such that the following holds. 

Let $\gamma\in {Lip}(\S^1,\R^2)$, $\left|\gamma'\right|\equiv 1$, and assume that for some $x_0\in \S^1$ and $\rho>0$ we have 
\begin{equation}\label{frt01se5t639}
\int_{\S^1}\brac{ \int_{B_\rho(x_0)}\frac{\left|\langle n(\gamma(y)), \gamma(y)-\gamma(x)\rangle\right|}{\left|\gamma(y)-\gamma(x)\right|^{2+s}}\,dy}^{\frac{1}{s}} \,dx<\delta.
\end{equation}
Moreover, assume that there is $y_0\in \S^1$ with $d:=\left|\gamma(x_0)-\gamma(y_0)\right|\leq \rho$. 

Then
\begin{equation}\label{wwhtp}
\gamma(\S^1)\cap B_{2d}(\gamma(x_0)) \subset  L_{\epsilon  d}(\gamma(x_0),\gamma(y_0)),
\end{equation} 
where 
\begin{equation*}
\begin{split}
&L(\gamma(x_0),\gamma(y_0)):=\left\lbrace (1-t)\gamma(x_0)+t\gamma(y_0)|\,t\in [0,1]\right\rbrace\quad \mbox{and }\\ 
&L_{\epsilon d}(\gamma(x_0),\gamma(y_0)):=\left\lbrace p\in \R^2|\, dist(p, L(\gamma(x_0),\gamma(y_0)))<\epsilon d \right\rbrace.
\end{split}
\end{equation*}
\end{lemma} 


As a consequence, see \cite[Remark after Lemma 2.3 and Theorem 1.4]{SvdM12} we have in particular obtain
\begin{proposition}[Strzelecki-von der Mosel] \label{topolmfd}
Let $s\in (0,1)$. Assume $\gamma\in {Lip}(\S^1,\R^2)$, $\left|\gamma'\right|\equiv 1$ a.e. in $\S^1$.

If 
\[
\int_{\S^1}\brac{\int_{\S^1}\frac{\left|\langle n(\gamma(y)), \gamma(y)-\gamma(x)\rangle\right|}{\left|\gamma(y)-\gamma(x)\right|^{2+s}}\,dy}^{\frac{1}{s}} \,dx<+\infty
\]
and if $\gamma$ is convex in the sense of \Cref{def:convexcurve} and has finite Willmore energy, then 
$\gamma(\S^1)$ is a topological one-manifold.
\end{proposition}

\section{Convex curves with finite energy are continuously differentiable}
\label{gfdcvret56}
In this section we prove that all \emph{convex} Lipschitz curves with finite $\mathscr{W}_{s,p}$-energy are actually $C^1$ if $s\geq \frac{1}{p}$ -- a fact that is interesting on its own. We also believe it holds in higher dimensions, but we will not pursue this here any further.

Precisely we have
\begin{theorem}\label{n77754fg} 
Assume $\gamma$ belongs to the set
\[
\gamma \in \tilde{X} = \left\lbrace \gamma\in \lip(\S^1,\R^2)\Bigg|\, \begin{aligned}
&\mbox{$\gamma$ is a homeomorphism}\\ 
&\left|\gamma'\right|\equiv 1 \mbox{  a.e. in  }\S^1\\
&\mbox{$\gamma$ is convex in $\S^1$}\\
\end{aligned} \right\rbrace.
\]
If for some $s \in (0,1)$ and $p \geq \frac{1}{s}$ we have finite energy,
\begin{equation}\label{dcsfght}
\mathscr{W}_{s,p}(\gamma)< +\infty,
\end{equation}
then $\gamma \in C^1$.
\end{theorem}
Let us stress that convexity is absolutely crucial in the assumption of \Cref{n77754fg} -- as evinced by \Cref{la:energyspace}.

Of course, \Cref{n77754fg} can be formulated for sets:
\begin{theorem}
Let $E$ be a bounded and convex set with nonempty interior. If for some $s\in (0,1)$ and some  $p\geq \frac{1}{s}$ it holds  
\begin{equation*}
\hat{\mathscr{W}}_{s,p}(E)<+\infty.
\end{equation*}
then $\partial E$ is  $C^1$.
\end{theorem}
  
In the proof below observe that the $C^1$-control is purely qualtitative, and we get no estimate whatsoever. This means in particular that when we consider minimizing sequences, we do not seem to get any uniform $C^1$-control.
  
It is also worth pointing out that this result is false in the supercritical regime (namely $p<1/s$): then it is possible to construct an example of a homeomorphism $\gamma\in Lip(\S^1,\R^2)$ that is a parametrization of the boundary of a convex set, which has finite $\mathscr{W}_{s,p}$ energy, but it is not $C^1$. Cf. \Cref{supercrtitcal}.

To prove \Cref{n77754fg} we will first need the following preliminary results. 

\begin{lemma}\label{hbcerdlot}
Let $\gamma\in Lip(\S^1,\R^2)$ be a homeomorphism satisfying the convexity assumption~\eqref{fcdghte} in $\S^1$  and such that $|\gamma'|\equiv 1$ a.e. in $\S^1$. Then, if $\gamma\notin C^1(\S^1,\R^2)$ there exists some $x_0\in\S^1$ and $v\neq w \in \S^1$ such that 
\begin{equation*}
\langle v, \gamma(x_0)-\gamma(x)\rangle\geq 0\quad \mbox{and}\quad \langle w,\gamma(x_0)-\gamma(x)\rangle\geq 0 \quad \forall x\in\S^1.
\end{equation*} 
\end{lemma} 

\begin{proof}
Assume to the contrary that for all $x_0\in\S^1$ 
\begin{equation}\label{gvfloiny}
\mbox{there exists a unique } \nu_{x_0}\in\S^1 \text{ such that } \langle \nu_{x_0},\gamma(x_0)-\gamma(x)\rangle\geq 0
\end{equation}
for all $x\in\S^1$.

Set $v(x_0) := \nu_{x_0}$. We claim that
\begin{equation}\label{cponti}
v\in\ C(\S^1,\S^1).
\end{equation}
If this was not the case, by compactness of the target $\S^1$, there must exists some $v_1 \neq v_2 \in \S^1$, a point $t_0\in\S^1$, and two sequences $t_n, \tilde{t}_n\in \S^1$ such that $t_n,\tilde{t}_n\to t_0$, but
\begin{equation*}
\lim_{n}v(t_n)=v_1\neq v_2=\lim_{n} v(\tilde{t}_n).
\end{equation*} 
According to~\eqref{gvfloiny}
\begin{equation*}
\langle v_{t_n}, \gamma(t_n)-\gamma(y)\rangle \geq 0\quad\mbox{and}\quad \langle v_{\tilde{t}_n},\gamma(\tilde{t}_n)-\gamma(y)\rangle\geq 0\quad\mbox{for all }y\in\S^1.
\end{equation*} 
Taking the limit for $n\to+\infty$ and using continuity of $\gamma$, we find
\begin{equation*}
\langle v_1,\gamma(t_0)-\gamma(y)\rangle\geq 0\quad\mbox{and}\quad \langle v_2,\gamma(t_0)-\gamma(y)\rangle\geq 0\quad\mbox{for all } y\in\S^1,
\end{equation*}
with $v_1\neq v_2$. This is in contradiction to the uniqueness of $\nu_{t_0}$ in \eqref{gvfloiny}.

That is, if \eqref{gvfloiny} was true, then \eqref{cponti} is true. However, up to replacing $\gamma(\cdot)$ with $\gamma(-\cdot)$, for a.e. $x\in\S^1$ we have that
\begin{equation*}
\gamma'(x)^\perp=\nu(x),
\end{equation*} 
see Corollary~\ref{uniqnorma}. This means that $\gamma'$ coincides with a continuous function a.e. and therefore, it is continuous -- if \eqref{gvfloiny} was true.

Thus, if $\gamma \not \in C^1$, \eqref{gvfloiny} cannot be true, and since by \eqref{gvfloiny} there exists at least one $v$, the logical negation of \eqref{gvfloiny}  implies that there must be at least two point $v,w \in \S^1$, and we can conclude.
\end{proof}

\begin{proposition}\label{nscefg}
Let $\gamma\in Lip(\S^1,\R^2)$ be a homeomorphism satisfying the convexity assumption~\eqref{fcdghte} in $\S^1$, and such that $\left|\gamma'\right|\equiv 1$ a.e. in $\S^1$. Also, assume that there exists $x_0\in \S^1$ and two unit vectors $v,w\in \S^1$ satisfying
\begin{equation}\label{dckshfgtr01}
v \neq w, \quad \langle v,\gamma(x_0)-\gamma(y)\rangle\geq 0, \quad\mbox{and}\quad \langle w,\gamma(x_0)-\gamma(y)\rangle \geq 0,  \quad \forall y \in \S^1.
\end{equation} 
Then, if we define the convex hull of $\gamma(\S^1)$, $E:=\mbox{co}(\gamma(\S^1))$, for every  $p\geq \frac{1}{s}$ it holds
\begin{equation}\label{gvdcrey6543}
\lim_{\varepsilon\to 0^+} \int_{\S^1\setminus B_{\varepsilon}(x_0)} \left|H_{\partial E}^s(\gamma(x))\right|^p\,dx=+\infty.
\end{equation}
\end{proposition}

\begin{proof}[Proof of \Cref{nscefg}]
Thanks to Theorem~\ref{ImFuTh} we have that there exists some neighbourhood $U_{x_0}\subset\S^1$ of $x_0$ and a convex function $f$ such that up to rotations
\begin{equation}\label{handf}
\gamma(t)=(\gamma_1(t),f(\gamma_1(t)))^T\quad\mbox{for all}\quad t\in U_{x_0}.
\end{equation}

We now find now two vectors $v'$ and $w'$ in $\S^1$ that still satisfy \eqref{dckshfgtr01} with better properties.
For two sequences $x_n, \tilde{x}_n\in \gamma_1(U_{x_0})$ such that $x_n\searrow \gamma_1(x_0)$,  $\tilde{x}_n\nearrow \gamma_1(x_0)$ and $f$ is differentiable in $x_n$ and $\tilde{x}_n$ we define 
\begin{equation*}
\begin{split}
v'&:=\lim_{n}\frac{1}{\sqrt{1+(f'(x_n))^2}}  \begin{pmatrix}
f'(x_n)\\
-1
\end{pmatrix},\\
w'&:= \lim_{n}\frac{1}{\sqrt{1+(f'(\tilde{x}_n))^2}}  \begin{pmatrix}
f'(\tilde{x}_n)\\
-1
\end{pmatrix}.
\end{split}
\end{equation*}
The convergence of the two limits above is guaranteed by the monotonicity of $f'$ wherever it exists, and independent of the particular sequences chosen.

As a consequence of Theorem~\ref{ImFuTh}, $v'$ and $w'$ almost satisfy \eqref{dckshfgtr01} -- however it could be that $v'=w'$. But indeed, if we had $v'=w'$, then it follows that $\lim_{n}f'(x_n)=\lim_{n}f'(\tilde{x}_n)$. In this case by convexity of $f$ we deduce for every $n$ that 
\begin{equation*}
f'(\tilde{x}_n)\leq \frac{f(\tilde{x}_n)-f(\gamma_1(x_0))}{\tilde{x}_n-\gamma_1(x_0)}\leq \frac{f(x_n)-f(\gamma_1(x_0))}{x_n-\gamma_1(x_0)}\leq f'(x_n).
\end{equation*}
Taking the limits, we obtain that $f$ is differentiable in $\gamma_1(x_0)$. As a consequence of this and Theorem~\ref{ImFuTh} we deduce that $\gamma$ is differentiable in $x_0$. This and \Cref{la:vxinconvexity} imply that there is a unique unit vector $v$ satisfying~\eqref{fcdghte} in $x_0$ -- but this is contradiction to \eqref{dckshfgtr01}. So indeed $v'\neq w'$. So, w.l.o.g. $v=v'$ and $w=w'$.

Without loss of generality we can assume that for some $\theta_0\in (\pi,2\pi)$ it holds $w=R_{\theta_0} v$, where $R_{\theta_0}$ denotes the counterclockwise rotation by $\theta_0$. The case $\theta_0\in (0,\pi)$ follows from the first one by renaming $v$ and $w$. Also, note that $\theta_0=\pi$ is not allowed since equation~\eqref{dckshfgtr01} would give that $\gamma(\S^1)$ is a line, contradicting the hypothesis that $\gamma$ is a homeomorphism.  


\underline{Assume first $\theta_0\in \left(\frac{3}{2}\pi,2\pi\right)$}, then up to a translation, rotation and by possibly choosing $U_{x_0}$ smaller, we obtain 
 for some $m>0$ 
\begin{equation}\label{TnLrPgDr}
\gamma(x_0)=0,\quad v=(0,-1)^T\quad\mbox{and}\quad w=\frac{1}{\sqrt{1+m^2}}(-m,-1)^T. 
\end{equation}
Note that as a consequence of~\eqref{TnLrPgDr} and~\eqref{dckshfgtr01} we have that 
\begin{equation}\label{Sessagesimale}
\gamma_2(y)\geq 0\quad\mbox{for all}\quad y\in\S^1.
\end{equation} 

Now, we claim that for a.e. $x\in V_{x_0}:=U_{x_0}\cap \left\lbrace x > x_0 \right\rbrace$ there exists some convex set $E_{x}\subset \R^2$, such that 
\begin{equation}\label{STEP1}
E\subset E_x\quad\mbox{and}\quad  \gamma(x) \in \partial E\cap \partial E_x,
\end{equation}
and furthermore, $H_{\partial E_x}^s$ satisfies
\begin{equation}\label{STEP2}
H_{\partial E_x}^s(\gamma(x))\gtrsim_{m,s} \frac{1}{\left|x-x_0\right|^s}.
\end{equation}
If claim~\eqref{STEP1} and~\eqref{STEP2} are true, then we can apply the Maximum Principle for the nonlocal mean curvature, Proposition~\ref{milan}, to obtain for a.e. $x\in V_{x_0}$ that
\begin{equation*}
H_{\partial E}^s(\gamma(x))\geq H_{\partial E_x}^s(\gamma(x))\gtrsim_{m,s} \frac{1}{\left|x-x_0\right|^s}.
\end{equation*}
In particular, from this last equation we evince that 
\begin{equation*}
\begin{split}
\int_{U_{x_0}\setminus B_{\epsilon}(x_0)}\left|H_{\partial E}^s(\gamma(x))\right|^p\,dx \gtrsim_{m,s} \int_{V_{x_0}\setminus B_\epsilon(x_0)} \frac{dx}{\left|x-x_0\right|^{ps}}.
\end{split}
\end{equation*}
Recalling that $p\geq\frac{1}{s}$, and taking the limit for $\epsilon\to 0^+$ we obtain~\eqref{gvdcrey6543}.

\textbf{Construction of the set $E_x$}:\\
To find a set $E_x$ with~\eqref{STEP1} and~\eqref{STEP2} we recall that for every $x\in V_{x_0}$ there exists some $\nu_x\in \S^1$ such that 
\begin{equation}\label{fdcbhuop}
\langle \nu_x, \gamma(x)-\gamma(y)\rangle \geq 0\quad\mbox{for all}\quad y\in \S^1. 
\end{equation}
Moreover, since $\gamma$ is Lipschitz, the set 
\begin{equation*}
S:=\left \{ x\in V_{x_0}|\, \gamma \text{ is differentiable in $x$} \right \}   
\end{equation*} 
is such that  $\sigma^1\left(V_{x_0}\setminus S\right)=0$. 
From \Cref{la:vxinconvexity}, we deduce that for $x\in S$, the normal $\nu_x\in \S^1$ is uniquely determined.
We adhere to the notation $\nu_x=({\nu_x}_1,{\nu_x}_2)^T$. Also, thanks to \Cref{ImFuTh} we have for every $x\in S$  
\begin{equation*}
\nu_x=\frac{1}{\sqrt{1+f'(\gamma_1(x))}}\begin{pmatrix}
f'(\gamma_1(x))\\
-1
\end{pmatrix}.
\end{equation*} 
Hence, under the assumptions provided in~\eqref{TnLrPgDr}, it follows from this last equation that
\begin{equation}\label{trdghui87}
-1\leq {\nu_x}_2 < 0\quad  \mbox{and}\quad {\nu_x}_1\geq 0 \quad \mbox{for all}\quad   x\in S.
\end{equation}

Let $x\in S$. Now, we begin by setting the non negative gradient
\begin{equation*}\label{mi}
m_1:=\frac{-{\nu_x}_1}{{\nu_x}_2}
\end{equation*}
and  
\begin{equation}\label{tildet}
\tilde{t}:=\begin{dcases}
\frac{-\gamma_2(x)+m_1\gamma_1(x)}{m_1}\quad &\mbox{if}\quad m_1>0\\   
1\quad &\mbox{if}\quad m_1=0.
\end{dcases}
\end{equation}
Then, we define the function
\begin{equation}\label{functionfx}
g_x(t):=\begin{dcases}
-mt\quad &\mbox{for}\quad t\in (-\infty,0)\\
0\quad\mbox{for}\quad &\mbox{for}\quad t\in [0,\tilde{t})\\
m_1(t-\tilde{t})\quad &\mbox{for}\quad t\in [\tilde{t},+\infty).
\end{dcases}
\end{equation}
and its  epigraph
\begin{equation*}
E_x:=\left\lbrace (y_1,y_2)\in \R^2|\, y_2\geq g_x(y_1)\right\rbrace.
\end{equation*}
Since $g_x$ is convex, we obtain that  
\begin{equation}\label{convexityred}
E_x\quad\mbox{is convex}.
\end{equation}

\textbf{Proof of~\eqref{STEP1}}:\\
From the fact that
\begin{equation}\label{savoiardo}
g_x(\gamma_1(x))=\gamma_2(x),
\end{equation}
and Proposition~\ref{colgate} it follows that 
\begin{equation}\label{naroba}
\gamma(x)\in\partial E_x\cap \partial E.
\end{equation}
According to equation~\eqref{dckshfgtr01},~\eqref{TnLrPgDr},~\eqref{fdcbhuop} and the fact that $E=\mbox{co}(\gamma(\S^1))$ we have that 
\begin{equation*}
\begin{dcases}
\langle (m,1)^T, P\rangle \geq 0 \quad&\mbox{for all}\quad  P=(P_1,P_2)^T\in E\\
\langle (0,1)^T, P\rangle\geq 0 \quad&\mbox{for all}\quad  P=(P_1,P_2)^T\in E\\
\langle ({\nu_x}_1,{\nu_x}_2)^T, \gamma(x)-P\rangle\geq 0\quad \quad&\mbox{for all} \quad P=(P_1,P_2)^T\in E.
\end{dcases}
\end{equation*}
Writing explicitly the above set of equations we obtain 
\begin{equation}\label{giromondo}
\begin{dcases}
P_2\geq -m P_1  \quad&\mbox{for all}\quad  P=(P_1,P_2)^T\in E\\
P_2\geq 0   \quad&\mbox{for all}\quad  P=(P_1,P_2)^T\in E\\
P_2\geq m_1(P_1-\gamma_1(x))+\gamma_2(x)\quad&\mbox{for all}\quad  P=(P_1,P_2)^T\in E.\\
\end{dcases}
\end{equation}
The set of inequalities in~\eqref{giromondo} guarantees that $g_x(P_1)\leq P_2$ for each $P\in E$, and therefore
\begin{equation}\label{modulare}
E\subset E_x.
\end{equation}
Equations~\eqref{convexityred},~\eqref{naroba} and~\eqref{modulare} yield to~\eqref{STEP1}. 

\textbf{Proof of~\eqref{STEP2}}:\\
To prove~\eqref{STEP2} it is enough to apply Proposition~\ref{lm:barrierestimates} in $x$ with $g_x=g$ and $G=E_x$. By doing so we obtain that 
\begin{equation*}
\begin{split}
 H_{\partial E_x}^s(\gamma(x)) &\gtrsim_{m,s} \frac{1}{\left(\left|g_x(\gamma_1(x))\right|+\left|\gamma_1(x)\right|\right)^s}\\   
&=\frac{1}{\left(\left|\gamma_2(x)\right|+\left|\gamma_1(x)\right|\right)^s}\\
&\geq \frac{1}{2^s \left|x\right|^s}.
\end{split}
\end{equation*}
This concludes the proof of~\eqref{STEP2}. And in turn the case of $\theta_0 \in \left(\frac{3\pi}{2},2\pi\right)$ is concluded.

Now, we discuss \underline{the case $\theta_0\in \left(\pi,\frac{3\pi}{2}\right)$}. In this case, up to translations and rotations of $\gamma(\S^1)$, we can assume that for some $m>0$ it holds that 
\begin{equation}\label{confo54}
\gamma(x_0)=0,\quad v=(0,-1)^T\mbox{  and }w=\frac{1}{\sqrt{1+m^2}}(-m, 1)^T.
\end{equation}
We observe that if we define $\tilde{w}:=\frac{1}{\sqrt{1+m^2}}(-m,-1)^T$, then $\tilde{w}=R_{\tilde{\theta}}v$ for some $\tilde{\theta}\in \left(\frac{3}{2}\pi,2\pi\right)$. 

If we plug in the expression for $w$ as in~\eqref{confo54} in equation~\eqref{dckshfgtr01}, we obtain for every $y\in \S^1$ that
\begin{equation*}
0\leq \left\langle\begin{pmatrix}
m\\
-1
\end{pmatrix},\begin{pmatrix}
\gamma_1(y)\\
\gamma_2(y)
\end{pmatrix}\right\rangle=m\gamma_1(y)-\gamma_2(y).
\end{equation*}   
From this inequality, and the fact that $\gamma_2\geq 0$, we obtain for every $y\in\S^1$
\begin{equation*}
\langle \tilde{w}, (\gamma(x_0)-\gamma(y))\rangle \geq 0.
\end{equation*}
In particular, we have that $v$ and $\tilde{w}$ satisfies~\eqref{dckshfgtr01}, $v\neq \tilde{w}$ and $\tilde{w}=R_{\tilde{\theta}}v$ for some $\tilde{\theta}\in \left(\frac{3}{2}\pi,2\pi\right)$. This means that we recovered the first case, and we conclude. 
\end{proof}

\begin{figure}
\includegraphics[width=300px]{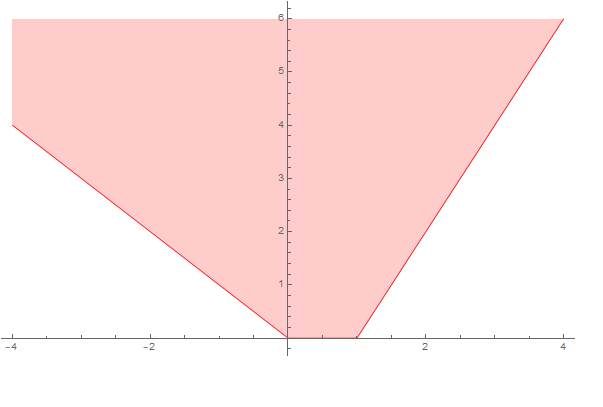}
\caption{\label{polipo} The barrier $G$ from \Cref{lm:barrierestimates}}
\end{figure}

\begin{proposition}\label{lm:barrierestimates}
Let $s\in(0,1)$, $m,\tilde{t}>0$ and $m_1\geq 0$ and consider the function given by  
\begin{equation*}
g(t):=\begin{dcases}
-mt\quad \mbox{for}\quad t\in (-\infty,0)\\
0\quad\mbox{for}\quad t\in [0,\tilde{t})\\
m_1(t-\tilde{t}) \quad\mbox{for}\quad t\in [\tilde{t},+\infty).
\end{dcases}
\end{equation*}
Set $G:=\left\lbrace (t,y)|\, y\geq g(t)  \right\rbrace$, see Figure~\ref{polipo}.

Then for some constant $c_{m,s}>0$ depending on $m$ and $s$, 
\begin{itemize}
\item if $m_1> 0$ for all $x\in (0,\tilde{t})\cup(\tilde{t},+\infty)$ 
\item if $m_1= 0$ for every $x >0$
\end{itemize}
we have 
\begin{equation*}
H_{\partial G}^s((x,g(x)))\geq \frac{c_{m,s}}{\left(\left|g(x)\right|+\left|x\right|\right)^s}.
\end{equation*}
\end{proposition}

\begin{proof}[Proof of Proposition~\ref{lm:barrierestimates}]
Let $x$ be as in the statement. Set $z:=(x,g(x))^T$. Define the unit vectors
\begin{equation*}
u:=\frac{1}{\sqrt{1+m^2}}(1,-m)^T\quad\mbox{and}\quad v:=\frac{1}{\sqrt{1+m_1^2}}(1,m_1)^T.
\end{equation*} 
Then we can parameterize the graph of $g$ using the curve 
\begin{equation*}
\mu(t):=\begin{dcases}
u t\quad\mbox{for}\quad t\in (-\infty,0]\\
(t,0)\quad\mbox{for}\quad t\in [0,\tilde{t})\\
v(t-\tilde{t})+(\tilde{t},0)\quad\mbox{for}\quad t\in [\tilde{t},+\infty).
\end{dcases}
\end{equation*}
We observe that if we denote $\hat{t}:=\frac{g(x)}{u_2}$ then
\begin{equation*}
\mu_2(t)\geq g(x)\quad\mbox{for all}\quad t\in (-\infty,\hat{t}].
\end{equation*}
From this it follows that for every $t\in (-\infty,\hat{t}]$
\begin{equation}\label{eq:202204}
\begin{split}
\left|\mu(t)-z\right|& \leq \left|\mu_1(t)-x\right|+\left|\mu_2(t)-g(x)\right|\\
&=\left|\mu_1(t)-x\right|+\left(\mu_2(t)-g(x)\right)\\
&\leq  \left|\mu_1(t)-x\right|+\left|u_2 t\right|\\
&= \left|\mu_1(t)-x\right|+ \left|-m u_1 t\right|\\
&\leq \left|\mu_1(t)-x\right|+m\left|u_1 t-x\right|\\
&\aleq_{m}\left|\mu_1(t)-x\right|.
\end{split}
\end{equation}

Now, we notice that $g$ admits second order Taylor expansion in $x$. In particular, recalling that $u^\perp:=(u_2,-u_1)^T$, we can use Remark~\ref{salamegen} and~\eqref{eq:202204} to estimate 
\begin{equation*}
\begin{split}
H_{\partial E}^s(z)&=\frac{2}{s}\int_{\R}\frac{\langle n(\mu(t)),\mu(t)-z\rangle }{\left|\mu(t)-z\right|^{2+s}}\,dt\\
&=\frac{2}{s}\int_{-\infty}^{\tilde{t}} \frac{\langle n(\mu(t)),\mu(t)-z\rangle }{\left|\mu(t)-z\right|^{2+s}}\,dt\\
&\geq \frac{2}{s}\int_{-\infty}^0 \frac{\langle n(\mu(t)),\mu(t)-z\rangle}{\left|\mu(t)-z\right|^{2+s}}\,dt\\
&=-\frac{2}{s}\int_{-\infty}^0 \frac{\langle n(\mu(t)),z\rangle}{\left|\mu(t)-z\right|^{2+s}}\,dt\\
&=-\frac{2}{s}\int_{-\infty}^0 \frac{\langle n(\mu(t)),z\rangle}{\left|\mu(t)-z\right|^{2+s}}\,dt\\
&=-\frac{2}{s} \int_{-\infty}^0\frac{ \langle u^\perp , (x,g(x))\rangle}{\left|\mu(t)-z\right|^{2+s}}\,dt \\
&=\frac{2}{s}\frac{1}{\sqrt{1+m^2}} \int_{-\infty}^0\frac{mx+g(x)}{\left|\mu(t)-z\right|^{2+s}}\,dt \\
&\gtrsim_{m,s} \int_{-\infty}^{\hat{t}}\frac{mx+g(x)}{\left|\mu_1(t)-x\right|^{2+s}}\,dt \\
&=_{m,s} \frac{mx+g(x)}{\left|u_1\hat{t}-x\right|^{1+s}}\\
\end{split}
\end{equation*}
Now, we notice that 
\begin{equation*}
\begin{split}
\left|u_1\hat{t}-x\right|&\leq \left|u_1\hat{t}\right|+\left|x\right|\\
&=\frac{\left|g(x)\right|}{m}+\left|x\right|\\ 
&\aleq_{m} \left|g(x)\right|+\left|x\right|.
\end{split}
\end{equation*}
Thus, recalling that $x,g(x)\geq 0$, we have
\begin{equation*}
H_{\partial E}^s(z)\gtrsim_{m,s} \frac{g(x)+mx}{\left|g(x)+x\right|^{1+s}}\geq \frac{c_{m,s}}{\left|g(x)+x\right|^s}.\qedhere
\end{equation*}

\end{proof}

\begin{proof}[Proof of \Cref{n77754fg}]
Let $E:=\mbox{co}(\gamma(\S^1))$. Then, thanks to Proposition~\ref{colgate} we have that $\gamma(\S^1)=\partial E$. Also, using Proposition~\ref{salame} we obtain for a.e. $x\in \S^1$ 
\begin{equation*}
H_{\partial E}^s(\gamma(x))=\frac{2}{s}\int_{\S^1} \frac{\langle n(\gamma(y)),\gamma(y)-\gamma(x)\rangle}{\left|\gamma(y)-\gamma(x)\right|^{2+s}}\,dy,
\end{equation*}
where $n(\gamma(y))=\gamma'(y)^\perp$ for a.e. $y\in\S^1$.

Making use of \Cref{nscefg} and the last identity,  we obtain the existence of some $x_0\in\S^1$ such that 
\begin{equation*}
\begin{split}
\mathscr{W}_{s,p}(\gamma)&=\lim_{\epsilon\to 0^+}\int_{\S^1\setminus B_{\epsilon}(x_0)}\left| \int_{\S^1} \frac{\langle n(\gamma(y)),\gamma(y)-\gamma(x)\rangle }{\left|\gamma(y)-\gamma(x)\right|^{2+s}}\,dy\right|^p\,dx\\
&=c_{s,p}\lim_{\epsilon\to 0^+} \int_{\S^1\setminus B_{\epsilon}(x_0)} \left|H_{\partial E}^s(\gamma(x))\right|^p\,dx\\
&=+\infty,
\end{split}
\end{equation*}
and this concludes the proof of the theorem.
\end{proof}

\begin{remark}\label{supercrtitcal}
It is important to observe that the above result holds only in critical and subcritical regime, $\mathscr{W}_{s,p}(\partial E)$ where $p \geq \frac{1}{s}$.
In the supercritical regime, $p < \frac{1}{s}$ any polygon has finite energy. 
\end{remark}

\section{Proof of Theorem~\ref{MainTheorem-23}}\label{FineConv09}

In Theorem~\ref{th:convergence} we prove that sequences in $X$ with bounded $\mathscr{W}_s$-energy weakly converge locally in $W_{2}^{s+1,\frac{1}{s}}(\S^1\setminus \Sigma)$ (where $\Sigma$ is finite) to a limit curve which belongs to $X$. Having this, \Cref{MainTheorem-23} is a consequence of the direct method of calculus of variations.

\begin{theorem}\label{th:convergence}
Assume we have a sequence of $\gamma_k \in C^1(\S^1,\R^2)$ such that 
\begin{itemize}
 \item each $\gamma_k$ is a homeomorphism onto its image,
 \item $\left|\gamma'_k\right|\equiv 1$ for all $k$,
 \item $\gamma_k$ are convex on $\S^1$, in the sense of \Cref{def:convexcurve}.
\end{itemize}
Assume moreover 
\begin{equation}\label{sfwrepolfh}
\Lambda := \sup_{k} \int_{\S^1} \brac{\int_{\S^1} \frac{\langle n(\gamma_k(y)) ,\gamma_k(y)-\gamma_k(x)\rangle}{|\gamma_k(x)-\gamma_k(y)|^{2+s}} \, dy}^{\frac{1}{s}}\, dx < +\infty.
\end{equation}
Then, {up to translation of $\gamma_k$, and up to taking a subsequence} the following holds for some finite set $\Sigma\subset\S^1$ 
\begin{itemize}
\item[1)] $\gamma_{k}$ converge uniformly to $\gamma$ in $\S^1$,
\item[2)] for each $x_0\in \S^1\setminus \Sigma$ there exists some $\rho_{x_0}$ such that $\gamma_{k}$ converges weakly to $\gamma$ in $W_2^{s+1,\frac{1}{s}}(\S^1 \cap B_{\rho_{x_0}}(x_0))$.
\item[3)] $\gamma\in C^1(\S^1,\R^2)$ and $\gamma$ satisfies all the bullet points above from $\gamma_k$.
\end{itemize}
Furthermore, it holds that 
\begin{equation}\label{883Dsft5}
\mathscr{W}_s(\gamma)\leq \liminf_{k}\mathscr{W}_s(\gamma_{k_j}).
\end{equation}
\end{theorem}

\begin{proof}[Proof of Theorem~\ref{th:convergence}]
\underline{As for 1),} by a translation argument we may assume that $\gamma_k(0)=0$ for all $k \in \R$. Since $|\gamma'_k| \equiv 1$, up to subsequence we obtain from the Arzela-Ascoli theorem the uniform convergence to some $\gamma \in \lip(\S^1,\R^3)$, $|\gamma'| \leq 1$ a.e. in $\S^1$.

\underline{Convexity is preserved:}
Fix $y \in \S^1$. Since $\gamma_k$ are by assumption differentiable at $y$, as a consequence of convexity, \eqref{fcdghte}, and~\Cref{la:vxinconvexity} we have that
\begin{equation*}
\langle n(\gamma_k(y)),\gamma_k(y)-\gamma_k(x)\rangle \geq 0\quad\mbox{for all } x\in \S^1,
\end{equation*}
where $n(\gamma_k(y))=\brac{\gamma_k'}^\perp(y)$. Since $n(\gamma_k(y))\in \S^1$ it admits a converging subsequence to some $v_y\in \S^1$, and using the uniform convergence of the full sequence $\gamma_k$, one deduces from the above equation that 
\begin{equation}\label{gbtredfg}
v_y\cdot\left(\gamma(y)-\gamma(x)\right)\geq 0\quad\mbox{for all }x\in \S^1.
\end{equation}
In particular, we have proved $\gamma$ is still convex in the sense of \Cref{def:convexcurve}.

\underline{As for 2), }
by a covering argument, Proposition~\ref{pr:uniformsmallness}, for each $\delta>0$ there exists a subset $\Sigma\subset \S^1$ such that $\#\Sigma\leq L$ with $L(\Lambda, \delta)\in \N$ and for each $x_0\in \S^1\setminus \Sigma$ there exists some $\rho_{x_0}\in (0,1)$ with respect to which  
\begin{equation}\label{fcdop}
\sup_{k}\int_{\S^1} \brac{\int_{B_{40\rho_{x_0}}(x_0)} \frac{\langle n(\gamma_k(y)),\gamma_k(y)-\gamma_k(x)\rangle }{|\gamma_k(x)-\gamma_k(y)|^{2+s}} \, dy}^{\frac{1}{s}} <\delta.
\end{equation}

Thus, by choosing $\delta$ small enough, we can apply \Cref{th:sobolevcontrol} and obtain that 
\begin{equation*}
\left[\gamma_k'\right]_{W_2^{s,\frac{1}{s}}(B_{5\rho_{x_0}}(x_0))}\leq C\delta.
\end{equation*}

By reflexivity of $W_2^{s,\frac{1}{s}}(B_{5\rho_{x_0}}(x_0))$ and Banach-Alaoglu combined with Rellich's Theorem, we find that $\gamma_k'$ weakly converges to $\gamma'$, up to a subsequence, in  $W_2^{s,\frac{1}{s}}(B_{5\rho_{x_0}}(x_0))$ and the convergence is pointwise a.e. in $B_{5\rho_{x_0}}(x_0)$ and strong in $L^1$. This proves 2).

\underline{$\gamma$ is in arclength-parametrization:} Since $\gamma_k'$ converges a.e. in $\S^1\setminus \Sigma$ and $\Sigma$ is a finite set, we obtain that 
\begin{equation}\label{fcdklop}
\left|\gamma'(x)\right|=\lim_{k}\left|\gamma_k'(x)\right|\equiv 1\quad \mbox{for a.e. } x\in\S^1.  
\end{equation}

\underline{$\gamma$ is locally biLipschitz:}
Also, by choosing $\delta$ small enough, as a consequence of \eqref{fcdop}, Proposition~\ref{pr:VMOcontrol} and \Cref{la:bilipvmo} we obtain that for each $x_0\in \S^1\setminus \Sigma$ and $k\in\N$
\begin{equation*}
\frac{1}{2}\left|x-y\right|\leq \left|\gamma_k(x)-\gamma_k(y)\right|\leq \left|x-y\right|\quad\mbox{for all}\quad x,y\in B_{\rho_{x_0}}(x_0).
\end{equation*}
From this and the uniform convergence of $\gamma_k$ to $\gamma$ we obtain that 
\begin{equation}\label{Giovedi}
\frac{1}{2}\left|x-y\right|\leq \left|\gamma(x)-\gamma(y)\right|\leq \left|x-y\right|\quad\mbox{for all}\quad x,y\in B_{\rho_{x_0}}(x_0).
\end{equation}
Now, thanks to the uniform convergence of $\gamma_k$ to $\gamma$ we obtain that if for some $x_0\in\S^1\setminus \Sigma$ and any $y_0\in \S^1$ it holds that 
\begin{equation}\label{gbtrep}
\left|\gamma(x_0)-\gamma(y_0)\right|<\frac{\rho_{x_0}}{100},
\end{equation}   
then for all $k$ sufficiently large it holds that 
\begin{equation*}
\left|\gamma_k(x_0)-\gamma_k(y_0)\right|<\frac{\rho_{x_0}}{100}.
\end{equation*} 
Using this last equation together with~\eqref{fcdop} and Theorem~\ref{th:smallenergyinjectivity}, we evince that
\begin{equation}\label{qsqsqsqewr}
\left|x_0-y_0\right|\leq \rho_{x_0}.
\end{equation}
Moreover, assume that $\gamma(x)=\gamma(y)$ for some $x,y\in\S^1$. Then, if $x\notin\Sigma$, as a consequence of~\eqref{gbtrep} and~\eqref{qsqsqsqewr} we deduce that $y\in B_{\rho_x}(x)$. In particular, using this last information and~\eqref{Giovedi} we obtain that $x=y$.  Analogously, one shows that if $y\notin \Sigma$ then $x=y$. In conclusion, 
\begin{equation}\label{jfdert}
\mbox{if  } \gamma(x)=\gamma(y) \mbox{  then either  } x=y \mbox{  or  } x,y\in\Sigma. 
\end{equation}

\underline{Lower semicontinuity:}
Now that we obtained point wise convergence of $\gamma_k'$, we notice that for a.e. $x,y\in \S^1\setminus \Sigma$ such that $x\neq y$, one has that 
\begin{equation*}
\frac{\langle n(\gamma(y)), \gamma(y)-\gamma(x)\rangle }{\left|\gamma(y)-\gamma(x)\right|^{2+s}}=\lim_{k\to +\infty} \frac{\langle n(\gamma_k(y), (\gamma_k(y)-\gamma_k(x)\rangle}{\left|\gamma_k(y)-\gamma_k(x)\right|^{2+s}}.
\end{equation*}
In particular, from Fatou's lemma we obtain that
\begin{equation*}
\left(\int_{\S^1}\frac{\langle n(\gamma(y)),\gamma(y)-\gamma(x)\rangle }{\left|\gamma(y)-\gamma(x)\right|^{2+s}}\,dy\right)^\frac{1}{s}\leq \liminf_{k\to +\infty} \left(\int_{\S^1}\frac{\langle n(\gamma_k(y)), \gamma_k(y)-\gamma_k(x)\rangle}{\left|\gamma_k(y)-\gamma_k(x)\right|^{2+s}}\,dy\right)^\frac{1}{s}.
\end{equation*}
By integrating with respect to $x$ in $\S^1$ the above expression and applying Fatou's Lemma a second time, we obtain 
\begin{equation*}
\mathscr{W}_s(\gamma)\leq \liminf_{k}\mathscr{W}_s(\gamma_k).
\end{equation*}
This concludes the proof of lower semicontinuity, i.e. \eqref{883Dsft5}. 

\underline{$\gamma$ is a homeomorphism}: In view of \Cref{topolmfd}, $\gamma(\S^1)$ is a topological one-manifold. Since $\gamma$ has only finitely many self intersection, namely only on $\Sigma$, this implies that actually $\gamma$ must be injective -- for the details of this argument see \cite{BRSV21}. 

\underline{$\gamma$ is $C^1$}: Since $\gamma$ has finite Willmore-energy it belongs to $X$ as in \eqref{def:setX}.

Therefore, applying \Cref{n77754fg} we deduce that $\gamma\in C^1(\S^1,\R^2)$. This last result together with~\eqref{gbtredfg} and~\eqref{fcdklop} proves $3)$.

\end{proof}

\appendix

\section{Facts about convex curves in \texorpdfstring{$\R^2$}{R2}}
In this section we recall geometric properties of convex curves and graphs, where we recall the notion of convexity in \eqref{fcdghte}.

\begin{lemma}\label{la:vxinconvexity}
Assume $\gamma\in Lip((a,b),\R^2)$, $\left|\gamma'\right|\equiv 1$ a.e. in $\S^1$ and $\gamma$ is convex in the sense of \eqref{fcdghte}.  Then, if $\gamma$ is differentiable at $x_0\in (a,b)$, it holds that $\langle \nu_{x_0}, \gamma'(x_0)\rangle=0$. Also, if $\gamma((a,b))$ is not a straight line $\nu_{x_0}$ is unique.
\end{lemma}

\begin{proof}
We observe that according to~\eqref{fcdghte} one has that
\begin{equation*}
0 \geq \left\langle \nu_{x_0}, \lim_{x\to x_0^{+}} \frac{\gamma(x)-\gamma(x_0)}{x-x_0}\right\rangle=\langle\nu_{x_0}, \gamma'(x_0)\rangle
\end{equation*}
and similarly
\begin{equation*}
0 \leq  \left\langle\nu_{x_0}, \lim_{x\to x_0^{-}} \frac{\gamma(x)-\gamma(x_0)}{x-x_0}\right\rangle=\langle\nu_{x_0}, \gamma'(x_0)\rangle.
\end{equation*}
From the above equations we deduce that $\langle \gamma'(x_0), \nu_{x_0}\rangle=0$. Also, we claim that if $\gamma((a,b))$ is not a line then
\begin{equation}\label{giagio}
\nu_{x_0}\quad\mbox{is unique}.
\end{equation}
Indeed, if~\eqref{giagio} was false we would necessarily have that there are two vectors $\nu_{x_0}, \tilde{\nu}_{x_0}\in\S^1$ such that $\langle\nu_{x_0},\gamma'(x_0)\rangle= 0$ and $\langle\tilde{\nu}_{x_0}, \gamma'(x_0)\rangle=0$, which implies that (up to renaming the two vectors) $\nu_{x_0}=-\gamma'(x_0)^\perp$ and $\tilde{\nu}_{x_0}=\gamma'(x_0)^\perp$. Thus, we would have that $\nu_{x_0}=-\tilde{\nu}_{x_0}$ and
\begin{equation*}
\nu_{x_0}\cdot (\gamma(x_0)-\gamma(y))\geq 0 \quad\mbox{and}\quad -\nu_{x_0}\cdot (\gamma(x_0)-\gamma(y))\geq 0.
\end{equation*}
This implies that $\gamma((a,b))$ is necessarily a straight line, providing a contradiction.
\end{proof}

\begin{lemma}\label{la:convex1}
Assume there are two distinct points $p_0, p_1 \in \R^2$, and set
\[
 T = p_1-p_0 \in \R^2 \setminus \{0\}
\]
Assume that there is $m$ on the segment $[p_1,p_0]$, i.e. if is some $\lambda \in [0,1]$ such that 
\begin{equation}\label{eq:convset:asdasd2}
 m = p_0 + \lambda T.
\end{equation}
Then if at $m$ we have the convexity condition \Cref{def:convexcurve}, that is if we find $v_{m} \in \R^2$ such that $|v_{m}|=1$ and
\begin{equation}\label{eq:convset:3240}
 \langle v_{m}, m-y\rangle \geq 0 \quad \forall y \in \{p_1,p_2\}
\end{equation}
Then necessarily
\[
 v_{m} \in \left \{ -\frac{T^\perp}{|T^\perp|}, \frac{T^\perp}{|T^\perp|} \right \}
\]
\end{lemma}
\begin{proof}
By scaling we may assume that $|T| = |p_1-p_0|=1$.

The assumption \eqref{eq:convset:asdasd2} readily implies
\[
\langle T^\perp, m-p_0\rangle =0
\]
Since $T=p_1-p_0$ we thus have
\begin{equation}\label{eq:towconvset:1234}
 0=\langle m-p_0,T^\perp \rangle=\langle m-p_1,T^\perp \rangle.
\end{equation}
Assume that for some $v_{m} \in \R^2$, $|v_{m}|=1$, we have \eqref{eq:convset:3240}, i.e. assume
\[
 \begin{cases}\langle v_{m}, m-p_0\rangle \geq 0 \quad &\text{and}\\
 \langle v_{m}, m-p_1\rangle \geq 0 
 \end{cases}
\]

Since $|T|=|T^\perp| =1$ we can write $v_{m} = \langle v_{m}, T \rangle T+ \langle v_{m}, T^\perp \rangle T^\perp$. Thus, taking into account \eqref{eq:towconvset:1234}  the previous equations become
\[
\begin{cases}
 \langle v_{m}, T\rangle \langle T, m-p_0\rangle \geq 0 \quad &\text{and}\\
 \langle v_{m}, T\rangle  \langle T, m-p_1\rangle \geq 0 
 \end{cases}
\]
But now 
\[
\langle T, m-p_1\rangle =\langle T, m-p_0 - T\rangle = \langle T, m-p_0 \rangle -1
\]
So we find 
\[
\begin{cases}
 \langle v_{m}, T\rangle \langle T, m-p_0\rangle \geq 0 \quad &\text{and}\\
 \langle v_{m}, T\rangle  \brac{\langle T, m-p_0\rangle-1} \geq 0 
 \end{cases}
\]
In view of \eqref{eq:convset:asdasd2} this implies 
\[
\begin{cases}
 \langle v_{m}, T\rangle \lambda \geq 0 \quad &\text{and}\\
 \langle v_{m}, T\rangle  \brac{\lambda-1} \geq 0 
 \end{cases}
\]
Since $\lambda \in [0,1]$ this implies $\langle v_{m}, T\rangle=0$, that is $v_{m} = \pm T^\perp$. We can conclude.
\end{proof}

\begin{lemma}\label{la:convex2}
Assume $\gamma: [0,1] \to \R^2$ is continuous, convex in the sense of \Cref{def:convexcurve} and that 
\[
 T = \gamma(1)-\gamma(0) \in \R^2 \setminus \{0\}.
\]
Assume further more that $\gamma(t) \not \in \{\gamma(0),\gamma(1)\}$ for all $t \in (0,1)$.
Then either 
\[
 \langle T^\perp, \gamma(t)-\gamma(0)\rangle \geq 0 \quad \forall t \in [0,1]
\]
or 
\[
 \langle T^\perp, \gamma(t)-\gamma(0)\rangle \leq 0 \quad \forall t \in [0,1]
\]
\end{lemma}
\begin{proof}
By continuity, injectivity, and the intermediate value theorem, if the claim is false,  the only way the claim is not true is if there exists $\bar{t} \in [0,1]$ with $\gamma(\bar{t}) \neq \gamma(0),\gamma(1)$ but 
\[
 \langle T^\perp, \gamma(\bar{t})-\gamma(0)\rangle = 0.
\]
This implies that one of the following must be true 
\[
 \begin{cases} 
 \gamma(\bar{t}) \in [\gamma(0),\gamma(1)] \quad &\text{or}\\
  \gamma(0) \in [\gamma(\bar{t}),\gamma(1)]\quad &\text{or}\\
  \gamma(1) \in [\gamma(\bar{t}),\gamma(0)]
 \end{cases}
\]
In either case we find from the convexity condition in \Cref{def:convexcurve} combined with \Cref{la:convex1} that  
\[
\begin{cases}
 \langle T^\perp, \gamma(\tilde{t}) - \gamma(t) \rangle \geq 0 \quad \forall t \in [0,1] \quad& \text{or}\\
 \langle T^\perp, \gamma(\tilde{t}) - \gamma(t) \rangle \leq 0 \quad \forall t \in [0,1]\\
 \end{cases}
\]
when $\tilde{t} \in \{0,1,\bar{t}\}$ is chosen correctly. Whichever of the three options $\tilde{t}$ is, we have $\gamma(\tilde{t}) - \gamma(0) \in {\rm span} \{T\}$, so we find either
\[
\begin{cases}
 \langle T^\perp, \gamma(0) - \gamma(t) \rangle \geq 0 \quad \forall t \in [0,1] \quad& \text{or}\\
 \langle T^\perp, \gamma(0) - \gamma(t) \rangle \leq 0 \quad \forall t \in [0,1]\\
 \end{cases}
\]
We can conclude.
\end{proof}

\begin{lemma}\label{le:cvdfqw}
Assume that $(x,f(x)): [0,1] \to \R^2$ is a continuous graph that is convex in the sense of \Cref{def:convexcurve}.

Then $f$ is either convex or concave on $[0,1]$.
\end{lemma}
\begin{proof}
Assume the claim is wrong. 
Then there must be some some $f$, and we may assume w.l.o.g. $f(0) = 0$, $f(1)=1$, such that for some $0 < \mu < \lambda < 1$  $(\lambda,f(\lambda))$ lies below the line connecting $(0,f(0))$ and $(1,f(1))$, and $f(\mu)$ lies above the line connecting $(0,f(0))$ and $(\lambda,f(\lambda))$. In formulas we have
\begin{equation}\label{eq:convexgraph:2133254}
f(\lambda ) <  \lambda f(1)
\end{equation}
and 
\begin{equation}\label{eq:convexgraph:213we}
f(\mu) > \frac{\mu}{\lambda} f(\lambda)
\end{equation}
\begin{figure}
\includegraphics[width=200px]{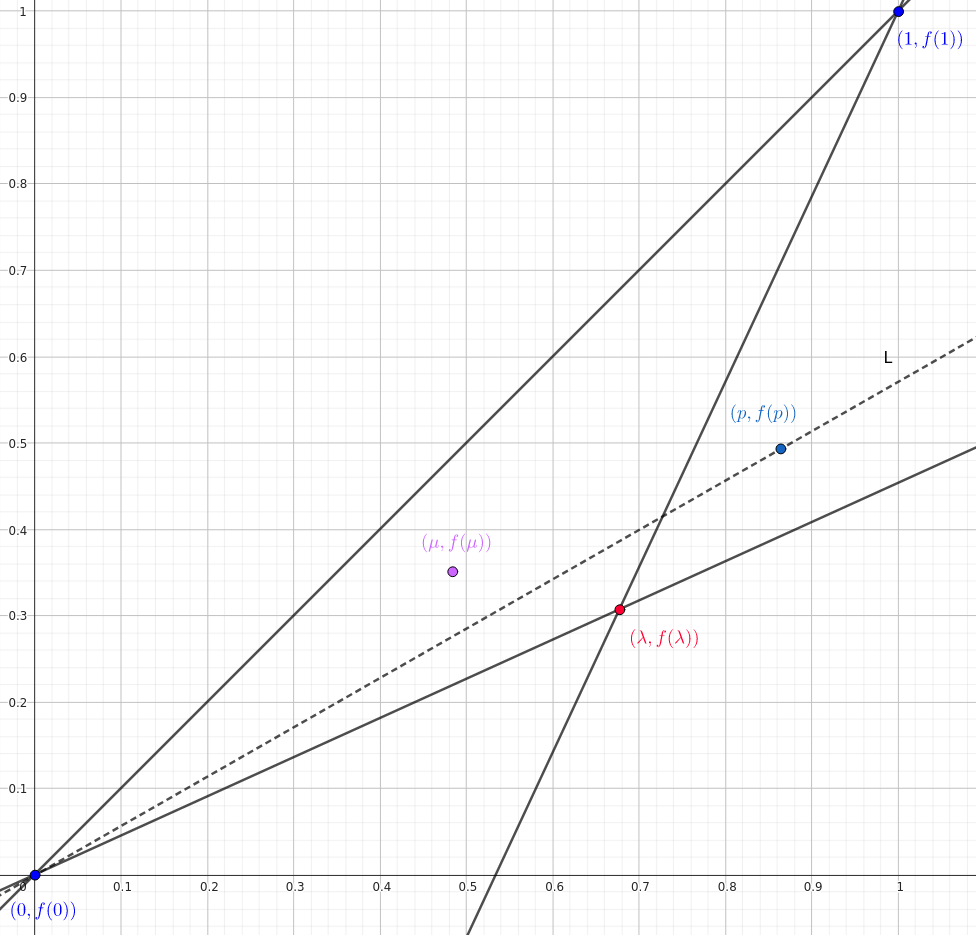}
\caption{\label{fig:cvdfqw}The geometry in the proof of \Cref{le:cvdfqw}}
\end{figure}

We now consider the line starting from $(0,f(0))$ with slope 
\[
m_\eps:= \frac{f(\lambda)-f(0)}{\lambda -0} + \eps,
\]
where $\eps$ is small. Cf. \Cref{fig:cvdfqw}. This line is the graph of 
\[
g_\eps(x) := m_\eps x
\]
We clearly have $g(0)=f(0)=0$. 
For all small enough $\eps > 0$
\[
 g_\eps(\mu) = \frac{\mu}{\lambda} f(\lambda)+\mu \eps \overset{\eqref{eq:convexgraph:213we}}{<} f(\mu)
\]
Also, for all $\eps > 0$
\[
g_\eps(\lambda) = f(\lambda) + \eps > f(\lambda)
\]
Lastly, for possibly even smaller $\eps >0$ we have
\[
g_\eps(1) = \frac{f(\lambda)}{\lambda} + \eps \overset{\eqref{eq:convexgraph:2133254}}{<} f(1)
\]
By the intermediate value theorem we find $p \in (\lambda,1)$ such that $g_\eps(p) = f(p)$

That is the curve $\gamma(x) := (x,f(x))$ and the line $L=(x,g_\eps(x))$ coincide at $0$ and $p$ but $\gamma(\mu)$ lies above the line $L$ and $\gamma(\lambda)$ lies below the line $L$ -- which is a contradiction to \Cref{la:convex2}.
\end{proof}

\begin{lemma}\label{iut4512}
Assume $\gamma: \S^1 \to \R^2$ is continuous {and injective} and fix $\theta_1 \neq \theta_2 \in \S^1$.

Set 
\[
 T := \gamma(\theta_2) - \gamma(\theta_1) \in \R^2 \setminus \{0\}.
\]

We consider the two different arcs connecting $\theta_1$ and $\theta_2$ in $\S^1$, we call them $[\theta_1,\theta_2]$ and $[\theta_2,\theta_1]$.

Then the curve is separated by $\gamma(\theta_1) + \lambda T$, more precisely we have either  
\[
 \begin{cases} \langle \gamma(\theta_1)-\gamma(t), T^\perp\rangle \geq 0 \quad \forall t \in [\theta_1,\theta_2] \quad& \text{and}\\
 \langle \gamma(\theta_1)-\gamma(t), T^\perp\rangle \leq 0 \quad \forall t \in [\theta_2,\theta_1] 
 \end{cases}
\]
or 
\[
 \begin{cases} \langle \gamma(\theta_1)-\gamma(t), T^\perp\rangle \leq 0 \quad \forall t \in [\theta_1,\theta_2] \quad& \text{and}\\
 \langle \gamma(\theta_1)-\gamma(t), T^\perp\rangle \geq 0 \quad \forall t \in [\theta_2,\theta_1] 
 \end{cases}
\]
\end{lemma}
\begin{proof}
In view of \Cref{la:convex2}, the only way the claim is false would be that there exists $t^+ \in (\theta_1,\theta_2)$ and $t^- \in (\theta_2,\theta_1)$ such that (the sign is w.l.o.g. negative)
\[
 \langle \gamma(\theta_1)-\gamma(t^+), T^\perp\rangle < 0, \text{ and }\langle \gamma(\theta_1)-\gamma(t^-), T^\perp\rangle < 0
\]
and otherwise we have 
 \[
 \langle \gamma(\theta_1)-\gamma(t), T^\perp\rangle \leq 0 \quad \forall t \in \S^1.
 \]
By continuity and injectivity, we may assume w.l.o.g. that for some $\lambda \in (0,1)$, $0 < \mu_- < \mu_+$ we have
\[
 \gamma(t^+) = \gamma(\theta_1) + \lambda T + \mu_+ T^\perp
\]
and 
\[
 \gamma(t^-) = \gamma(\theta_1) + \lambda T + \mu_- T^\perp
\]

\[
  0 > \langle \gamma(\theta_1)-\gamma(t^-), T^\perp\rangle > \langle \gamma(\theta_1)-\gamma(t^+), T^\perp\rangle.
\]
Set 
\[
 w_1 := \gamma(t^-)-\gamma(t^+)
\]
\[
 w_2 := \gamma(t^-)-\gamma(\theta_1)
\]
\[
 w_3 := \gamma(t^-)-\gamma(\theta_2)
\]
We then observe from $w_1 = \gamma(\theta_1)-\gamma(t^+) - \gamma(\theta_1) - \gamma(T^-)$ we have
\[
 \langle w_1 , T^\perp \rangle <0.
\]
Also we have 
\[
 \langle w_2,T^\perp \rangle >0
\]
and since $w_3 = w_2 -T$ we find 
\[
 \langle w_3,T^\perp \rangle =\langle w_2,T^\perp \rangle > 0.
\]
Moroever we have 
\[
 \langle w_1, T\rangle = 0,
\]
\[
 \langle w_2, T \rangle = \lambda,
\]
\[
 \langle w_3, T \rangle = \lambda -1.
\]

Now let $v_{t^-} \in \R^2$, $|v_{t^-}|=1$ such that 
\[
 \langle v_{t^-},\gamma(t^-)-\gamma(t)\rangle \geq 0 \quad \forall t \in \S^1.
\]
That is, we have 
\[
  \langle v_{t^-},w_i\rangle \geq 0 \quad i=1,2,3
\]
Take $\sigma, \delta \in \R$ such that $v_{t^-} = \sigma T + \delta T^\perp$. Then we have 
\[
  \sigma \underbrace{\langle  T,w_1\rangle}_{=0} + \delta \underbrace{\langle T^\perp ,w_1\rangle}_{< 0} \geq 0 
\]
This implies $\delta = 0$. Then we comp
\[
  \sigma \underbrace{\langle  T,w_2\rangle}_{=\lambda} + \underbrace{\delta}_{=0} \langle T^\perp ,w_2\rangle \geq 0 
\]
\[
  \sigma \underbrace{\langle  T,w_3\rangle}_{=\lambda-1} + \underbrace{\delta}_{=0} \langle T^\perp ,w_3\rangle \geq 0 
\]
Since $\lambda \in (0,1)$ these equalities imply that $\sigma =0$, a contradiction to $|v_{t^-}|=1$.
\end{proof}

\begin{proposition}\label{colgate}
Let $\gamma\in Lip(\S^1,\R^2)$ be injective and satisfying the convexity assumption~\eqref{fcdghte} in $\S^1$. Then, if we define the convex set $E:=\mbox{co}(\gamma(\S^1))$ it holds that 
\begin{equation*}
\partial E=\gamma(\S^1).
\end{equation*} 
\end{proposition}

\begin{proof}[Proof of Proposition~\ref{colgate}]

We begin by proving that 
\begin{equation}\label{fincu-1}
\gamma(\S^1)\subset \partial E.
\end{equation}
To do so, we first we observe that $E\neq \emptyset$, and  since $\gamma(\S^1)$ is compact, then we can apply Theorem~1.4.3 in~\cite{hiriart2004fundamentals} and obtain that $E$ is compact.

Also, being $\gamma(\S^1)$ a connected set, we can apply Theorem~1.3.7 in~\cite{hiriart2004fundamentals} and obtain that 
\begin{equation}\label{goliardo}
E=\left\lbrace \lambda\gamma(t_1)+(1-\lambda)\gamma(t_2)|\, t_1,t_2\in\S^1\mbox{  and  }\lambda\in [0,1] \right\rbrace.
\end{equation}

In particular, from~\eqref{goliardo} and equation~\eqref{fcdghte} we deduce that for every  $t\in\S^1$ and $P\in\S^1$ there are $t_1,t_2\in \S^1$ and $\lambda\in [0,1]$ such that
\begin{equation*}
\begin{split}
\langle \nu_t, \gamma(t)-P\rangle & =\langle \nu_t, \gamma(t)-\lambda\gamma(t_1)-(1-\lambda)\gamma(t_2)\rangle\\
&=\lambda \langle \nu_t, \gamma(t)-\gamma(t_1)\rangle + (1-\lambda)  \langle \nu_t, \gamma(t)-\gamma(t_2)\rangle\\
&\geq 0,
\end{split}
\end{equation*}
In particular, we have just proved that for every $t\in \S^1$ there exists some $\nu_t\in\S^1$ such that for every $Q\in E$  
\begin{equation*}
\langle \nu_t, \gamma(t)-Q\rangle\geq 0.
\end{equation*}
From this and the fact that $\gamma(t)\in E$, we readily evince~\eqref{fincu-1}.

To conclude the proof of the proposition we claim that
\begin{equation}\label{fincu-2}
\partial E\subset \gamma(\S^1).
\end{equation}
To prove this, let $P\in \partial E$. Then, applying Lemma~4.2.1 in~\cite{hiriart2004fundamentals} we obtain that there exists some $\nu_P \in\S^1$ such that
\begin{equation}\label{federet56}
\langle \nu_P, P-\gamma(t)\rangle\geq 0\quad\mbox{for all}\quad t\in\S^1.
\end{equation}  
According to~\eqref{goliardo}, there are $t_1,t_2\in\S^1$ and $\lambda\in [0,1]$ such that
\begin{equation*}
P=\lambda \gamma(t_1)+(1-\lambda)\gamma(t_2).
\end{equation*}
If either $\gamma(t_1)=\gamma(t_2)$ or $\lambda\in \left\lbrace 0,1 \right\rbrace$ then $P\in \gamma(\S^1)$, and there is nothing to prove. We assume therefore that 
\begin{equation*}
\gamma(t_1)\neq \gamma(t_2)\quad\mbox{and}\quad \lambda\in (0,1).
\end{equation*}
Thus, we observe that 
\begin{equation*}
0\leq \langle  \nu_P, P-\gamma(t_1)\rangle=\langle \nu_P, \gamma(t_2)-\gamma(t_1)\rangle(1-\lambda)
\end{equation*} 
and also
\begin{equation*}
0\leq  \langle  \nu_P, P-\gamma(t_2)\rangle=\langle \nu_P, \gamma(t_1)-\gamma(t_2)\rangle \lambda.
\end{equation*}
From these two equations we evince that if we denote $T:= \frac{\gamma(t_2)-\gamma(t_1)}{\left|\gamma(t_2)-\gamma(t_1)\right|}$
\begin{equation*}
\nu_P\in \left\lbrace T^\perp, -T^\perp\right\rbrace. 
\end{equation*}
and without any loss of generality we can assume that $\nu_p=T^\perp$.

Now, using Lemma~\ref{iut4512} with $\theta_1=t_1$ and $\theta_2=t_2$ we have that either
\begin{equation}\label{gianna}
\begin{cases} \langle \gamma(t_1)-\gamma(t), T^\perp\rangle \geq 0 \quad \forall t \in [t_1,t_2] \quad& \text{and}\\
\langle \gamma(t_1)-\gamma(t), T^\perp\rangle \leq 0 \quad \forall t \in [t_2,t_1] 
\end{cases}
\end{equation}
or
\begin{equation*}
\begin{cases} \langle \gamma(t_1)-\gamma(t), T^\perp\rangle \leq 0 \quad \forall t \in [t_1,t_2] \quad& \text{and}\\
\langle \gamma(t_1)-\gamma(t), T^\perp\rangle \geq 0 \quad \forall t \in [t_2,t_1].
\end{cases}
\end{equation*}
Without loss of generality, we assume that~\eqref{gianna} holds, the second case being analogous. 

On the other hand, as a consequence of~\eqref{federet56}, we also have for every $t\in\S^1$ that
\begin{equation*}
\begin{split}
0\leq \langle\nu_P, P-\gamma(t) \rangle & = \langle T^{\perp}, \left|P-\gamma(t_1)\right|T +\gamma(t_1)-\gamma(t)\rangle\\
&=\langle T^\perp,\gamma(t_1)-\gamma(t)\rangle
\end{split}
\end{equation*}
Thus, in view of this and equation~\eqref{gianna}, we obtain that 
\begin{equation*}
\langle T^\perp, \gamma(t_1)-\gamma(t)\rangle=0\quad\mbox{for all}\quad t\in [t_2,t_1].
\end{equation*}  
In particular, from this we infer the existence of some $\mu\in C([t_2,t_1])$  such that $\gamma(t)=\mu(t) T +\gamma(t_1)$. Since by assumptions we have that $\mu(t_1)=0$ and $\mu(t_2)=\left|\gamma(t_2)-\gamma(t_1)\right|$, then by continuity there exists some $t^*\in [t_2,t_1]$ such that 
\begin{equation*}
\mu(t^*)=(1-\lambda)\left|\gamma(t_2)-\gamma(t_1)\right|.
\end{equation*} 
In particular, for such $t^*$ it holds that
\begin{equation*}
\begin{split}
\gamma(t^*)&=\mu(t^*)T+\gamma(t_1)\\
&=(1-\lambda)\left(\gamma(t_2)-\gamma(t_1)\right)+\gamma(t_1)\\
&=\lambda \gamma(t_1)+(1-\lambda)\gamma(t_2)\\
&=P.
\end{split}
\end{equation*}
We have just proved that $P\in \gamma(\S^1)$, and~\eqref{fincu-2} follows.  
\end{proof}

\begin{proposition}\label{convcurvfunc}
Let $\gamma\in Lip((a,b),\R^2)$, $\left|\gamma'\right|\equiv 1$ a.e. in $(a,b)$,  injective and satisfying the convexity assumption~\eqref{fcdghte} in $(a,b)$. Furthermore, we assume that  there exists some Lipschitz $f:\gamma_1((a,b))\to \R$ such that 
\begin{equation}\label{hosdgter}
\gamma(t)=(\gamma_1(t),f(\gamma_1(t)))^T\quad\mbox{for all}\quad t\in (a,b).    
\end{equation}
Then, up to replacing $\gamma(\cdot)$ with $\gamma((a+b)-\cdot)$, for a.e. $t\in (a,b)$ it holds that 
\begin{equation}\label{FSAQW}
\gamma_1'(t)=\frac{1}{\sqrt{1+f'(\gamma_1(t))^2}}.
\end{equation}
Also, up to rotating $\gamma$ by $\pi$, $f$ is convex. 
\end{proposition}

\begin{proof}
If $\gamma((a,b))$ is a straight line then the statement easily follows. For this reason, from now on we assume that $\gamma((a,b))$ is not a straight line.

We notice that the convexity of $f$ follows immediately from Lemma~\ref{le:cvdfqw} up to rotating $\gamma$ by $\pi$. 

Now, we evince from the injectivity of $\gamma$ and~\eqref{hosdgter} that $\gamma_1:(a,b)\to \gamma_1((a,b))$ is a Lipschitz bijection. Therefore, up to replacing $\gamma(\cdot)$ with $\gamma((a+b)-\cdot)$, we have that
\begin{equation}\label{eq:croterf}
\gamma_1'>0 \quad \mbox{a.e. in} \quad (a,b). 
\end{equation} 
From this we obtain that if $S\subset \gamma_1((a,b))$ such that $\sigma^1(S)=0$, then using the coarea formula (see for instance~\cite{evans2018measure}) we obtain 
\begin{equation}\label{eq:coare}
\begin{split}
\int_{\gamma_1^{-1}(S)}\left|\gamma_1'(x)\right|\,dx&=\int_{\R}\sigma^0\left(\gamma_1^{-1}(S)\cap \gamma_1^{-1}(y)\right)\,dy\\
&=\int_{\R}\chi_{S}(y)\,dy=0\\
&=\int_{S}\,dy\\
&=0.
\end{split}
\end{equation}
Henceforth, from~\eqref{eq:croterf} and~\eqref{eq:coare} we infer that $\sigma^1(\gamma_1^{-1}(S))=0$. Analogously, one can prove that if $S\subset (a,b)$ and $\sigma^1(S)=0$ then $\sigma^1(\gamma_1(S))=0$. We conclude that for every $S\subset (a,b)$
\begin{equation}\label{eq:zeno}
\sigma^1(S)=0 \iff  \sigma^1(\gamma_1(S))=0.
\end{equation}

Therefore, making use of~\eqref{eq:zeno} and the fact that $\gamma_1$ and $f$ are respectively differentiable a.e. in $(a,b)$ and $\gamma_1(a,b)$, we deduce for a.e. $t\in (a,b)$ 
\begin{equation}\label{eq:sre-67132}
\gamma'(t)=(\gamma_1'(t),f'(\gamma_1(t))\gamma_1'(t))^T, 
\end{equation}
and using the hypothesis that $\left|\gamma'\right|\equiv 1$ a.e. in $(a,b)$ together with~\eqref{eq:croterf} we obtain for a.e. $t\in (a,b)$ that
\begin{equation*}
\gamma_1'(t)=\frac{1}{\sqrt{1+f'(\gamma_1(t))^2}}.
\end{equation*}
This concludes the proof of~\eqref{FSAQW}.\qedhere

\end{proof}

\begin{theorem}[Implicit Function Theorem, see~\cite{clarke1990optimization}]\label{ImFuTh}
Let $\gamma\in Lip([a,b],\R^2)$, injective, $\left|\gamma'\right|\equiv 1$ a.e. and $\gamma$ satisfies the convexity assumption~\eqref{fcdghte} in $[a,b]$. Then, for every $x_0\in (a,b)$ there exists some open neighbourhood  $U_{x_0}\subset(a,b)$ of $x_0$ and a convex function $f$, such that up to rotations we have that 
\begin{equation}\label{DRK}
\gamma(t)=(\gamma_1(t),f(\gamma_1(t)))^T\quad\mbox{for all}\quad t\in U_{x_0},
\end{equation}
and 
\begin{equation}\label{DRK-2}
\gamma \mbox{ is differentiable in }t\in U_{x_0} \iff \mbox{$f$ is differentiable in $\gamma_1(t)$}.
\end{equation}
Moreover, up to replacing $\gamma(\cdot)$ with $\gamma((a+b)-\cdot)$,  for all $t\in U_{x_0}$ where $\gamma$ is differentiable
\begin{equation}\label{FSAQW2}
\nu_t=\frac{1}{\sqrt{1+f'(\gamma_1(t))^2}}(f'(\gamma_1(t)),-1)^T,
\end{equation}
where $\nu_t$ is as in Definition~\ref{def:convexcurve}.
\end{theorem}

\begin{proof}
If $\gamma([a,b])$ coincides with a straight line then the statement easily follows. Therefore, from now on we assume that $\gamma([a,b])$ is not a line.

We begin by setting the distance function
\begin{equation*}
d(y):=\inf\left\lbrace   \left|x-y\right||x\in \gamma([a,b]) \right\rbrace.
\end{equation*}
We fix some $x_0\in (a,b)$, and up to a translation we assume that $\gamma(x_0)=0$. Then, for $\epsilon>0$ small enough we consider the ball $B_\epsilon$, and define the sets
\begin{equation*}
B_{\epsilon}^+:=\left\lbrace x\in B_\epsilon|\, \nu_y\cdot (\gamma(y)-x)\geq 0 \mbox{ for all }y\in [a,b]\mbox{ and  $\nu_y$ as in~\eqref{fcdghte}} \right\rbrace,
\end{equation*}
and  
\begin{equation*}
B_{\epsilon}^-:=B_{\epsilon}\setminus B_{\epsilon}^{+}.
\end{equation*}
Then, in $B_{\epsilon}$ we define the function 
\begin{equation*}
\tilde{d}(x):=\begin{dcases}
-d(x)\quad\mbox{for all}\quad x\in B_{\epsilon}^+\\
d(x)\quad\mbox{for all}\quad x\in B_{\epsilon}^-.
\end{dcases}
\end{equation*}
Being $\gamma([a,b])$ closed, we have that 
\begin{equation}\label{kresta}
p\in \gamma([a,b]) \iff d(p)=0 \iff \tilde{d}(p)=0.
\end{equation}
Also, thanks to \cite[Proposition~2.4.1 ]{clarke1990optimization} we have that $d$ is Lipschitz. From this, it follows that $\tilde{d}$ is Lipschitz as well and therefore $\nabla d$ and  $\nabla \tilde{d}$ exist a.e. in $\R^2$. Moreover, thanks to \cite[Theorem 2.5.1]{clarke1990optimization} we can define the generalized gradient of $\tilde{d}$ in $0$ as
\begin{equation}\label{gengrad}
\partial \tilde{d}(0)=\mbox{co}\left\lbrace \lim_{n}\nabla \tilde{d}(x_n)|\, x_n\to 0,\, x_n\notin \Omega_{\tilde{d}}       \right\rbrace,
\end{equation}
where $\Omega_{\tilde{d}}\subset\R^2$ is any subset of measure zero containing the set of points where $\tilde{d}$ is not differentiable. Similarly, we denote by $\Omega_d$ any subset of measure zero containing all the points where $d$ is not differentiable. 

Being $\gamma$ Lipschitz, we can assume that $\gamma([a,b])\subset \Omega_{\tilde{d}}\subset \Omega_d$.

Now, we claim that 
\begin{equation}\label{clforimfv}
\nabla d(x)\neq 0\quad\mbox{for all}\quad x\in \R^2\setminus \Omega_d. 
\end{equation}
As a matter of fact, being $\gamma([a,b])$ closed, we observe that for every point $x\in\R^2$ there exists some point $c_x\in \gamma([a,b])$, not necessarily unique, such that 
\begin{equation*}
d(x)=\left|c_x-x\right|.
\end{equation*} 
Then, if $x\in \R^2\setminus \Omega_d$ and we set $v:=c_x-x$ we obtain
\begin{equation*}
\begin{split}
\nabla d(x)\cdot v&=\lim_{t\to 0^+} \frac{d(x+t(c_x-x))-d(x)}{t}\\
&=\lim_{t \to 0^+} \frac{(1-t)\left|c_x-x\right|-\left|c_x-x\right|}{t}\\
&=-\left|c_x-x\right|.
\end{split}
\end{equation*} 
This proves~\eqref{clforimfv}.  Finally, we apply  in~\cite[Proposition~2.5.4]{clarke1990optimization}, and obtain that for every $x\in \R^2 \setminus \Omega_d $ we have that there exists a unique closest point $c_x$ (i.e. the projection is well-defined) and 
\begin{equation*}
\nabla d(x)=\frac{x-c_x}{\left|x-c_x\right|}.
\end{equation*}
In particular, we deduce that 
\begin{equation*}
\nabla \tilde{d}(x)=\begin{dcases}
\frac{-x+c_x}{\left|c-c_x\right|}\quad\mbox{for all}\ x\in B_{\epsilon}^{+}\setminus \Omega_d\\
\frac{x-c_x}{\left|c-c_x\right|}\quad\mbox{for all}\ x\in B_{\epsilon}^{-}\setminus \Omega_d.
\end{dcases}
\end{equation*}
Also, from this and the convexity condition~\eqref{fcdghte} we observe that  for all $x\in B_{\epsilon}\setminus \Omega_d$ it holds
\begin{equation}\label{wser}
\nabla\tilde{d}(x)\cdot (c_x-\gamma(y))\geq 0\quad\mbox{for all}\quad y\in [a,b].
\end{equation}
Also, we observe that if $x_n$ is as in~\eqref{gengrad} then we have that 
\begin{equation*}
\lim_{n}x_n=\lim_{n} c_{x_n}=0=\gamma(x_0).
\end{equation*}
In particular, from this last equation and~\eqref{wser}  we obtain 
\begin{equation*}
\lim_{n}\nabla \tilde{d}(x_n)\cdot (\gamma(x_0)-\gamma(y))\geq 0 \quad\mbox{for all}\quad y\in [a,b].
\end{equation*}
Exploiting this and~\eqref{gengrad} we evince that
\begin{equation}\label{grocj}
\partial \tilde{d}(0)\subset \mbox{co}\left\lbrace \nu_{x_0}\in\S^1|\, \nu_{x_0} \mbox{   satisfies   }~\eqref{fcdghte}\mbox{   in   } x_0 \right\rbrace=:N_x
\end{equation}
Also, we notice that $0\notin \partial N_x$. Indeed, if this was the case we would obtain the existence of two vectors $v,w\in \S^1$ such that $v=-w$ and
\begin{equation*}
\langle v,\gamma(x_0)-\gamma(y)\rangle\geq 0 \mbox{ and } \langle w,\gamma(x_0)-\gamma(y)\rangle\geq 0,
\end{equation*}
giving that $\gamma$ is a line, and providing a contradiction. 

Since $N_x$ is convex, closed, contained in $B_1$ and $0\notin N_x$, then up to a rotation we can assume that 
\begin{equation*}
\mbox{for every }v=(v_1,v_2)^T\in N_x \mbox{ we have that }v_2<0.
\end{equation*} 
In particular, making use of~\eqref{grocj} we obtain
\begin{equation*}
\mbox{for every }v=(v_1,v_2)^T\in \partial \tilde{d}(0) \mbox{ we have that }v_2<0.
\end{equation*}
As a consequence of this, we are entitled now to apply the Implicit Function Theorem for Lipschitz maps in \cite[Section~7.1]{clarke1990optimization}, and obtain the existence of some neighbourhood $V_{0}\subset B_{\epsilon}$ of $0$ and a Lipschitz function $f:V_{0}\to \R$ such that 
\begin{equation*}
\tilde{d}(y_1,f(y_1))=0\quad\mbox{for all}\quad y_1\in V_{0}.  
\end{equation*} 
Now, from~\eqref{kresta} and this last identity we deduce that for all $t\in (a,b)$ such that $\gamma_1(t)\in V_{0}$ it holds that 
\begin{equation*}
\gamma(t)=(\gamma_1(t),f(\gamma_1(t)))^T.
\end{equation*}
Being $\gamma_1$ continuous we can define $U_{x_0}=\gamma_1^{-1}(V_{0})$ and obtain~\eqref{DRK}. In particular, thanks to Lemma~\ref{le:cvdfqw} we obtain that, up to rotating $\gamma$ by $\pi$, $f$ is convex. This concludes the proof of~\eqref{DRK}.

Now, we need to prove~\eqref{DRK-2}. To do so, first we  assume that $\gamma$ is differentiable in $t\in U_{x_0}$. Being $f$ convex we have that 
\begin{equation}\label{1dert}
\lim_{h\searrow 0} \frac{f(\gamma_1(t)+h)-f(\gamma_1(t))}{h}\mbox{ and } \lim_{h\nearrow 0} \frac{f(\gamma_1(t)+h)-f(\gamma_1(t))}{h}
\end{equation}
exist. To prove the differentiability of $f$ in $\gamma(t)$ we only need to prove that the two limits coincide. To this end, thanks to Proposition~\ref{convcurvfunc}, up to replacing $\gamma(\cdot)$ with $\gamma((a+b)-\cdot)$, we can assume that~\eqref{FSAQW} holds. Thus, if we set $\epsilon>0$ such that $B_\epsilon(t)\subset U_{x_0}$, then for every $h\in (0,\epsilon)$ it holds
\begin{equation*}
\gamma_1(t+h)-\gamma_1(t)=\int_{t}^{t+h}\frac{dz}{\sqrt{1+f'(\gamma_1(z))^2}}\geq \frac{1}{\sqrt{1+M^2}}h=:C h,    
\end{equation*}
where $M:=\left\|f'\right\|_{L^\infty(B_\epsilon(t))}$ and so $C\in (0,1)$. Analogously, we can prove that for every $h\in(0,\epsilon)$
\begin{equation*}
\gamma_1(t-h)\leq -C h+\gamma_1(t)
\end{equation*}
Making use of this and the convexity of $f$ we evince 
\begin{equation*}
\begin{split}
\frac{f(\gamma_1(t)+C h)-f(\gamma_1(t))}{C h}&\leq \frac{f(\gamma_1(t+h))-f(\gamma_1(t))}{\gamma_1(t+h)-\gamma_1(t)}\\
&=\frac{f(\gamma_1(t+h))-f(\gamma_1(t))}{h}\frac{h}{\gamma_1(t+h)-\gamma_1(t)}\\
&=\frac{\gamma_2(t+h)-\gamma_2(t)}{h}\frac{h}{\gamma_1(t+h)-\gamma_1(t)}.
\end{split}
\end{equation*}
Similarly, we have that
\begin{equation*}
\begin{split}
\frac{f(\gamma_1(t)-Ch)-f(\gamma_1(t))}{-C h}&\geq \frac{f(\gamma_1(t-h))-f(\gamma_1(t))}{\gamma_1(t-h)-\gamma_1(t)}\\
&= \frac{f(\gamma_1(t-h))-f(\gamma_1(t))}{-h}\frac{-h}{\gamma_1(t-h)-\gamma_1(t)}\\
&=\frac{\gamma_2(t-h)-\gamma_2(t)}{-h}\frac{-h}{\gamma_1(t-h)-\gamma_1(t)}.\\
\end{split}
\end{equation*}
Finally, since $\gamma$ is differentiable in $t$ and by convexity 
\begin{equation*}
\frac{f(\gamma_1(t)-C h)-f(\gamma_1(t))}{-Ch}\leq \frac{f(\gamma_1(t)+Ch)-f(\gamma_1(t))}{C h},
\end{equation*}
we obtain that the limits in~\eqref{1dert} coincide, proving that $f$ is differentiable in $\gamma_1(t)$.

On the other hand, assume that $f$ is differentiable in $\gamma_1(t)$. Then, using convexity we can define for every $z\in V_0$
\begin{equation*}
f_{+}'(z):=\lim_{h\searrow 0}\frac{f(z+h)-f(z)}{h}\mbox{    and   }f_{-}'(z):=\lim_{h\nearrow 0}\frac{f(z)-f(z-h)}{h}.
\end{equation*}
In particular, by convexity we know that $f'(z)=f_{+}'(z)$ and $f'(z)=f_{-}'(z)$ for a.e. $z\in V_{0}$. 

In particular, it holds that
\begin{equation}\label{eq:enzima}
f'(\gamma_1(t))=f_{+}'(\gamma_1(t))=f_{-}'(\gamma_1(t)) 
\end{equation}
and $f_{+}'$ and $f_{-}'$ are both non decreasing. Then, for $l>t$ and every $h>0$ small enough 
\begin{equation*}
0 \leq f_{+}'(\gamma_1(l))-f'(\gamma_1(t))\leq \frac{f(\gamma_1(l)+h)-f(\gamma_1(l))}{h}-f'(\gamma_1(t)).
\end{equation*}
Taking the limit for $l \searrow z$, we obtain for every $h>0$ small enough
\begin{equation*}
0\leq \lim_{l \to t^+} f_{+}'(\gamma_1(l))-f'(\gamma_1(t))\leq \frac{f(\gamma_1(t)+h)-f(\gamma_1(t))}{h}-f'(\gamma_1(t)),
\end{equation*}
and taking the limit for $h\searrow 0$ we obtain that 
\begin{equation}\label{eq:mjpjip}
\lim_{l \searrow t} \left|f_{+}'(\gamma_1(l))-f'(\gamma_1(t))\right|=0.
\end{equation}
Analogously, we can prove that 
\begin{equation}\label{eq:mqpqmp}
\lim_{l \nearrow t} \left|f'(\gamma_1(t))-f_{-}'(\gamma_1(l))\right|=0.
\end{equation}
Thus, if we define the function
\begin{equation*}
\hat{f}'(\gamma_1(z)):=\begin{dcases}
f_{+}'(\gamma_1(z))\quad\mbox{for}\quad z>t\\
f_{-}'(\gamma_1(z))\quad\mbox{for}\quad z\leq t,
\end{dcases}
\end{equation*}
as a consequence of~\eqref{eq:mjpjip} and~\eqref{eq:mqpqmp} we have that $\hat{f}'(\gamma_1)$ is continuous in $t$. Also, recall that by definition $\hat{f}'(\gamma_1(l))=f'(\gamma_1(l))$ for a.e. $l\in U_{x_0}$.

Therefore, using this and~\eqref{FSAQW} we have if $h\in (-\delta,\delta)$ for $\delta>0$ small enough
\begin{equation*}
\begin{split}
\frac{\gamma_1(t+h)-\gamma_1(t)}{h}&=\frac{1}{h}\int_{t}^{t+h}\frac{1}{\sqrt{1+f'(\gamma_1(z))^2}}\,dz\\
&=\frac{1}{h}\int_{t}^{t+h}\frac{1}{\sqrt{1+\hat{f}'(\gamma_1(z))^2}}\,dz.
\end{split}
\end{equation*}
Since $\hat{f}'(\gamma_1(t))$ is continuous in $t$, we have that $t$ is a Lebesgue point and thanks to the Lebesgue differentiation Theorem and equation~\eqref{eq:enzima} 
\begin{equation*}
\lim_{h\to 0}\frac{\gamma_1(t+h)-\gamma_1(t)}{h}=\frac{1}{\sqrt{1+\hat{f}'(\gamma_1(t))^2}}=\frac{1}{\sqrt{1+f'(\gamma_1(t))^2}}
\end{equation*}
and we conclude.

Now, we prove~\eqref{FSAQW2}. To do so, we notice that since $f$ is convex then for all $x \in \gamma_1((a,b))$ where $f$ is differentiable and every $y\in \gamma_1((a,b))$ one has that 
\begin{equation*}
\left\langle \begin{pmatrix}
f'(x) \\-1 
\end{pmatrix}, \begin{pmatrix}
x\\f(x)    
\end{pmatrix}-\begin{pmatrix}
 y\\ f(y)   
\end{pmatrix} \right\rangle \geq 0.    
\end{equation*}
In view of~\eqref{DRK-2}  this readily implies that for all $t\in (a,b)$ where $\gamma$ is differentiable and every $t_1\in (a,b)$ it holds that 
\begin{equation*}
\frac{1}{\sqrt{1+f'(\gamma_1(t))^2}}\left\langle \begin{pmatrix}
f'(\gamma_1(t)) \\-1 
\end{pmatrix}, \begin{pmatrix}
\gamma_1(t)\\f(\gamma_1(t))    
\end{pmatrix}-\begin{pmatrix}
 \gamma_1(t_1)\\ f(\gamma_1(t_1))   
\end{pmatrix} \right\rangle. 
\end{equation*}
We also notice that thanks to Lemma~\ref{la:vxinconvexity} for all $t\in (a,b)$ where $\gamma$ is differentiable there exists a unique vector $\nu_t\in\S^1$ such that for every $t_1\in (a,b)$ it holds that
\begin{equation*}
\langle\nu_t,\gamma(t)-\gamma(t_1)\rangle\geq 0.    
\end{equation*}
Therefore, we have proved that for all $t\in (a,b)$ where $\gamma$ is differentiable  
\begin{equation*}
\nu_t=\frac{1}{\sqrt{1+f'(\gamma_1(t))^2}} \begin{pmatrix}
f'(\gamma_1(t)) \\-1 
\end{pmatrix}   
\end{equation*}
concluding the proof of~\eqref{FSAQW}.
\end{proof}

In view of Proposition~\ref{convcurvfunc} and \Cref{ImFuTh}  we obtain
\begin{corollary}\label{uniqnorma}
Let $\gamma\in Lip(\S^1,\R^2)$ be a homeomorphism, $\left|\gamma'\right|\equiv 1$ and $\gamma$ satisfies~\eqref{fcdghte} in some $(a,b)\subset\S^1$. Then, if $\gamma$ satisfies~\eqref{conclock}, for every $x\in (a,b)$ where $\gamma$ is differentiable it holds
\begin{equation}\label{depriv}
\nu_x=
\gamma'(x)^\perp.
\end{equation}
\end{corollary}

\section{Integration formula for the mean curvature -- Proof of Proposition~\ref{salame}}\label{catamarano}

\begin{proof}[Proof of Proposition~\ref{salame}]
Thanks to the Implicit Function Theorem~\ref{ImFuTh} for every $z\in \S^1$ there exists some neighbourhood $U_z\subset \S^1$ and a convex function $f\in Lip(\R)$ such that (up to rotations) it holds that 
\begin{equation*}
\gamma(t)=(\gamma_1(t),f(\gamma_1(t)))\quad\mbox{for all}\quad t\in U_z.
\end{equation*}
Also, we notice that for every $x\in \gamma(U_z)$ there exists some $\epsilon_x>0$ such that 
\begin{equation}\label{cod-pio-87}
B_{\epsilon_x}(x)\cap E= \left\lbrace (y_1,y_2)\in B_{\epsilon_x}(x)|\,f(y_1)\leq y_2  \right\rbrace.
\end{equation}

Being $f$ convex, thanks to Alexandrov Theorem, we obtain that $f$ is twice differentiable a.e. in $\gamma_1(U_z)$. Let us stress what this actually means: for a.e. $p \in \gamma_1(U_z)$ there exists $a, b \in \R$ such that 
\begin{equation}\label{eq:alexandrov:aslkdjasi}
 \lim_{q \to p} \frac{f(q)-f(p) - a(q-p)-\frac{b}{2} (q-p)^2 }{|q-p|^2}=0.
\end{equation}

Thus, we denote by 
\begin{equation*}
\mathcal{A}_f:=\left\lbrace p\in \gamma_1(U_z)|\,\mbox{$f$ is twice differentiable in $p$} \right\rbrace.  
\end{equation*}
Let $\gamma_1(x)\in \mathcal{A}_f$ be fixed. Also, up to translations and rotations we can assume w.l.o.g. that $\gamma(x)=0$ and $\gamma'(x)=(1,0)$.

Now, we compute the nonlocal mean curvature in $\gamma(x)=0$. 
\begin{equation*}
\begin{split}
H_{\partial E}^s(0)&= P.V.\int_{\R^2} \frac{\chi_{E^c}(y)-\chi_{E}(y)}{\left|y\right|^{2+s}}\,dy\\
&=\lim_{\epsilon\to 0^+}\int_{\R^2\setminus B_{\epsilon}} \frac{\chi_{E^c}(y)-\chi_{E}(y)}{\left|y\right|^{2+s}}\,dy\\
&=\lim_{\epsilon\to 0^+}\left(\int_{E^c\setminus B_{\epsilon}}\frac{dy}{\left|y\right|^{2+s}}- \int_{E\setminus B_{\epsilon}}\frac{dy}{\left|y\right|^{2+s}} \right).
\end{split}
\end{equation*}
Then, we observe that for every $y\in\R^2\setminus \left\lbrace 0\right\rbrace$
\begin{equation*}
\mbox{div}_y \frac{y}{\left|y\right|^{2+s}}=-\frac{s}{\left|y\right|^{2+s}}. 
\end{equation*}

Since by construction $E$ is a Lipschitz set, we have that for $\sigma^1$ a.e. $y\in \partial E$ the outward unit normal $\nu_E(y)$ at $E$ in $y$ is well defined. Note that $\nu_{E}(y)=-\nu_{E^c}(y)$. Thus, applying the Divergence Theorem for Lipschitz sets we obtain that 
\begin{equation*}
\begin{split}
H_{\partial E}^s(0)&=\frac{1}{s}\lim_{\epsilon\to 0^+}\bigg(-\int_{E^c\setminus B_{\epsilon}}\mbox{div}_y\frac{y}{\left|y\right|^{2+s}}\,dy +\int_{E\setminus B_{\epsilon}}\mbox{div}_y\frac{y}{\left|y\right|^{2+s}}\,dy \bigg)\\
&=\frac{1}{s}\lim_{\epsilon\to 0^+}\bigg(-\int_{\partial\left(E^c\setminus B_{\epsilon}\right)}\nu_{E^c\setminus B_{\epsilon}}(y)\cdot\frac{y}{\left|y\right|^{2+s}}\,d\sigma^1+\int_{\partial\left(E\setminus B_{\epsilon}\right)}\nu_{E\setminus B_\epsilon}(y)\cdot\frac{y}{\left|y\right|^{2+s}}\,d\sigma^1 \bigg)\\
&=\frac{1}{s}\lim_{\epsilon\to 0^+}\bigg(-\int_{\partial E^c\setminus B_{\epsilon}}\nu_{E^c}(y)\cdot\frac{y}{\left|y\right|^{2+s}}\,d\sigma^1-\int_{ E^c\cap \partial B_{\epsilon}}\nu_{B_{\epsilon}}(y)\cdot\frac{y}{\left|y\right|^{2+s}}\,d\sigma^1\\
&+\int_{\partial E\setminus B_{\epsilon}}\nu_{E}(y)\cdot\frac{y}{\left|y\right|^{2+s}}\,d\sigma^1+\int_{E\cap \partial B_{\epsilon}}\nu_{B_\epsilon}(y)\cdot\frac{y}{\left|y\right|^{2+s}}\,d\sigma^1 \bigg)\\
&=\frac{1}{s}\lim_{\epsilon\to 0^+}\bigg(2\int_{\partial E\setminus B_\epsilon} \nu_{E}(y)\cdot \frac{y}{\left|y\right|^{2+s}}\,d\sigma^1             + \frac{1}{\epsilon^{1+s}}\left(-\sigma^1\left(E\cap \partial B_{\epsilon}\right)+\sigma^1(E^c\cap \partial B_{\epsilon})\right)\bigg)\\
&=\frac{1}{s}\lim_{\epsilon\to 0^+}\bigg(2\int_{\partial E\setminus B_\epsilon} \nu_{E}(y)\cdot \frac{y}{\left|y\right|^{2+s}}\,d\sigma^1 + \frac{1}{\epsilon^{1+s}}\left(-2\pi\epsilon+2\sigma^1(E^c\cap \partial B_{\epsilon})\right)\bigg).
\end{split}
\end{equation*}
Now, assuming that $\epsilon$ is small enough, by convexity of $E$ we have that 
\begin{equation}\label{cd101-oki}
\sigma^1(E^c\cap \partial B_\epsilon)\geq \sigma^1(\left\lbrace P_2<0 \right\rbrace \cap \partial B_\epsilon)=\pi\epsilon.
\end{equation}
Also, thanks to~\eqref{cod-pio-87}, for some $c>0$ uniform with respect to $\epsilon$, we evince that
\begin{equation}\label{ssensolikjh}
\begin{split}
\sigma^1(E^c\cap \partial B_\epsilon)&=\sigma^1(\left\lbrace  P_2<f(P_1)\right\rbrace\cap \partial B_\epsilon)\\
&=\pi\epsilon+\sigma^1(\left\lbrace 0<P_2<f(P_1)\right\rbrace\cap \partial B_\epsilon)\\
&\leq \pi\epsilon+ c\left(f(\epsilon)+f(-\epsilon)\right).
\end{split}
\end{equation}
Thanks to the twice differentiability of $f$ in the origin, \eqref{eq:alexandrov:aslkdjasi}, and the assumption that $f(0)=f'(0)=0$, we obtain that 
\begin{equation*}
\begin{split}
f(\epsilon)=f''(0)\epsilon^2+o(\epsilon^2).
\end{split}
\end{equation*}
Similarly, one also computes 
\begin{equation*}
f'(-\epsilon)= f''(0)\epsilon^2+o(\epsilon^2).
\end{equation*}
Therefore, using these last two equations,~\eqref{cd101-oki} and~\eqref{ssensolikjh} we infer that
\begin{equation*}
\pi \epsilon\leq \sigma^1(E^c\cap \partial B_{\epsilon})\leq \pi \epsilon+ O(\epsilon^2)
\end{equation*}
These estimates allow us to conclude that 
\begin{equation*}
0 \leq \left(-2\pi\epsilon+2\sigma^1(E^c\cap \partial B_{\epsilon})\right)\leq O(\epsilon^2),
\end{equation*}
from which we evince that 
\begin{equation*}
\lim_{\epsilon\to 0^+}\frac{1}{\epsilon^{1+s}}\left(-2\pi\epsilon+2\sigma^1(E^c\cap \partial B_{\epsilon})\right)=0.
\end{equation*}
In particular, this means that 
\begin{equation*}
\begin{split}
H_{\partial E}^s(0)&=\frac{2}{s}\lim_{\epsilon\to 0^+} \int_{\partial E \setminus B_\epsilon}\nu_{E}(y)\cdot \frac{y}{\left|y\right|^{2+s}}\,d\sigma^1\\
&=\frac{2}{s}\lim_{\epsilon\to 0^+}\int_{\S^1\setminus \gamma^{-1}(B_\epsilon\cap \gamma(\S^1))}n(\gamma(y))\cdot\frac{\gamma(y)}{\left|\gamma(y)\right|^{2+s}}\,dy\\
&=\frac{2}{s}\lim_{\epsilon\to 0^+}\int_{\S^1 \setminus B_{\epsilon}(x)}n(\gamma(y))\cdot\frac{\gamma(y)}{\left|\gamma(y)\right|^{2+s}}\,dy\\
&=: \frac{2}{s}\int_{\S^1}n(\gamma(y))\cdot\frac{\gamma(y)}{\left|\gamma(y)\right|^{2+s}}\,dy
\end{split}
\end{equation*}
where the penultimate identity is due to the fact that $f'\in L^\infty(B_{\epsilon})$.
\end{proof}

\begin{remark}\label{salamegen}
It is important to observe that~\eqref{-fgrop0-0} holds true also under different hypothesis on $E$ and $\gamma$. 

For instance, let $E\subset\R^2$ be convex and with nonempty interior. Also, let $\gamma:\R\to \R^2$ be  BiLipschitz and $\left|\gamma'\right|=1$ a.e. in $\R$, satisfying  the convexity assumption in~\eqref{fcdghte}  in $\R$. Then, thanks to Theorem~\ref{ImFuTh}, up to a  rotation of $\gamma$, for every $x_0\in\R$ there exists some neighbourhood $U_{x_0}$ and a convex function $f:\gamma_1(U_{x_0})\to \R$ such that 
\begin{equation*}
\gamma(t)=\left(\gamma_1(t),f(\gamma_1(t))\right)\quad\mbox{for all}\quad t \in U_{x_0}.   
\end{equation*}
Then, following almost verbatim the proof of Proposition~\ref{salame}, we obtain that for all $x\in U_{x_0}$ where $f$ admits a second order Taylor expansion it holds that 
\begin{equation*}
H_{\partial E}^s((x,f(x)))=\int_{\R}n(\gamma(t))\cdot \frac{\gamma(t)-(x,f(x))}{\left|\gamma(t)-(x,f(x))\right|^{2+s}}\,dt,   
\end{equation*}
where (up to replacing $\gamma(\cdot)$ with $\gamma(-\cdot)$) $n(\gamma(t)):=\gamma'(t)^\perp$ for a.e. $t\in\R$.
\end{remark}

\section{The covering argument: Locally uniform smallness}
In what follows we show locally uniform smallness for sequences of functions with uniformly bounded global integral energy. This is a standard covering argument, see for instance~\cite{SU81} Proposition 4.3 and Theorem 4.4. 
For the sake of completeness we provide here a proof of the result.

\begin{proposition}\label{pr:uniformsmallness}
For any $\eps > 0$ and $\Lambda > 0$ there exists $L = L(\eps,\Lambda)$ such that the following holds.

For any sequence $F_k : \S^1 \times \S^1$
\[
 \sup_{k} \int_{\S^1} \brac{\int_{\S^1} |F_k(x,y)|dy}^p dx \leq \Lambda,
\]
there exists a subsequence (still denoted by $F_{k}$) and set $\Sigma \subset \S^1$ consisting of at most $L$ points such that for any $x_0 \in \S^1 \setminus \Sigma$ there exists a radius $\rho =\rho_{x_0} > 0$ and an index $K \in \N$ such that
\[
 \sup_{k \geq K} \int_{\S^1} \brac{\int_{B(x_0,\rho)} |F_{k}(x,y)|dy}^p dx < \eps.
\]
and
\[
 \sup_{k \geq K} \int_{B(x_0,\rho)} \brac{\int_{\S^1} |F_{k}(x,y)|dy}^p dx < \eps.
\]

\end{proposition}
\begin{proof}
We give the details for the convenience of the reader, but only consider the first inequality, an easy adaptation implies that first and second inequality hold simultaneously.

Pick $\delta << \frac{\eps}{2^p\Lambda}$ and let $m\in\N$. Then cover $\S^1$ by finitely many intervals $B(x_i,\delta 2^{-m})$ such that every point $x\in \S^1$ is covered at most two times. Then we have
	\begin{align*}
		& \int_{\R} \brac{\sum_i \int_{B(x_i,\delta 2^{-m})} |F_k(x,y)|dy}^p dx \\
		& \leq 2^p\Lambda = \tfrac{2^p\Lambda}{\eps} \eps.
	\end{align*}
Since
\[
 \sum_{i=1}^\infty (a_i)^p \leq \brac{\sum_{i=1}^\infty a_i}^p
\]
we have
	\begin{align*}
		& \sum_i \int_{\R} \brac{\int_{B(x_i,\delta 2^{-m})} |F_k(x,y)|dy}^p dx \\
		& \leq \int_{\R} \brac{\sum_i \int_{B(x_i,\delta 2^{-m})} |F_k(x,y)|dy}^p dx \\
		& \leq 2^p\Lambda
	\end{align*}

Hence for every $k$ and any $m$ there exist at most $L:=L(\eps,\Lambda):= \lfloor \tfrac{2^p\Lambda}{\eps} \rfloor$ intervals $B(x_i,\delta 2^{-m})$ such that
	\[
		\int_{\R} \brac{\int_{B(x_i,\delta 2^{-m})} |F_k(x,y)|dy}^p dx  \geq \eps.
	\]

Next we observe that for any fixed $m$ we can pass to a subsequence of $(F_k)_k$ (not relabeled) so that
\[
\#\{i:	\sup_k \int_{\R} \brac{\int_{B(x_i,\delta 2^{-m})} |F_k(x,y)|dy}^p dx \geq \eps \} \leq L.
\]
This is true since for any $k$ there are only $L$ many intervals for which
\[
\int_{\R} \brac{\int_{B(x_i,\delta 2^{-m})} |F_k(x,y)|dy}^p dx \geq \eps
\]
We can then pass to a further subsequence that for any $m$ there exists a $K=K(m)$ such that
\[
\#\{i:	\sup_{k \geq K(m)} \int_{\R} \brac{\int_{B(x_i,\delta 2^{-m})} |F_k(x,y)|dy}^p dx \geq \eps \} \leq L.
\]
Set
\[
 \Sigma_m := \S^1 \setminus \{x \in B(x_i,\delta 2^{-m}): \text{for some $i$ with} \int_{\R} \brac{\int_{B(x_i,\delta 2^{-m})} |F_k(x,y)|dy}^p dx \leq \eps\}
\]
We then have
\[
 \Sigma_m \subset \bigcup_{\ell = 1}^L B(x_{i_\ell},\delta 2^{-m})
\]
Nesting the balls $B(x_i,\delta 2^{-m})$ as $m \to \infty$ we can ensure that $\Sigma_{m+1} \subset \Sigma_m$.

Set
\[
 \Sigma := \bigcap_{m} \Sigma_m.
\]
Then by the definition of the $\mathcal{H}^0$-measure, since $\Sigma$ is covered by $\Sigma_m$ which consists of at most $L$ intervals of sidelength $2^{-m} \delta$ (and this for each $m$) we have
\[
 \mathcal{H}^0(\Sigma) \leq L.
\]
Also clearly, if $x \not \in \Sigma$ there must be some $m \in \N$ such that $x \not \in \Sigma_m$, but $x \in B(x_{i;m},\delta 2^{-m})$ for some $x_{i;m}$ such that
\[
\int_{\R} \brac{\int_{B(x_i,\delta 2^{-m})} |F_k(x,y)|dy}^p dx \leq \eps
.
\] But then for $\rho := \dist(x,\partial B(x_{i;m},\delta 2^{-m}))$ we have
\[
\int_{\R} \brac{\int_{B(x,\rho)} |F_k(x,y)|dy}^p dx \leq \eps.
\]
\end{proof}

\bibliographystyle{abbrv}
\bibliography{bib}

\begin{thebibliography}{10}

\bibitem{SimonBlattgraph}
S.~Blatt.
\newblock (private communication).
\newblock 2023.

\bibitem{blatt2023fractional}
S.~Blatt, G.~Giacomin, J.~Scheuer, and A.~Schikorra.
\newblock {A fractional Willmore-type energy functional--subcritical observations}.
\newblock {\em MATRIX Annals (accepted), preprint arXiv:2306.16941}, 2023.

\bibitem{BR15}
S.~Blatt and P.~Reiter.
\newblock Regularity theory for tangent-point energies: the non-degenerate sub-critical case.
\newblock {\em Adv. Calc. Var.}, 8(2):93--116, 2015.

\bibitem{BRS16}
S.~Blatt, P.~Reiter, and A.~Schikorra.
\newblock Harmonic analysis meets critical knots. {C}ritical points of the {M}\"{o}bius energy are smooth.
\newblock {\em Trans. Amer. Math. Soc.}, 368(9):6391--6438, 2016.

\bibitem{BRS19}
S.~Blatt, P.~Reiter, and A.~Schikorra.
\newblock On {O}'{Hara} knot energies. {I}: {Regularity} for critical knots.
\newblock {\em J. Differ. Geom.}, 121(3):385--424, 2022.

\bibitem{BRSV21}
S.~Blatt, P.~Reiter, A.~Schikorra, and N.~Vorderobermeier.
\newblock Scale-invariant tangent-point energies for knots.
\newblock {\em arXiv preprint arXiv:2104.10238}, 2021.

\bibitem{buck-orloff}
G.~Buck and J.~Orloff.
\newblock A simple energy function for knots.
\newblock {\em Topology Appl.}, 61(3):205--214, 1995.

\bibitem{CCC20}
X.~Cabr{\'e}, M.~Cozzi, and G.~Csat{\'o}.
\newblock A fractional {Michael}-{Simon} {Sobolev} inequality on convex hypersurfaces.
\newblock {\em Ann. Inst. Henri Poincar{\'e}, Anal. Non Lin{\'e}aire}, 40(1):185--214, 2023.

\bibitem{CFSW18}
X.~Cabr\'{e}, M.~M. Fall, J.~Sol\`a-Morales, and T.~Weth.
\newblock Curves and surfaces with constant nonlocal mean curvature: meeting {A}lexandrov and {D}elaunay.
\newblock {\em J. Reine Angew. Math.}, 745:253--280, 2018.

\bibitem{caffarelli2009nonlocal}
L.~Caffarelli, J.-M. Roquejoffre, and O.~Savin.
\newblock Nonlocal minimal surfaces.
\newblock {\em Commun. Pure Appl. Math.}, 63(9):1111--1144, 2010.

\bibitem{CFS23}
M.~{Caselli}, E.~{Florit-Simon}, and J.~{Serra}.
\newblock {Yau's conjecture for nonlocal minimal surfaces}.
\newblock {\em arXiv e-prints}, page arXiv:2306.07100, June 2023.

\bibitem{CSSV23}
H.~{Chan}, S.~{Dipierro}, J.~{Serra}, and E.~{Valdinoci}.
\newblock {Nonlocal approximation of minimal surfaces: optimal estimates from stability}.
\newblock {\em arXiv e-prints}, page arXiv:2308.06328, Aug. 2023.

\bibitem{CFMN18}
G.~Ciraolo, A.~Figalli, F.~Maggi, and M.~Novaga.
\newblock Rigidity and sharp stability estimates for hypersurfaces with constant and almost-constant nonlocal mean curvature.
\newblock {\em J. Reine Angew. Math.}, 741:275--294, 2018.

\bibitem{clarke1990optimization}
F.~H. Clarke.
\newblock {\em Optimization and nonsmooth analysis}.
\newblock SIAM, 1990.

\bibitem{Hitchhiker}
E.~Di~Nezza, G.~Palatucci, and E.~Valdinoci.
\newblock Hitchhiker's guide to the fractional {S}obolev spaces.
\newblock {\em Bull. Sci. Math.}, 136(5):521--573, 2012.

\bibitem{DV16}
S.~{Dipierro} and E.~{Valdinoci}.
\newblock {Nonlocal minimal surfaces: interior regularity, quantitative estimates and boundary stickiness}.
\newblock {\em arXiv e-prints}, page arXiv:1607.06872, July 2016.

\bibitem{DV22}
S.~{Dipierro} and E.~{Valdinoci}.
\newblock {Some perspectives on (non)local phase transitions and minimal surfaces}.
\newblock {\em arXiv e-prints}, page arXiv:2207.04783, July 2022.

\bibitem{evans2018measure}
L.~Evans.
\newblock {\em Measure theory and fine properties of functions}.
\newblock Routledge, 2018.

\bibitem{FHW94}
M.~H. Freedman, Z.-X. He, and Z.~Wang.
\newblock M\"{o}bius energy of knots and unknots.
\newblock {\em Ann. of Math. (2)}, 139(1):1--50, 1994.

\bibitem{GG20}
S.~Ghinassi and M.~Goering.
\newblock Menger curvatures and {{\(C^{1,\alpha}\)}} rectifiability of measures.
\newblock {\em Arch. Math.}, 114(4):419--429, 2020.

\bibitem{MR3099262}
M.~Giaquinta and L.~Martinazzi.
\newblock {\em An introduction to the regularity theory for elliptic systems, harmonic maps and minimal graphs}, volume~11 of {\em Appunti. Scuola Normale Superiore di Pisa (Nuova Serie) [Lecture Notes. Scuola Normale Superiore di Pisa (New Series)]}.
\newblock Edizioni della Normale, Pisa, second edition, 2012.

\bibitem{Goering21}
M.~Goering.
\newblock Characterizations of countably {{\(n\)}}-rectifiable {Radon} measures by higher-dimensional {Menger} curvatures.
\newblock {\em Real Anal. Exch.}, 46(1):1--36, 2021.

\bibitem{GM99}
O.~Gonzalez and J.~H. Maddocks.
\newblock Global curvature, thickness, and the ideal shapes of knots.
\newblock {\em Proc. Natl. Acad. Sci. USA}, 96(9):4769--4773, 1999.

\bibitem{hiriart2004fundamentals}
J.-B. Hiriart-Urruty and C.~Lemar{\'e}chal.
\newblock {\em Fundamentals of convex analysis}.
\newblock Springer Science \& Business Media, 2004.

\bibitem{KS12}
E.~Kuwert and R.~Sch\"{a}tzle.
\newblock The {W}illmore functional.
\newblock In {\em Topics in modern regularity theory}, volume~13 of {\em CRM Series}, pages 1--115. Ed. Norm., Pisa, 2012.

\bibitem{Meurer18}
M.~Meurer.
\newblock Integral {Menger} curvature and rectifiability of {{\(n\)}}-dimensional borel sets in euclidean {{\(N\)}}-space.
\newblock {\em Trans. Am. Math. Soc.}, 370(2):1185--1250, 2018.

\bibitem{MS23}
C.~{Mihaila} and B.~{Seguin}.
\newblock {A definition of fractional k-dimensional measure: bridging the gap between fractional length and fractional area}.
\newblock {\em arXiv e-prints}, page arXiv:2303.11542, Mar. 2023.

\bibitem{OH91}
J.~O'Hara.
\newblock Energy of a knot.
\newblock {\em Topology}, 30(2):241--247, 1991.

\bibitem{OH92}
J.~O'Hara.
\newblock Family of energy functionals of knots.
\newblock {\em Topology Appl.}, 48(2):147--161, 1992.

\bibitem{OH94}
J.~O'Hara.
\newblock Energy functionals of knots. {II}.
\newblock {\em Topology Appl.}, 56(1):45--61, 1994.

\bibitem{PratsSaksman17}
M.~Prats and E.~Saksman.
\newblock A {${\rm T}(1)$} theorem for fractional {S}obolev spaces on domains.
\newblock {\em J. Geom. Anal.}, 27(3):2490--2538, 2017.

\bibitem{R16}
T.~Rivi\'{e}re.
\newblock Weak immersions of surfaces with {$L^2$}-bounded second fundamental form.
\newblock In {\em Geometric analysis}, volume~22 of {\em IAS/Park City Math. Ser.}, pages 303--384. Amer. Math. Soc., Providence, RI, 2016.

\bibitem{SU81}
J.~Sacks and K.~Uhlenbeck.
\newblock The existence of minimal immersions of {$2$}-spheres.
\newblock {\em Ann. of Math. (2)}, 113(1):1--24, 1981.

\bibitem{S86}
L.~Simon.
\newblock Existence of {W}illmore surfaces.
\newblock In {\em Miniconference on geometry and partial differential equations ({C}anberra, 1985)}, volume~10 of {\em Proc. Centre Math. Anal. Austral. Nat. Univ.}, pages 187--216. Austral. Nat. Univ., Canberra, 1986.

\bibitem{Stein61}
E.~M. Stein.
\newblock The characterization of functions arising as potentials.
\newblock {\em Bull. Amer. Math. Soc.}, 67:102--104, 1961.

\bibitem{Menger09}
P.~Strzelecki, M.~Szuma{\'n}ska, and H.~von~der Mosel.
\newblock A geometric curvature double integral of {Menger} type for space curves.
\newblock {\em Ann. Acad. Sci. Fenn., Math.}, 34(1):195--214, 2009.

\bibitem{SvdM12}
P.~Strzelecki and H.~von~der Mosel.
\newblock Tangent-point self-avoidance energies for curves.
\newblock {\em J. Knot Theory Ramifications}, 21(5):1250044, 28, 2012.

\bibitem{LifengWang23}
L.~Wang.
\newblock Inequalities in homogeneous {T}riebel-{L}izorkin and {B}esov-{L}ipschitz spaces.
\newblock {\em Commun. Pure Appl. Anal.}, 22(4):1318--1393, 2023.

\end{thebibliography}

\end{document}